\documentclass[reqno, 11pt]{amsart}
\usepackage[utf8]{inputenc}
\usepackage{amsmath}
\usepackage{amsthm}
\usepackage{amssymb}
\usepackage{mathtools}
\usepackage{booktabs}
\usepackage{amsxtra}
\usepackage{amsbsy}
\usepackage{comment}
\usepackage{mathrsfs}
\usepackage{longtable}
\usepackage{asymptote}

\usepackage[english]{babel}
\usepackage{blindtext}
\usepackage{url}
\usepackage{bbm}
\usepackage{euscript}
\usepackage{pifont}
\usepackage{hyperref}
\usepackage{wasysym}
\hypersetup{
    colorlinks=true,
    linkcolor=blue,
    filecolor=magenta,
    urlcolor=cyan,
}
\usepackage{float}
\usepackage{graphicx}
\usepackage{epstopdf}

    \usepackage[
    top    = 2.4cm,
    bottom = 2.4cm,
    left   = 2.4cm,
    right  = 2.4cm]{geometry}

\newtheorem{thm}{Theorem}[section]
\newtheorem{lem}[thm]{Lemma}

\newtheorem{rem}[thm]{Remark}
\newtheorem*{conjecture*}{Conjecture}
\newtheorem*{thm*}{Theorem}

\theoremstyle{remark}
\newtheorem*{remark}{Remark}

\theoremstyle{definition}

\newcommand{\Z}{\mathbb{Z}}

\newcommand{\F}{\mathbb{F}}
\newcommand{\M}{\mathbb{M}}

\newcommand{\RNum}[1]{\uppercase\expandafter{\romannumeral #1\relax}}

\renewcommand{\F}{\textup{\textsf{F}}}
\newcommand{\Sy}{\textup{\textsf{S}}}
\newcommand{\A}{\textup{\textsf{A}}}
\newcommand{\G}{\textup{\textsf{G}}}
\newcommand{\J}{\textup{\textsf{J}}}
\renewcommand{\M}{\textup{\textsf{M}}}
\newcommand{\McL}{\M^c\textup{\textsf{L}}}
\renewcommand{\S}{\textup{\textsf{S}}}
\newcommand{\U}{\textup{\textsf{U}}}

\newcommand{\cod}{\normalfont{\mbox{cod}}}
\newcommand{\ON}{\normalfont{\mbox{O'N}}}
\newcommand{\cd}{\normalfont{\mbox{cd}}}
\newcommand{\GL}{\normalfont{\mbox{GL}}}
\newcommand{\HS}{\normalfont{\mbox{HS}}}
\newcommand{\PSL}{\normalfont{\mbox{PSL}}}

\newcommand{\PSU}{\normalfont{\mbox{PSU}}}

\newcommand{\Aut}{\mbox{Aut}}
\newcommand{\irr}{\mbox{Irr}}

\newskip\aline \newskip\halfaline
\title{On the characterization of some non-abelian simple groups using codegree set}

\author[H. Wang]{Hongning Wang} \email{howang1196@gmail.com}

\author[X. Zhang]{Xuning Zhang} \address{College of Science, China Three Gorges University, Yichang, Hubei 443002, China } \email{zhxn2021@hotmail.com}

\author[S. Zhang]{Selina Zhang} \email{amberofthenoblefire@gmail.com}

\author[M. Chen]{Michelle Chen} \email{learner99@gmail.com}


\begin{document}

\maketitle

\begin{abstract}
    Let $G$ be a finite group and $\chi\in \irr(G)$. The codegree of $\chi$ is defined as $\cod(\chi)=\frac{|G:\ker(\chi)|}{\chi(1)}$ and $\cod(G)=\{\cod(\chi) \ |\ \chi\in \irr(G)\}$ is called  the set of codegrees of $G$. In this paper, we show that the set of codegrees of $\Sy_4(4), \U_4(2)$, $\Sy_4(q)\ (q \geq 4)$, $\U_4(3)$, ${}^2\F_4(2)'$, $\J_3$, $\G_2(3)$, $\A_9$, $\J_2$, $\PSL(4,3)$, $\McL$, $\Sy_4(5)$, $\G_2(4)$, $\HS$, $\ON$ and $\M_{24}$ determines the group up to isomorphism.
\end{abstract}

\section{Introduction}
Let $G$ be a finite group and $\irr(G)$ the set of all irreducible characters of $G$. For any $\chi\in \irr(G)$ the codegree of $\chi$, written $\cod(\chi)$, is defined to be $\cod(\chi)=\frac{|G:\ker(\chi)|}{\chi(1)}$. We denote $\cod(G)=\{\cod(\chi)\ |\ \chi\in \irr(G)\}$.
The concept of codegrees was originally considered in \cite{Chillag1} where the codegree of $\chi$ was defined as $\frac{|G|}{\chi(1)}$. The definition was later modified in \cite{Qian} so that $\cod(\chi)$ is  independent of $N$ when $N\leq \ker(\chi)$. Many properties of codegrees have been studied, such as the relationship between the codegrees and the element orders, codegrees of $p$-groups, and groups with few codegrees.

Let $\cd(G)= \{\chi(1)\ |\ \chi \in \irr(G)\}$. In the 1990s, Bertram Huppert studied the character degrees of simple groups have a very high degree of individuality. He conjectured a relationship between the structure of groups and its codegrees.

 \textbf{Huppert's conjecture:} Let $H$ be any finite non-abelian simple group and $G$ a finite group such that $\cd(G)$ = $\cd(H)$. Then $G \cong H \times A$, where $A$ is abelian.

Recently, an alternative version of Huppert's conjecture has been proposed by Qian  (question 20.79 in the Kourovka Notebook).

{\bf Codegree version of Huppert's conjecture:} Let $H$ be any finite non-abelian simple group and $G$ a finite group such that $\cod(G)=\cod(H)$. Then $G\cong H$.

This conjecture has been considered and shown to hold for $\PSL(2,q)$ in \cite{BahriAkh}. In \cite{Ahan}, the conjecture was proven for ${}^2B_2(2^{2f+1})$, where $f \ge 1$, $\PSL(3,4)$, $\mathrm{\A}_7$, and $\J_1$. The conjecture holds in the cases where $H$ is $\M_{11}, \M_{12}, \M_{22}, \M_{23}, \PSL(3,3), \PSL(3,q)$ where $4<q\not\equiv 1\ (\bmod\ 3)$, $\PSU(3,q)$ where $4<q \not\equiv -1\ (\bmod\ 3)$, $\G_2(2)'$, $\PSL(3,q)$ where $4<q\equiv 1\ (\bmod\ 3)$, $\PSU(3,q)$ where $4<q \equiv -1\ (\bmod\ 3)$, $\PSL(4,2)$, or ${}^2\G_2(3^{2f+1}), f\geq 1$ by \cite{gkl,ly,gzy}. In this paper we prove the following result which shows that the conjecture holds in the cases where $H$ is $\Sy_4(4), \U_4(2)$, $\Sy_4(q)\ (q \geq 4)$, $\U_4(3)$, ${}^2\F_4(2)'$, $\J_3$, $\G_2(3)$, $\A_9$, $\J_2$, $\PSL(4,3)$, $\McL$, $\Sy_4(5)$, $\G_2(4)$, $\HS$, $\ON$, or $\M_{24}$.

\begin{thm}\label{main}
Assume  a simple group $H$ is isomorphic to $\Sy_{4}(4), \U_{4}(2), \Sy_{4}(q), \U_{4}(3), {_{}^{2}\F_{2}(2)}', \J_{3}$, $\G_{2}(3)$, $\A_9$, $\J_2$, $\PSL(4,3)$, $\McL$, $\S_4(5)$, $\G_2(4)$, $\HS$, $\ON$, or $\M_{24}$. Let $G$ be a finite group with $\cod(G)=\cod(H)$. Then $G\cong H$.
\end{thm}

\section{Lemmas}

Our notation is standard and follows from Isaacs' book \cite{Isaacs} and Atlas \cite{Conway}. The following lemmas will be useful in our later proof.

\begin{lem} \label{Mal2}
(\cite[Lemma 4.2]{Mor}) Let $S$ be a finite non-abelian simple group. Then there exists $1_S\neq \phi\in \normalfont{\irr}(S)$ that extends to $\normalfont{\Aut}(S)$. \end{lem}

\begin{lem} \label{Liu2.5}
(\cite[Theorem 4.3.34]{JaKer} or \cite[Lemma 5]{Bianchi}) Let $N$ be a minimal normal subgroup of $G$ such that $N= S_1\times \cdots \times S_t$ where $S_i\cong S$ is a non-abelian simple group for each $i=1, \dots, t$. If $\chi\in \normalfont{\irr}(S)$ extends to $\normalfont{\Aut}(S)$, then $\chi\times \cdots \times \chi\in \normalfont{\irr}(N)$ extends to $G$.
\end{lem}

\begin{remark}
Cases for finite groups $G$ with $|\cod(G)|\le 3$ are addressed in \cite{Aliz}. So the non-abelian simple group has at least four character codegrees.
\end{remark}

\begin{lem} \label{list}
 Let $S$ be a non-abelian finite simple group. If $|\normalfont{\cod}(S)| \leq 20$, then one of the following holds:
\begin{enumerate}
    \item[(a)] $|\cod(S)|=4$ and $S=\normalfont{\PSL}(2,2^f)$ for $f\geq 2$, or

    \item[(b)] $|\cod(S)|=5$ and $S=\normalfont{\PSL}(2,p^f)$, $p\neq 2, p^f>5$, or

    \item[(c)] $|\cod(S)|=6$ and $S={}^2B_2(2^{2f+1})$, $f\geq 1$ or $G=\normalfont{\PSL}(3,4)$, or

    \item[(d)] $|\cod(S)|=7$ and $S=\normalfont{\PSL}(3,3), \A_7, \M_{11}$, or $\J_1$ or

    \item[(e)] $|\cod(S)|=8$ and $S=\normalfont{\PSL}(3,q)$ where $4<q\not\equiv 1\ (\bmod\ 3)$ or $S=\PSU(3,q)$ where $4<q \not\equiv -1\ (\bmod\ 3)$ or $S=\G_2(2)'$ or

    \item[(f)] $|\cod(S)|=9$ and $S=\normalfont{\PSL}(3,q)$ where $4<q\equiv 1\ (\bmod\ 3)$ or $S=\PSU(3,q)$ where $4<q \equiv -1\ (\bmod\ 3)$ or

    \item[(g)] $|\cod(S)|=10$ and $S=\M_{22}$ or

    \item[(h)] $|\cod(S)|=11$ and $S=\normalfont{\PSL}(4,2), \M_{12}, \M_{23},$ or ${}^2\G_2(3^{2f+1}), f\geq 1$ or

    \item[(i)] $|\cod(S)|=12$ and $S=\Sy_4(4)$ or

    \item[(j)] $|\cod(S)|=13$ and $S=\U_4(2)$, $\Sy_4(q)$ for $q=2^f>4$ or

    \item[(k)] $|\cod(S)|=14$ and $S=\U_4(3), {{}^2\F_4(2)}'$, $\J_3$ or

    \item[(l)] $|\cod(S)|=15$ and $S=\G_2(3)$ or

    \item[(m)] $|\cod(S)|=16$ and $S=\A_9$, $\J_2$

    \item[(n)] $|\cod(S)|=17$ and $S=\PSL^4(3)$, or $\McL$ or

    \item[(o)] $|\cod(S)| = 18$ and $S = \Sy_4(5)$, $\G_2(4)$ or $\HS$

    \item[(p)] $|\cod(S)| = 19$ and $S = \ON$ or

    \item[(r)] $|\cod(S)| = 20$ and $S = \M_{24}$ or $\G_2(q)$ for prime power $q = p^n \geq 7$ with $p$ prime and $q \equiv 2,3,4 \pmod{6}$.
\end{enumerate}
\end{lem}

\begin{proof} For simple group $S$, each nontrivial irreducible character is faithful. Then $|\cod(S)|=|\cd(S)|$ and the result is true by \cite[Theorem 1.1]{Aziz}.
\end{proof}

The character degree sets for the relevant simple groups can be found via \cite{Conway}. We list the relevant codegree sets in the following Lemma for easy reference.

\begin{lem}
If simple group $S$ is isomorphic to $\PSL(2,k)$, ${}^2B_2(q)$, $\PSL(3,4)$, $\A_7$, $\J_1$, $\M_{11}$, $\PSL(3,3)$, $\G_2(2)'$, $\M_{22}$, $\PSL(4,2)$, $\M_{12}$, $\M_{23}$, $\PSL(3,q)$, $\PSU(3,q)$, ${}^2\G_2(q)$, $\Sy_{4}(4)$, $\Sy_{4}(q)$, $\U_{4}(2)$, $\U_{4}(3)$, ${_{}^{2}\F_{2}(2)}'$, $\J_{3}$, $\G_{2}(3)$, $\A_9$, $\J_2$, $\PSL(4,3)$, $\McL$, $\S_4(5)$, $\G_2(4)$, $\HS$, $\ON$, or $\M_{24}$, then $\cod(S)$ can be found in Table 1.
\end{lem}

\begin{proof}
The character degrees of simple groups appeared in the Lemma can be found in \cite{Ahan,BahriAkh,Conway,Suzuki,Ward} and the corresponding codegree set can be easily obtained.
\end{proof}

\begin{longtable}{ll}
\caption{Codegree set for some simple groups}
\\
 \hline
 Group $G$  & Codegree set $\cod(G)$ \\ [1pt] \hline

$\PSL(2,k)$ ($k=2^f\geq 4$) & $\{1, \ k(k-1), \ k(k+1), \ k^2-1\}$  \\   \hline
$\PSL(2,k)$ ($k>5$)  & $\{ 1, \frac{k(k-1)}{2}, \frac{k(k+1)}{2}, \frac{k^2-1}{2}, \ k(k-\epsilon(k))\}$, $\epsilon(k)=(-1)^{(k-1)/2}$ \\ [1pt] \hline

${}^2B_2(q)$ &  $\{1, \ (q-1)(q^2+1), \ q^2(q-1), \ 2^{3f+2}(q^2+1), \ q^2(q-2r+1),$ \\ [1pt]
$(q=2r^2=2^{2f+1})$    &  $q^2(q+2r+1) \}$ \\ [1pt] \hline

 $\PSL(3,4)$ & $\{ 1, \ 2^4{\cdot}3^2{\cdot}7, \ 2^6{\cdot}3^2, \ 2^6{\cdot}5, \ 2^6{\cdot}7, \ 3^2{\cdot}5{\cdot}7 \}$\\ [1pt] \hline

$\A_7$
& $\{1, \ 2^2{\cdot}3{\cdot}5{\cdot}7, \ 2^2{\cdot}3^2{\cdot}7, \ 2^2{\cdot}3^2{\cdot}5, \ 2^3{\cdot}3{\cdot}7, \ 2^3{\cdot}3{\cdot}5, \ 2^3{\cdot}3^2\}$ \\[1pt] \hline

 $\J_1$ &
$\{1, \ 3{\cdot}5{\cdot}11{\cdot}19, \ 2{\cdot}3{\cdot}5{\cdot}7{\cdot}11, \ 2^3{\cdot}3{\cdot}5{\cdot}19, \ 7{\cdot}11{\cdot}19, $ \ $2^3{\cdot}3{\cdot}5{\cdot}11, \ 2^3{\cdot}3{\cdot}5{\cdot}7 \}$ \\[1pt] \hline

 $\M_{11}$ &
$\{1, \ 2^3{\cdot}3^2{\cdot}11, \ 2^4{\cdot}3^2{\cdot}5, \ 3^2{\cdot}5{\cdot}11, \ 2^2{\cdot}3^2{\cdot}5, \ 2^4{\cdot}11, \ 2^4{\cdot}3^2 \}$ \\ [1pt]\hline

 $\PSL(3,3)$ &
$\{1, \ 2^2{\cdot}3^2{\cdot}13, \ 2^4{\cdot}3^3, \ 3^3{\cdot}13, \ 2^3{\cdot}3^3, \ 2^4{\cdot}13, \ 2^4{\cdot}3^2\}$ \\ [1pt] \hline

 $\G_2(2)'$ &
 $\{1, \ 2^4{\cdot}3^2{\cdot}7, \ 2^5{\cdot}3^3, \ 2^4{\cdot}3^3, \ 2^5{\cdot}3^2, \ 2^5{\cdot}7, \ 2^3{\cdot}3^3, \ 3^3{\cdot}7\}$ \\ [1pt] \hline

 $\M_{22}$ &
$\{1, \ 2^7{\cdot}3{\cdot}5{\cdot}11, \ 2^7{\cdot}7{\cdot}11, \ 2^7{\cdot}3^2{\cdot}7, \ 2^7{\cdot}5{\cdot}7, \ 2^6{\cdot}3^2{\cdot}5, \ 2^6{\cdot}3{\cdot}11, \ 2^7{\cdot}3{\cdot}5, \ 2^4{\cdot}3^2{\cdot}11, \ 2^7{\cdot}3^2\}$ \\ [1pt] \hline

 $\PSL(4,2)$ &
 $\{1, \ 2^6{\cdot}3^2{\cdot}5, \ 2^5{\cdot}3^2{\cdot}5, \ 2^4{\cdot}3^2{\cdot}7, \ 2^6{\cdot}3{\cdot}5, \ 2^4{\cdot}3^2{\cdot}5, \ 2^6{\cdot}3^2, \ 2^6{\cdot}7, \ 2^3{\cdot}3^2{\cdot}5, \ 3^2{\cdot}5{\cdot}7, \ 2^5{\cdot} 3^2\}$ \\  [1pt] \hline

  $\M_{12}$ &
 $\{1, \ 2^6{\cdot}3^3{\cdot}5, \ 2^2{\cdot}3^3{\cdot}5{\cdot}11, \ 2^6{\cdot}3{\cdot}11, \ 2^5{\cdot}5{\cdot}11, \ 2^6{\cdot}3^3, \ 2^5{\cdot}3^2{\cdot}5, \ 2^6{\cdot}3{\cdot}5,$ \\ [1pt]
&  $2^3{\cdot}3^2{\cdot}11, \ 2^2{\cdot}3{\cdot}5{\cdot}11, \ 2^2{\cdot}3^3{\cdot}5\}$ \\ [1pt] \hline

  $\M_{23}$
  & $\{1, \ 2^6{\cdot}3^2{\cdot}5{\cdot}7{\cdot}23, \ 2^7{\cdot}7{\cdot}11{\cdot}23, \ 2^6{\cdot}3^2{\cdot}7{\cdot}11, \ 2^7{\cdot}3{\cdot}5{\cdot}23, \ 2^7{\cdot}3^2{\cdot}5{\cdot}7,$\\ [1pt]
  & $2^6{\cdot}3^2{\cdot}23, \ 3^2{\cdot}5{\cdot}11{\cdot}23, \ 2^6{\cdot}7{\cdot}23, \ 2^7{\cdot}7{\cdot}11, \ 2^4{\cdot}3^2{\cdot}5{\cdot}7\}$ \\ [1pt] \hline

  $\PSL(3,q)$ &
$\{ 1, \ (q^2+q+1)(q^2-1)(q-1), \ q^2(q^2+q+1)(q-1)^2, $ \\ [1pt]
$4<q\not\equiv 1\ (\bmod\ 3)$ & $\ q^3(q^2+q+1), \ q^2(q^2-1)(q-1), \ q^3(q^2-1), \ q^3(q^2-1)(q-1), \ q^3(q-1)^2 \}$
\\ [1pt] \hline

$\PSL(3,q)$ &
$ \{ 1, \ \frac{1}{3}(q^2+q+1)(q+1)(q-1)^2, \ \frac{1}{3}q^2(q^2+q+1)(q-1)^2, \ \frac{1}{3}q^3(q^2+q+1),$  \\ [1pt]
$4<q\equiv 1 \ (\bmod\ 3)$ & $\frac{1}{3}q^2(q+1)(q-1)^2, \ \frac{1}{3}q^3(q-1)(q+1), \ \frac{1}{3}q^3(q+1)(q-1)^2, \ \frac{1}{3}q^3(q-1)^2, \ q^3(q-1)^2 \}$ \\ [1pt] \hline

 $\PSU(3,q)$ &
 $\{ 1, \ (q^2-q+1)(q+1)^2(q-1), \ q^3(q^2-q+1), \ q^2(q^2-q+1)(q+1)^2,$ \\ [1pt]
 $4<q\not\equiv -1\ (\bmod\ 3)$  & $q^3(q+1)^2(q-1), \ q^3(q+1)^2, \  q^2(q+1)^2(q-1), \ q^3(q-1)(q+1) \}$\\ [1pt] \hline

  $\PSU(3,q)$
  &$ \{ 1, \ \frac{1}{3}(q^2-q+1)(q+1)^2(q-1), \ \frac{1}{3}q^3(q^2-q+1), \ \frac{1}{3}q^2(q^2-q+1)(q+1)^2,$  \\ [1pt]
  $4<q\equiv -1 \ (\bmod\ 3)$ & $\frac{1}{3}q^3(q+1)^2(q-1), \ \frac{1}{3}q^3(q+1)^2, \ \frac{1}{3}q^2(q+1)^2(q-1), \ \frac{1}{3}q^3(q-1)(q+1), \ q^3(q+1)^2 \}$ \\ [1pt] \hline

  ${}^2\G_2(q)$, $q=3^{2f+1}$ &
  $\{1, \ q^3(q^2-1), \ (q^2-1)(q^2-q+1), \ q^2(q^2-1), \ q^3(q-1),$  \\ [1pt]
  $(m=3^{f}, f \geq 1)$ &  $2{\cdot}3^{5f+3}(q+1)(q+1-3m), \ 2{\cdot}3^{5f+3}(q+1)(q+1+3m),$  \\ [1pt]
  &  $3^{5f+3}(q^2-q+1), \ q^3(q+1), \ q^3(q+1-3m), \ q^3(q+1+3m)\}$  \\ [1pt] \hline

 $\Sy_4(4)$ &
 $\{1, \ 2^7{\cdot}5^2{\cdot}17, \ 2^7{\cdot}3^2{\cdot}5^2 , \ 2^7{\cdot}3^2{\cdot}17, \ 2^8{\cdot}3{\cdot}5^2, \ 2^8{\cdot}3^2{\cdot}5, \ 2^8{\cdot}5^2,  \ 2^6{\cdot}3{\cdot}5^2, $  \\ [1pt]
  &  $2^8{\cdot}17,  \ 2^8{\cdot}3{\cdot}5, \ 3^2{\cdot}5^2{\cdot}17, \ 2^6{\cdot}3^2 \}$ \\ [1pt] \hline

  $\U_4(2)$  &
 $\{1, \ 2^6{\cdot}3^4, \ 2^5{\cdot}3^3{\cdot}5, \ 2^5{\cdot}3^4, \ 2^6{\cdot}3^3, \ 2^4{\cdot}3^4, \ 2^3{\cdot}3^3{\cdot}5,$ \\ [1pt]
  &  $2^5{\cdot}3^3, \ 2^3{\cdot}3^4, \ 2^6{\cdot}3^2, \ 2^4{\cdot}3^3, \ 3^4{\cdot}5, \ 2^6{\cdot}5  \}$ \\ [1pt] \hline

  $\Sy_4(q)$ (for $q=2^f>4$)
  & $\{1, \ 2q^3(q-1)^2(q+1)^2, \ 2q^3(q^2+1)(q-1)^2, \ (q^2+1)(q-1)^2(q+1)^2,$ \\ [1pt]
  &  $\ 2q^3(q^2+1)(q+1)^2, \ q^4(q-1)(q+1), \ q^4(q+1)^2, \ q^4(q^2+1), \ q^4(q+1)(q-1)^2, $ \\ [1pt]
  &  $\ q^4(q-1)(q+1)^2, \ q^3(q-1)^2(q+1), \ q^3(q-1)(q+1)^2, \ q^4(q-1)^2 \}$ \\ [1pt] \hline

  % $S_4(q)$ for $q = 2^f > 4$  &
 %$\{1,\ q^3(q-1)^2(q+1)^2,\ q^3(q^2+1)(q-1)^2,\ %\frac{1}{2}(q^4-1)(q^2-1),$ \\ [1pt] $q\ is\ odd $ &
 % $\ q^3(q^2+1)(q+1)^2, \ \frac{1}{2}q^4(q^2-1),\  \frac{1}{2}q^4(q+1)^2,\ \frac{1}{2}q^4(q^2+1),\ \frac{1}{2}q^4(q+1)(q-1)^2 ,$ \\ [1pt]
 % &  $\ \frac{1}{2}q^4(q-1)(q+1)^2 , \ \frac{1}{2}q^3(q-1)^2(q+1) , \ \frac{1}{2}q^3(q-1)(q+1)^2 , \ \frac{1}{2}q^4(q-1)^2 \}$ \\ [1pt] \hline

 $\U_4(3)$ &
 $\{1, \ 2^7{\cdot}3^5{\cdot}5, \ 2^7{\cdot}3^6, \ 2^6{\cdot}3^4{\cdot}7, \ 2^5{\cdot}3^6, \ 2^7{\cdot}3^3{\cdot}5, \ 2^6{\cdot}3^5, \ 2^4{\cdot}3^6, \ 2^7{\cdot}3^4, \ 2^5{\cdot}3^5, $  \\ [1pt]
  &  $2^3{\cdot}3^6, \ 3^6{\cdot}7, \ 2^7{\cdot}5{\cdot}7, \ 3^6{\cdot}5 \}$ \\ [1pt] \hline

  ${}^2\F_4(2)'$  &
 $\{1, \ 2^{10}{\cdot}3^3{\cdot}5^2, \ 2^{11}{\cdot}5^2{\cdot}13, \ 2^{10}{\cdot}3^2{\cdot}5^2, \ 2^9{\cdot}3^2{\cdot}13, \ 2^{11}{\cdot}3^3, \ 2^{11}{\cdot}5^2, \ 2^7{\cdot}3^2{\cdot}5, \ 2^{10}{\cdot}3^3, $  \\ [1pt]
  &  $2^{11}{\cdot}13, \ 2^{10}{\cdot}5^2, \ 2^9{\cdot}3^3, \ 2^5{\cdot}5^2{\cdot}13, \ 3^3{\cdot}5^2{\cdot}13 \}$ \\ [1pt] \hline

  $\J_3$ &
 $\{1, \ 2^7{\cdot}3^5{\cdot}19, \ 2^7{\cdot}3^5{\cdot}5, \ 2^5{\cdot}3{\cdot}5{\cdot}17{\cdot}19, \ 2^6{\cdot}3^5{\cdot}5, \ 2^3{\cdot}3^4{\cdot}5{\cdot}19, \ 2^5{\cdot}3^4{\cdot}17, $ \\ [1pt]
   & $2^7{\cdot}17{\cdot}19, \ 2^7{\cdot}3^5, \ 3^4{\cdot}17{\cdot}19, \ 2^6{\cdot}3^4{\cdot}5, \ 3^5{\cdot}5{\cdot}17, \ 2^6{\cdot}3{\cdot}5{\cdot}19, \ 2^6{\cdot}3{\cdot}5{\cdot}17 \}$ \\ [1pt] \hline

 $\G_2(3)$ &
 $\{1, \ 2^5{\cdot}3^6{\cdot}13, \ 3^6{\cdot}7{\cdot}13, \ 2^5{\cdot}3^5{\cdot}7, \ 2^6{\cdot}3^6, \ 2^3{\cdot}3^6{\cdot}7, \ 2^3{\cdot}3^5{\cdot}13, \ 2^5{\cdot} 3^6,\ 2^6{\cdot}3^5, $ \\ [1pt]
   & $3^6{\cdot}13, \ 2^5{\cdot}3^5, \ 2^3{\cdot}3^6, \ 2^6{\cdot}7{\cdot}13, \ 2^6{\cdot}3^4, \ 3^6{\cdot}7 \}$ \\ [1pt] \hline

 $\A_9$ & $\{1,\ 2^3{\cdot}3^4{\cdot}5{\cdot}7, \ 2^6{\cdot}3^3{\cdot}5, \ 2^6{\cdot}3{\cdot}5{\cdot}7, \ 2^4{\cdot}3^4{\cdot}5, \ 2^6{\cdot}3^4, \ 2^5{\cdot}3^3{\cdot}5, \ 2^2{\cdot}3^3{\cdot}5{\cdot}7,$ \\ [1pt]
 & $ 2^3{\cdot}3^4{\cdot}5, \ 2^4{\cdot}3^3{\cdot}5, \ 2^6{\cdot}3^3, \ 2^3{\cdot}3^3{\cdot}7, \ 2^3{\cdot}3^3{\cdot}5, \ 2^5{\cdot}5{\cdot}7, \ 2^6{\cdot}3{\cdot}5, \ 2^3{\cdot}3{\cdot}5{\cdot 7}\}$ \\ [1pt] \hline

 $\J_2$ & $\{1, \ 2^6{\cdot}3^3{\cdot}5^2, \ 2^7{\cdot}3^2{\cdot}5^2, \ 2^5{\cdot}3{\cdot}5^2{\cdot}7, \ 2^7{\cdot}3{\cdot}5^2, \ 2^6{\cdot}3^3{\cdot}5, \ 2^6{\cdot}3{\cdot}5{\cdot}7, \ 2^6{\cdot}3{\cdot}5^2, \ 2^2{\cdot}3^3{\cdot}5{\cdot}7, $ \\ [1pt]
 & $ \ 2^7{\cdot}3^3, \ 2^7{\cdot}5^2, \  2^2{\cdot}3^3{\cdot}5^2, \ 2^7{\cdot}3{\cdot}7, \ 2^2{\cdot}3{\cdot}5^2{\cdot}7, \ 2^5{\cdot}3^2{\cdot}7, \ 2^3{\cdot}3^2{\cdot}5^2\}$ \\ [1pt] \hline

 $\PSL(4,3)$ & $\{1, \ 2^6{\cdot}3^6{\cdot}5, \ 2^7{\cdot}3^5{\cdot}5, \ 2^5{\cdot}3^6{\cdot}5, \ 2^7{\cdot}3^6, \ 2^6{\cdot}3^4{\cdot}13, \ 2^6{\cdot}3^4{\cdot}5, \ 2^5{\cdot}3^6, \ 2^7{\cdot}3^3{\cdot}5, $ \\ [1pt]
 & $ \ 2^6{\cdot}3^5, \ 2^2{\cdot}3^6{\cdot}5, \ 2^5{\cdot}3^4{\cdot}5, \ 2^7{\cdot}3^4, \ 3^6{\cdot}13, \ 2^7{\cdot}5{\cdot}13, \ 2^5{\cdot}3^5, \ 2^3{\cdot}3^6 \}$ \\ [1pt] \hline

 $\M^c L$ & $\{1, \ 2^6{\cdot}3^6{\cdot}5^3{\cdot}7, \ 2^7{\cdot}3^5{\cdot}5^3, \ 2^5{\cdot}3^4{\cdot}5^3{\cdot}11, \ 2^6{\cdot}3^6{\cdot}5^2, \ 3^6{\cdot}5^3{\cdot}11, \ 2^6{\cdot}3^6{\cdot}11, $ \\ [1pt]
 & $
 2{\cdot}3^6{\cdot}5^2{\cdot}7, \ 2^5{\cdot}3^4{\cdot}7{\cdot}11, \
 2^3{\cdot}3^3{\cdot}5^3{\cdot}7, \ 2^7{\cdot}5^3{\cdot}11, $ \\ [1pt]
 & $ 2^4{\cdot}3^4{\cdot}5^3, \ 2^7{\cdot}5^3{\cdot}7, \  2^6{\cdot}3^5{\cdot}7, \
 2^7{\cdot}3^6, \  3^6{\cdot}5^3, \ 2^7{\cdot}3^3{\cdot}5^2 \}$ \\[1pt] \hline

 $\Sy_4(5)$ & $\{1, \ 2^6{\cdot}3^2{\cdot}5^4, \ 2^3{\cdot}3^2{\cdot}5^3{\cdot}13, \ 2^6{\cdot}3^2{\cdot}5^3, \ 2^5{\cdot}3{\cdot}5^4, \ 2^5{\cdot}5^3{\cdot}13, \ 2^3{\cdot}3^2{\cdot}5^4, \ 2^5{\cdot}3^2{\cdot}5^3, \   2^4{\cdot}3{\cdot}5^4,$ \\ [1pt]  &$ 2^2{\cdot}3^2{\cdot}5^4, \ 2^3{\cdot}3{\cdot}5^4, \ 2^6{\cdot}3^3{\cdot}5^2, \ 2^5{\cdot}3{\cdot}5^3, \ 2^3{\cdot}3^2{\cdot}5^3, \ 5^4{\cdot}13, \ 2^2{\cdot}3{\cdot}5^4, \ 2^6{\cdot}3^2{\cdot}13, \ 2^4{\cdot}3{\cdot}5^3\}$ \\ [1pt] \hline

 $\G_2(4)$ & $\{1, \ 2^{12}{\cdot}3^3{\cdot}5{\cdot}7, \ 2^{11}{\cdot}3^2{\cdot}5^2{\cdot}7, \ 2^{10}{\cdot}3^2{\cdot}7{\cdot}13, \ 2^{11}{\cdot}3^3{\cdot}13, \ 2^{10}{\cdot}3^3{\cdot}5^2,$\\[1pt]  &$ 2^{11}{\cdot}5^2{\cdot}13, \  2^{11}{\cdot}3^3{\cdot}7, \ 2^{12}{\cdot}3{\cdot}5^2, \  2^{10}{\cdot}3^3{\cdot}7, \  2^{12}{\cdot}3^2{\cdot}5, \ 2^{12}{\cdot}3{\cdot}7,$\\[1pt]  &$ 2^{10}{\cdot}3{\cdot}5^2, \ 2^{12}{\cdot}3{\cdot}5, \ 3^3{\cdot}5^2{\cdot}7{\cdot}13, \ 2^6{\cdot}3^3{\cdot}5{\cdot}7, \ 2^{12}{\cdot}13, \ 2^{10}{\cdot}3^2{\cdot}5\}$\\ [1pt] \hline

 $\HS$ & $1, \ 2^8{\cdot}3^2{\cdot}5^3{\cdot}7, \ 2^9{\cdot}3^2{\cdot}5^3, \ 2^8{\cdot}3^2{\cdot}5^3, \ 2^9{\cdot}3^2{\cdot}5{\cdot}11, \ 2^9{\cdot}3{\cdot}5^3, \ 2^9{\cdot}5^3, \ 2^8{\cdot}3^2{\cdot}5^2,$ \\ [1pt] &$ 2^9{\cdot}3{\cdot}5{\cdot}7, \ 2^2{\cdot}3^2{\cdot}5^3{\cdot}11, \ 2^4{\cdot}3{\cdot}5^3{\cdot}7, \ 2^8{\cdot}5^3, \ 2^2{\cdot}3^2{\cdot}5^3{\cdot}7,$ \\ [1pt] &$ 2^8{\cdot}3^2{\cdot}11, \ 2^9{\cdot}3^2{\cdot}5, \ 2^6{\cdot}5^2{\cdot}11, \ 2^8{\cdot}3^2{\cdot}7, \ 2^2{\cdot}3^2{\cdot}5{\cdot}7{\cdot}11\}$ \\ \hline

 $\ON$ & $\{1, \ 2^3 {\cdot}3^2{\cdot}5{\cdot}7^3{\cdot}11{\cdot}31, \ 2^3{\cdot}3^4{\cdot}5{\cdot}7^3{\cdot}31, \ 2^7{\cdot}3^4{\cdot}5{\cdot}7^3, \ 2^2{\cdot}3^4{\cdot}5{\cdot}7^3{\cdot}31, \ 2^9{\cdot}3^4{\cdot}7^3, \ 2^3{\cdot}3^4{\cdot}5{\cdot}7^3{\cdot}1,$ \\ [1pt] & $ 2^7{\cdot}3^2{\cdot}5{\cdot}7^2{\cdot}3, \ 2^9{\cdot}3^2{\cdot}5{\cdot}7^3, \ 2^9{\cdot}3^2{\cdot}5{\cdot}11{\cdot}31, \ 2^8{\cdot}3^4{\cdot}7^3, \ 2^6{\cdot}3^4{\cdot}5{\cdot}11{\cdot}19, $ \\ [1pt] & $ 2^9{\cdot}3^4{\cdot}5{\cdot}19, \ 2^8{\cdot}3^4{\cdot}5{\cdot}31,
\ 2^8{\cdot}7^3{\cdot}31, \ 3^4{\cdot}5{\cdot}11{\cdot}19{\cdot}31, \ 2^8{\cdot}7^2{\cdot}11{\cdot}19,
\ 7^3{\cdot}11{\cdot}19{\cdot}31, \ 2^4{\cdot}3^4{\cdot}7^2{\cdot}31\}$ \\ \hline

 $\M_{24}$ & $\{1, \ 2^{10}{\cdot}23, \ 2^8{\cdot}3{\cdot}5{\cdot}11, \ 2^7{\cdot}3{\cdot}5{\cdot}23, \ 2^{10}{\cdot}3^2{\cdot}5, \ 2^4{\cdot}3^3{\cdot}7{\cdot}23, \ 2^6{\cdot}3{\cdot}5{\cdot}7{\cdot}11, \ 2^{10}{\cdot}3{\cdot}5{\cdot}7,$ \\ [1pt] & $2^7{\cdot}3^3{\cdot}5{\cdot}7, \ 2^{10}{\cdot}3^3{\cdot}5, \ 2^{10}{\cdot}3^3{\cdot}7, \ 2^{10}{\cdot}3{\cdot}7{\cdot}11, \ 2^9{\cdot}3{\cdot}7{\cdot}23, \ 2^9{\cdot}3^3{\cdot}23, \ 2^{10}{\cdot}3^2{\cdot}5{\cdot}11,$ \\ [1pt] & $2^{10}{\cdot}3^3{\cdot}5{\cdot}7, \ 2^8{\cdot}3{\cdot}5{\cdot}11{\cdot}23, \ 2^{10}{\cdot}3^2{\cdot}5{\cdot}23, \ 2^{10}{\cdot}3{\cdot}7{\cdot}11{\cdot}23, \ 2^{10}{\cdot}3^3{\cdot}5{\cdot}7{\cdot}11\}$ \\ \hline
\end{longtable}

At the end of this section. We introduce the result about the maximal subgroups of $\Sy_{4}(q)$, $q>4$ and $q$ is even in Table 2. This is obtained from \cite[Table 8.14 ]{JohnDer}. And we can easily calculate the order of these subgroups.

%\begin{table}[H]
%\caption{Maximal subgroups of $\mathrm {PSL}_5(3)$}

%\begin{tabular}[Table 2]{|c|c|c|} \hline
%    \mbox{ Subgroup}  &\mbox{ Order}  &\mbox{ Index}  \\ \hline
%   $3^4:\mathrm{GL}_4(3)$ & $1965150720$ &  $121$ \\ \hline
%   $3^6:(\mathrm{PSL}_3(3)\times\mathrm{GL}_2(3))$ & $196515072$ &  $1210$ \\ \hline
%   $3^6:(\mathrm{GL}_2(3)\times\mathrm{PSL}_3(3))$ & $196515072$ &  $1210$ \\ \hline
%   $3^4:\mathrm{GL}_4(3)$ & $1965150720$ &  $121$ \\ \hline
%   $121:5$ & $605$ &  $393030144$      \\  \hline
%   $O(5,3):2$ & $51840$ &  $4586868$      \\  \hline
%   $M_{11}$   & $7920$  &   $30023136$        \\ \hline
%   $M_{11}$   & $7920$  &   $30023136$        \\ \hline
%\end{tabular}
%\end{table}

\begin{lem}\label{maxS4q}
Let $K$ be a maximal subgroup of $\mathrm \Sy_4(q)$ where $q=2^f>4$. Then $K$ is isomorphic to one of the groups in Table 2.
\end{lem}

\begin{table}[H]
\caption{Maximal subgroups of $\mathrm \Sy_4(q)$, $q>4$, $q\ is\ even$}

\begin{tabular}[Table 2]{|c|c|c|} \hline
    \mbox{ Subgroup}  &\mbox{ Order}  &\mbox{ Notes}  \\ \hline
   $E_{q}^{3}:GL_{2}(q)$ & $q^3(q^2-1)(q^2-q)$ &  point \quad stabiliser \\ \hline
   $[q^4]:{C_{q-1}}^2$ & $q^4(q-1)^2$ &      \\  \hline
   $Sp_{2}(q)\wr2$   & $2q^2(q-1)^2(q+1)^2$  &           \\ \hline
   $Sp_{2}(q^2):2$    & $2q^2(q^4-1)$         &     \\  \hline
   ${C_{q-1}}^{2}:D_{8}$   &  $8(q-1)^2$ & $q\ne 4$   \\   \hline
   ${C_{q+1}}^{2}:D_{8}$   &  $8(q+1)^2$ &      \\  \hline
   $C_{q^2+1}:4$   & $4(q^2+1)$   &      \\   \hline
   $Sp_{4}(q_{0})$ &  $q_{0}^{4}(q_{0}^{4}-1)(q_{0}^{2}-1)$     & $q=q_{0}^{r}$, $r$\ is\ a\ prime     \\ \hline
   $SO_{4}^{+}(q)$ & $\frac{1}{2}q^2(q^2-1)$ &   \\ \hline
   $SO_{4}^{-}(q)$ & $q^2(q^2+1)(q^2-1)$     &  \\  \hline
   $Sz(q)$   &$q^2(q^2+1)(q-1)$    & $f\ge 3$ \ $f$\ is\ odd    \\ \hline
\end{tabular}
\end{table}

Below we list the result about the maximal subgroups of $\mathrm{\PSL}(4,3)$ in Table 3. This is obtained from \cite{Conway}.

\begin{lem}\label{psl43}
Let $K$ be a maximal subgroup of $\mathrm{\PSL}(4,3)$. Then $K$ is isomorphic to one of the groups in Table 3.
\end{lem}

\begin{table}[H]\label{table 3}
\caption{Maximal subgroups of $\mathrm {\PSL}(4,3)$}

\begin{tabular}[Table 2]{|c|c|c|} \hline
    \mbox{ Subgroup structure}  &\mbox{ Order}  &\mbox{ Index}  \\ \hline
   $3^3:\mathrm{L}_3(3)$ & $151632$ &  $40$ \\ \hline
   $3^3:\mathrm{L}_3(3)$ & $151632$ &  $40$ \\ \hline
   $\mathrm{U}_4(2)$ & $51840$ &  $117$      \\  \hline
   $\mathrm{U}_4(2)$ & $51840$ &  $117$      \\  \hline
   $3^4:2(\mathrm{A}_4\times \mathrm{A}_4).2$   & $46656$  &   $130$        \\ \hline
   $(4\times\mathrm{A}_6):2$ & $2880$ & $2106$ \\  \hline
   $\mathrm{S}_6$   &  $720$ & $8424$   \\   \hline
   $\mathrm{S}_4\times\mathrm{S}_4$ &$576$ &  $10530$    \\  \hline
\end{tabular}
\end{table}

\begin{lem} \label{perfect}
Let $G$ be a finite group. If no element of $\cod(G)$ is a power of a prime, then $G$ is a perfect group.
\end{lem}
\begin{proof}
Let $G$ be a finite group and suppose $G'<G$. Let $N$ be a maximal normal subgroup of $G$ containing $G'$. Then, $G/N$ is an abelian simple group of order $p$ for some prime $p$. Let $\chi\in \irr(G/N)$ be a non-principal character. Since $G/N$ is abelian, it follows that $\cod(\chi)=\frac{|G/N:\ker(\chi)|}{\chi(1)}=\frac{|G/N|}{1}=p$. Since we can view $\cod(G/N)$ as a subset of $\cod(G)$, then there exists an element of $\cod(G)$ that is a power of a prime. \end{proof}

\begin{rem}\label{perfectgroup}
\normalfont{If $\cod(G)=\cod(H)$ where $G$ is a finite group and $H$ is a finite, non-abelian simple group, it follows that $G$ is a perfect group. Otherwise $\cod(G/G')\subseteq \cod(G)$ contains a power of a prime, which is a contradiction by \cite[Lemma 3.6]{Aliz}.}
\end{rem}

For notation, let $N$ be a normal subgroup of $G$ and $\irr(G|N):= \irr(G)-\irr(G/N)$. For any $\lambda \in \irr(N)$, let $I_G(\lambda)$ be the inertia group of $\lambda$ in $G$. The set of irreducible constituents of $\lambda^{I_G(\lambda)}$ is denoted $\irr(I_G(\lambda)|\lambda)$.

\section{Main Results}

\begin{lem}\label{redS44}
Let $G$ be a finite group with $\cod(G)=\cod(\Sy_{4}(4))$. If $N$ is a maximal normal subgroup of $G$, then $G/N\cong \Sy_{4}(4)$.
\end{lem}

\begin{proof}
Let $N$ be a maximal normal subgroup of $G$. Since $\cod(G)=\cod(\Sy_{4}(4))$, we see that $G$ is perfect by our earlier remark. Then $G/N$ is a non-abelian simple group. Since $\cod(G/N)\subseteq \cod(G)$, we see that $|\cod(G/N)|$ is either $4, 5, 6, 7, 8, 9, 10, 11$, or $12$ by Lemma \ref{list} and we analyze the cases separately.

\item[(a)] $|\cod(G/N)|=4$. By Lemma \ref{list}, $G/N \cong \PSL(2,k)$ where $k = 2^f \geq 4,$Then $\cod(G/N)=\{1, k(k-1), k(k+1), k^2-1\}$.
Since $3^2\cdot5^2\cdot17$ is the only nontrivial odd codegree in $\cod(\Sy_{4}(4))$, we have that $k^2-1=3^2\cdot5^2\cdot17$.This is a contradiction since $k\in \mathbb{Z}$.

\item[(b)] $|\cod(G/N)|=5$. By Lemma \ref{list}, $G/N\cong \PSL(2,k)$ where $k$ is an odd prime power and $\cod(G/N)=\left\{ 1, \frac{k(k-1)}{2}, \frac{k(k+1)}{2}, \frac{k^2-1}{2}, k(k-\epsilon(k))\right\}$ where $\epsilon(k)=(-1)^{(k-1)/2}$.
Then $k(k-\epsilon(k))/2$, $k(k-\epsilon(k))\in \cod(G)$, a contradiction for we can't find a codegree which is the half of another codegree in $\cod(\Sy_{4}(4))$.

\item[(c)] $|\cod(G/N)|=6$. By observation, $3^2\cdot 5\cdot7$ is the only nontrivial odd codegree in $\cod(\PSL(3,4))$. However $3^2\cdot5\cdot7\notin\cod(\Sy_{4}(4))$, so $\cod(\PSL(3,4))\not\subset \cod(\Sy_{4}(4))$.
Suppose $\cod(G/N)=\cod({}^2B_2(q))$. $(q-1)(q^2+1)=3^2\cdot5^2\cdot17$ since these are the only nontrivial odd codegrees. We note that $3$ does not divide the order of the Suzuki groups, it is a contradiction.

\item[(d)]  $|\cod(G/N)|=7$. We note that the nontrivial odd codegree in codegree sets of $\PSL(3,3)$, $\A_{7}$, $\M_{11}$ and $\J_{1}$ can't be $3^2\cdot5^2\cdot17$.

\item[(e)] $|\cod(G/N)|=8$.
If $G/N\cong \PSU(3,q)$ with $4<q\not\equiv -1\ (\bmod\ 3)$. Then if $q$ is odd, we have $q^3(q^2-q+1)=3^2\cdot5^2\cdot17$. Since no cube of a prime divides $3^2\cdot5^2\cdot17$, we have a contradiction. If $q$ is even, we have $(q^2-q+1)(q+1)^2(q-1)=3^2\cdot5^2\cdot17$. Then this does not yield an integer root for $q$, a contradiction.

If $G/N \cong \PSL(3,q)$ with $4<q\not\equiv 1\ (\bmod 3)$ and if $q$ is odd, then $q^3(q^2+q+1)=3^2\cdot5^2\cdot17$. We have the same contradiction with $q$. If $q$ is even, then $(q^2+q+1)(q^2-1)(q-1)=3^2\cdot5^2\cdot17$. This is a contradiction since $q$ is an integer.

If $G/N \cong \G_2(2)'$, we see $3^3\cdot7\notin\cod(\Sy_{4}(4))$, so $\cod(\G_2(2)')\not\subseteq \cod(\Sy_{4}(4))$. It is a contradiction.

\item[(f)] $|\cod(G/N)|=9$.
Suppose $G/N \cong \PSL(3,q)$ with $4<q \equiv 1\ (\bmod\ 3)$. If $q$ is odd, then $\frac{1}{3}q^3(q^2+q+1)=3^2\cdot5^2\cdot17$. This is a contradiction since $q \in \Z$. If $q$ is even, then $\frac{1}{3}(q^2+q+1)(q+1)(q-1)^2=3^2\cdot5^2\cdot17$. Then This also does not yield an integer root for $q$. This is a contradiction.

$G/N \cong \PSU(3,q))$ with $4<q \equiv -1\ (\bmod 3)$. If $q$ is odd, then $\frac{1}{3}q^3(q^2-q+1)=3^2\cdot5^2\cdot17$. This does not yield an integer root for $q$, a contradiction. If $q$ is even, then $\frac{1}{3}(q^2-q+1)(q+1)^2(q-1)=3^2\cdot5^2\cdot17$. This also does not yield an integer root for $q$, a contradiction.

\item[(g)] $|\cod(G/N)|=10$. Then $G/N\cong \M_{22}$. We note that $2^7\cdot7\cdot11 \notin \cod (\Sy_{4}(4))$, we have that $\cod (\M_{22})\not\subseteq \cod(\Sy_{4}(4))$. It is a contradiction.

\item[(h)] $|\cod(G/N)|=11$. By observation, the codegree sets of $\PSL(4,2), \M_{12}$ and $\M_{23}$ contain $3^2 \cdot 5 \cdot 7$, $2^2 \cdot 3 \cdot 5 \cdot 11$ and $3^2 \cdot5\cdot11 \cdot 23$ respectively, none of which are elements of the codegree set of $\Sy_{4}(4)$.

If $G/N\cong {}^2\G_2(q)$, $q=3^{2f+1}$. Thus $3^2{\cdot}5^2{\cdot}17$ could equal to either $3^{5f+3}(q^2-q+1)$, $q^3(q+1-3m)$, or $q^3(q+1+3m)$. By considering the power of $3$ or $5$ in those numbers, one see that this cannot happen.

\item[(i)] $|\cod(G/N)|=12$. Then $G/N\cong \Sy_{4}(4)$. It is easily checked that $\cod(G/N)=\cod(G)$. Thus $G/N\cong \Sy_{4}(4)$.
\end{proof}

\begin{lem}\label{redU42}
Let $G$ be a finite group with $\cod(G)=\cod(\U_{4}(2))$. If $N$ is a maximal normal subgroup of $G$, then $G/N \cong \U_{4}(2)$.
\end{lem}

\begin{proof}
Let $N$ be a maximal normal subgroup of $G$.
Similarly with above lemma, we have that $G/N$ is a non-abelian simple group and $|\cod(G/N)|$ is $4, 5, 6, 7, 8, 9, 10, 11, 12$, or $13$. Next we analyze the cases separately.

\item[(a)]$|\cod(G/N)|=4$. Then $G/N\cong \PSL(2,k)$ where $k=2^f\geq 4$ and $\cod(G/N)=\{1, k(k-1), k(k+1), k^2-1\}$.
Hence $k^2-1=3^4\cdot5$, $k\in\mathbb{Z}$ which is a contradiction.

\item[(b)] $|\cod(G/N)|=5$. Then $G/N\cong \PSL(2,k)$ where $k$ is an odd prime power and $\cod(G/N)=\left\{ 1, \frac{k(k-1)}{2}, \frac{k(k+1)}{2}, \frac{k^2-1}{2}, k(k-\epsilon(k))\right\}$ where $\epsilon(k)=(-1)^{(k-1)/2}$. Then $k(k-\epsilon(k))/2$, $k(k-\epsilon(k))\in \cod(G)$, and this implies that $\frac{k(k\pm 1)}{2}=2^3\cdot3^4$ since $2^3\cdot3^4$ is the only element of $\cod(G)$ that is half of another codegree. In either case, $k\in \Z$ which is a contradiction.

\item[(c)] $|\cod(G/N)|=6$. If $G/N\cong {}^2B_2(q)$, then $(q-1)(q^2+1)=3^4\cdot5$ since these are the only nontrivial odd codegrees. We note that $3$ does not divide the order of the Suzuki group, a contradiction.

%by comparing the largest 2-parts of the codegrees. Then $q=4$ which is a contradiction since $q\geq 8$.
If $G/N\cong \PSL(3,4)$, then $3^2\cdot5\cdot7 \notin \cod(\U_{4}(2))$ and we clearly see that $\cod(G/N)\not\subset \cod(G)$.

\item[(d)] $|\cod(G/N)|=7$. Then $G/N\cong \PSL(3,3)$, $\A_{7}$, $\M_{11}$ or $\J_{1}$. We need only to observe that none of the codegree sets for this case are the same.
%Thus we have $G/N\cong \PSL(3,3)$ by Lemma \ref{list}.

\item[(e)]$|\cod(G/N)|=8$.
If $G/N\cong \PSU(3,q)$ with $4<q\not\equiv -1\ (\bmod\ 3)$. Then if $q$ is odd, we have $q^3(q^2-q+1)=3^4\cdot5$. Since no cube of a prime divides $3^4\cdot5$, we have a contradiction. If $q$ is even, we have $(q^2-q+1)(q+1)^2(q-1)=3^4\cdot5$. Then we can get a contradiction which means $q$ is not an integer.

If $G/N \cong \PSL(3,q)$ with $4<q\not\equiv 1\ (\bmod 3)$ and if $q$ is odd, then $q^3(q^2+q+1)=3^4\cdot5$. We have the same contradiction with $q$. If $q$ is even, then $(q^2+q+1)(q^2-1)(q-1)=3^4\cdot5$. This does not yield an integer root for $q$, it is a contradiction.

If $G/N \cong \G_2(2)'$ we see $3^3\cdot7\notin \cod(\U_{4}(2))$, so $\cod(\G_2(2)')\not\subseteq\cod(\U_{4}(2))$. It is a contradiction.

\item[(f)] $|\cod(G/N)|=9$.
Suppose $G/N \cong \PSL(3,q)$ with $4<q\equiv 1\ (\bmod\ 3)$. If $q$ is odd, then $\frac{1}{3}q^3(q^2+q+1)=3^4\cdot5$. Then $q$ is not an integer, it is a contradiction. If $q$ is even, then $\frac{1}{3}(q^2+q+1)(q+1)(q-1)^2=3^4\cdot5$. This equation means $q$ is not an integer, a contradiction.

Suppose $G/N\cong \PSU(3,q))$ with $4<q\equiv -1\ (\bmod 3)$. If $q$ is odd, then $\frac{1}{3}q^3(q^2-q+1)=3^4\cdot5$. We have the same contradiction for $q$ is an integer. If $q$ is even, then $\frac{1}{3}(q^2-q+1)(q+1)^2(q-1)=3^4\cdot5$. $q$ is not an integer, it is a contradiction.

 \item[(g)]$|\cod(G/N)|=10$. Then $G/N\cong \M_{22}$. We observe that $2^7\cdot3\cdot5\cdot11 \notin \cod(\U_{4}(2))$. Thus we have $\cod(G/N)\not\subseteq\cod(G)$.

\item[(h)] $|\cod(G/N)|=11$. If $G/N\cong \PSL(4,2)$, $\M_{12}$ or $\M_{23}$. By observation, the codegree sets of $\PSL(4,2), \M_{12}$ and $\M_{23}$ contain $3^2 \cdot 5 \cdot 7$, $2^2 \cdot 3 \cdot 5 \cdot 11$ and $3^2 \cdot5\cdot11 \cdot 23$ respectively, none of which are elements of the codegree set of $\U_{4}(2)$. It is easily checked that $\cod(G/N)\not\subseteq\cod(G)$. It is a contradiction.

If $G/N \cong {}^2\G_2(q)$, ($q=3^{2f+1}$). We have that $3^4\cdot5$ equals to either $3^{5f+3}(q^2-q+1)$, $q^3(q+1-3m)$, or $q^3(q+1+3m)$. This is a contradiction since $q$ is not an integer.

\item[(i)] $|\cod(G/N)|=12$.  Then $G/N\cong \Sy_{4}(4)$. By observation, $3^2\cdot5^2\cdot17\notin\cod(\U_{4}(2))$. Thus $\cod(G/N)\not\subseteq\cod(G)$, a contradiction.

\item[(j)] $|\cod(G/N)|=13$. Then $\cod(G/N)=\cod(\U_{4}(2))$ or $\cod(\Sy_{4}(q))$.

$\cod(G/N)=\cod(\Sy_{4}(q))$, for $q=2^f>4$. We can easily checked that $3^4\cdot5=(q^2+1)(q+1)^2(q-1)^2$, as the are the only nontrivial codegrees in each set. If $3^4\mid(q+1)^2(q-1)^2$, then $q^2+1\mid5$, this implies $q=2$, which is a contradiction for $(2+1)^2(2-1)\ne3^4$. If $5\mid(q+1)^2(q-1)^2$, then $q^2+1\mid3^4$, this implies $q\not\in\mathbb{Z}$, which is a contradiction.

$\cod(G/N)=\cod(\U_{4}(2))$. We need only observe that the smallest codegrees for this case are the same. Thus $G/N \cong \U_{4}(2)$.
\end{proof}

\begin{lem}\label{redS4q}
Let $G$ be a finite group with $\cod(G)=\cod(\Sy_{4}(q))$, $q>4$ is even. If $N$ is a maximal normal subgroup of $G$, then $G/N \cong \Sy_{4}(q)$.
\end{lem}

\begin{proof}
Let $N$ be a maximal normal subgroup of $G$.
Similarly with above lemma, we have that $G/N$ is a non-abelian simple group and $|\cod(G/N)|$ is $4, 5, 6, 7, 8, 9, 10, 11, 12$, or $13$. Next we analyze the cases separately.

\item[(a)] $|\cod(G/N)|=4$. By Lemma $2.3,$ Then $G/N\cong \PSL(2,k)$ where $k=2^f\geq 4$ and $\cod(G/N)=\{1, k(k-1), k(k+1), k^2-1\}$. If $q$ is even, then $k^2-1=(q^2+1)(q-1)^2(q+1)^2$, as they are the only nontrivial odd codegree of each set. If $k+1\mid(q-1)^2(q+1)^2$, then $q^2+1\mid k-1$. Clearly $(q-1)^2(q+1)^2-(q^2+1)\ne2$, this implies that $k+1-(k-1)\ne2$, it is a contradiction. If $k-1\mid (q-1)^2(q+1)^2$, then $q^2+1\mid k+1$. Obviously $(q-1)^2(q+1)^2>q^2+1$, this implies $k-1>k+1$, a contradiction.

\item[(b)] $|\cod(G/N)|=5$. Then $G/N\cong \PSL(2,k)$ where $k$ is an odd prime power and $\cod(G/N)=\left\{ 1, \frac{k(k-1)}{2}, \frac{k(k+1)}{2}, \frac{k^2-1}{2}, k(k-\epsilon(k))\right\}$ where $\epsilon(k)=(-1)^{(k-1)/2}$. Then $k(k-\epsilon(k))/2$, $k(k-\epsilon(k))\in\cod(G)$, a contradiction for we can't find a codegree which is the half of another codegree in $|\cod(\Sy_{4}(q))|$.

\item[(c)] $|\cod(G/N)|=6$. If $G/N\cong {}^2B_2(q)$, $q$ is even, then $(s-1)(s^2+1)=(q^2+1)(q-1)^2(q+1)^2$ since these are the only nontrivial odd codegrees. If $s-1\mid (q-1)^2(q+1)^2$, then $q^2+1\mid s^2+1$. This implies $q$ is not an integer, a contradiction. If $s^2+1\mid (q-1)^2(q+1)^2$, then $q^2+1\mid s-1$, we have that $q$ is not an integer, a contradiction.

%by comparing the largest 2-parts of the codegrees. Then $q=4$ which is a contradiction since $q\geq 8$.
If $G/N\cong \PSL(3,4)$, then $\cod(G/N)=\{ 1,\ 2^4\cdot3^2\cdot7,\ 2^6\cdot3^2,\ 2^6\cdot5,\ 2^6\cdot7,\ 3^2\cdot5\cdot7 \}$. If $q$ is even, it can be easily checked that $3^2\cdot5\cdot7=(q^2+1)(q-1)^2(q+1)^2$. It is a contradiction for $q>4$.

\item[(d)]$|\cod(G/N)|=7$. Then $G/N\cong \PSL(3,3)$, $\A_{7}$, $\M_{11}$ or $\J_{1}$. We need only to observe that none of the codegree sets for this case are the same.
%Thus we have $G/N\cong \PSL(3,3)$ by Lemma \ref{list}.
Suppose $G/N \cong \PSL(3,3)$. Note that $3^3\cdot13$ is the only nontrivial odd codegree. If $q$ is even, $3^3\cdot13=(q^2+1)(q-1)^2(q+1)^2$ is a contradiction for $q>4$.

Suppose $G/N \cong \A_{7}$. Since every nontrivial codegree in $\cod(\A(7))$ is even, then $\cod(\A(7))$ is the subset of $\cod(G)$ after deleting the unique nontrivial odd codegree. If $q$ is even, then by comparing the 2-parts of the codegrees of $\cod(\A(7))$ and $\cod(G)$, we have a contradiction since $q>4$.

Suppose $G/N \cong \M_{11}$. Then $G/N\cong \M_{11}$.
If $q$ is even, then $3^2\cdot5\cdot11=(q^2+1)(q-1)^2(q+1)^2$, as they are the only nontrivial odd codegree of each set. Note that $(3^2,5\cdot11)=1$. If $5\cdot11\mid (q-1)^2(q+1)^2$, then $q^2+1\mid 3^2$. It is a contradiction since $q \in\mathbb{Z}$. If $3^2\mid(q-1)^2(q+1)^2$, then $q^2+1\mid5\cdot11$. Obviously $(q-1)^2(q+1)^2>q^2+1$, this implies $3^2>5\cdot11$, a contradiction.

Suppose $G/N \cong \J_{1}$. Then $G/N\cong \J_{1}$. Since the only nontrivial odd codegree in $\cod \Sy_4(q))$ is $(q^2+1)(q-1)^2(q+1)^2$. We need only note that $\cod(\J_{1})$ has two nontrivial odd codegrees, it is a contradiction.

\item[(e)] $|\cod(G/N)|=8$.
If $G/N \cong \PSL(3,s)$ with $4<s\not\equiv 1\ (\bmod 3)$. We see that $s^3(s+1)(s-1)^2$ the only nontrivial codegree which is divided by other three codegrees in $\cod(\PSL(3,s))$. In $\cod(\Sy_{4}(q))$, there are two nontrivial codegrees in this case, which are $q^4(q+1)(q-1)^2$ and $q^4(q+1)^2(q-1)$. Then $s^3(s+1)(s-1)^2$ may be equal to $q^4(q+1)(q-1)^2$ or $q^4(q+1)^2(q-1)$. Furthermore $s^3(s-1)^2=q^4(q+1)(q-1)$ or $q^4(q+1)^2$, as $s^3(s-1)$ is the smallest codegree among three codegrees which divide same codegree in $\cod(\PSL(3,s))$ and $q^4(q+1)(q-1)$ or $q^4(q+1)^2$ satisfies similar condition in $\cod(G)$. This means $q\notin\mathbb{Z}$, a contradiction.

If $G/N\cong \PSU(3,s)$ with $4<s\not\equiv -1\ (\bmod\ 3)$. Similar with the analysis above paragraph, we see that $s^3(s+1)^2(s-1)=q^4(q+1)(q-1)^2$ or $q^4(q+1)^2(q-1)$. This means $q \notin\mathbb{Z}$, a contradiction.

If $G/N \cong \G_2(2)'$ we see $3^3\cdot7=(q^2+1)(q+1)^2(q-1)^2$, as they are the only nontrivial odd codegrees in each set. By setting this equality, we obtain no integer solution for $q$, a contradiction.

\item[(f)] $|\cod(G/N)|=9$.

Suppose $G/N \cong \PSL(3,s)$ with $4<s\equiv 1\ (\bmod\ 3)$. We have $\frac{1}{3}s^3(s+1)(s-1)^2=q^4(q+1)(q-1)$ or $q^4(q+1)^2$. This means $q \notin\mathbb{Z}$, a contradiction.

Suppose $G/N \cong \PSU(3,q))$ with $4<q\equiv -1\ (\bmod 3)$. We have $\frac{1}{3}s^3(s+1)^2(s-1)=q^4(q+1)(q-1)$ or $q^4(q+1)^2$. This means $q \notin\mathbb{Z}$, a contradiction.

\item[(g)] $|\cod(G/N)|=10$. Then $G/N \cong \M_{22}$. We compare the largest 2-part of each codegrees set and obtain that $q$ is not an integer, a contradiction.

\item[(h)] $|\cod(G/N)|=11$. If $G/N \cong \PSL(4,2)$, $\M_{12}$ or $\M_{23}$.

Suppose $G/N \cong \PSL(4,2)$. It is easily checked that $3^2\cdot5\cdot7$ is the only nontrivial odd codegree in $\cod(\PSL(4,2))$. Hence $3^2\cdot5\cdot7=(q^2+1)(q-1)^2(q+1)^2$. Note that $3^2\cdot5\cdot7=315$ is not divisible by a square of prime, a contradiction.

If $G/N \cong {}^2\G_2(s)$, $s=3^{2t+1}$. We have that $(q^2+1)(q-1)^2(q+1)^2$ equals to either $3^{5s+3}(s^2-s+1)$, $s^3(s+1-3m)$, or $s^3(s+1+3m)$. This is a contradiction since $q$ is not an integer.

\item[(i)] $|\cod(G/N)|=12$. Then $G/N \cong \Sy_{4}(4)$. By observation, $s^3(s+1)(s-1)^2$ is the only nontrivial codegree which is divided by other three codegrees in $\cod(\Sy_{4}(4))$. In $\cod( \Sy_{4}(q))$, there are two nontrivial codegrees in this case, which are $q^4(q+1)(q-1)^2$ and $q^4(q+1)^2(q-1)$. Then $s^3(s+1)(s-1)^2$ may be equal to $q^4(q+1)(q-1)^2$ or $q^4(q+1)^2(q-1)$. Furthermore $2^8\cdot5^2=q^4(q+1)(q-1)$ or $q^4(q+1)^2$, as $2^8\cdot5^2$ is the smallest codegree among three codegrees which divide same codegree in $\cod(\Sy_{4}(4))$ and $q^4(q+1)(q-1)$ or $q^4(q+1)^2$ satisfies similar condition in $\cod(G)$. This means a contradiction for $q>4$.

\item[(j)]$|\cod(G/N)|=13$. We need only observe that $3^4\cdot5=(q^2+1)(q+1)^2(q-1)^2$, as they are the only nontrivial odd codegrees of each set. If $3^4 \mid (q+1)^2(q-1)^2$, then $q^2+1 \mid 5$, this implies $q$ is not an integer, which is a contradiction. If $5\mid(q+1)^2(q-1)^2$, then $q^2+1\mid3^4$, this implies $q\not\in\mathbb{Z}$, which is a contradiction.

Thus $\cod(G/N)=\cod(\Sy_{4}(s))$ for $s=2^t>4$. Comparing the smallest codegrees, we see $q=s$. Thus, $G/N \cong \Sy_{4}(q)$.
\end{proof}

\begin{lem}\label{redU43}
Let $G$ be a finite group with $\cod(G)=\cod(\U_{4}(3))$. If $N$ is a maximal normal subgroup of $G$, then $G/N\cong \U_{4}(3)$.
\end{lem}

\begin{proof}
Let $N$ be a maximal normal subgroup of $G$. Then $G/N$ is a non-abelian simple group and $|\cod(G/N)|\leq 14$.
%Since there is no nontrivial odd codegree in $\cod(M_{22})$, we have that  $G/N\cong \Alt_7$ or $M_{22}$. It is easily checked that   $\cod(\Alt_7)\not\subset \cod(M_{22})$.  Therefore $G/N\cong M_{22}$.

\item[(a)]$|\cod(G/N)|=4$. Then $G/N \cong \PSL(2,k)$ where $k=2^f\geq 4$, and $\cod(G/N)=\{1, k(k-1), k(k+1), k^2-1\}$. Thus $k^2-1=3^6\cdot7$ or $3^6\cdot5$, since these are nontrivial odd codegrees. It is a contradiction since $k\in\Z$.

\item[(b)]$|\cod(G/N)|=5$. Then $G/N\cong \PSL(2,k)$ where $k$ is an odd prime power and $\cod(G/N)=\left\{ 1, \frac{k(k-1)}{2}, \frac{k(k+1)}{2}, \frac{k^2-1}{2}, k(k-\epsilon(k))\right\}$ where $\epsilon(k)=(-1)^{(k-1)/2}$. So $k(k-\epsilon(k))/2$, $k(k-\epsilon(k))\in \cod(G)$. This implies that $\frac{k(k\pm 1)}{2}=2^5\cdot3^6$ or $\frac{k(k\pm 1)}{2}=2^4\cdot3^6$ since $2^4\cdot3^6$ is half of $2^5\cdot3^6$, and $2^3\cdot3^6$ is half of $2^4\cdot3^6$ in $\cod(G)$. In either case, $k\in \Z$ which is a contradiction.

\item[(c)] $|\cod(G/N)|=6$. By observation, $\cod(\PSL(3,4))\not\subset \cod(\U_{4}(3))$ since $3^2\cdot5\cdot7 \notin \cod(G)$.

Suppose $\cod(G/N)=\cod({}^2B_2(q))$. We note that $(q-1)(q^2+1)=3^6\cdot7$ or $3^6\cdot5$, since these are nontrivial odd codegrees. We note that $3$ does not divide the order of the Suzuki groups, a contradiction.

\item[(d)] Suppose $|\cod(G/N)|=7$. We note that the nontrivial odd codegree in codegree sets of $\J_1, \M_{11}$, and $\PSL(3,3)$ are $3\cdot5\cdot11\cdot19$, $3^2\cdot5\cdot 11$ and $3^3\cdot 13$, respectively. Also, $\cod(\A_7) \not\subset \cod(\U_{4}(3))$ by an easy observation.

\item[(e)]$|\cod(G/N)|=8$.

If $G/N\cong \PSU(3,q)$ with $4<q\not\equiv -1\ (\bmod\ 3)$ then if $q$ is odd we have $q^3(q^2-q+1)=3^6\cdot7$ or $3^6\cdot5$.This equation means a contradiction since $q \in \Z$. If $q$ is even we have $(q^2-q+1)(q+1)^2(q-1)=3^6\cdot7$ or $3^6\cdot5$. Then we have a same contradiction since $q \in \Z$.

If $G/N \cong \PSL(3,q)$ with $4<q\not\equiv 1\ (\bmod 3)$ and $q$ is odd, then $q^3(q^2+q+1)=3^6\cdot7$ or $3^6\cdot5$. We have the same contradiction with $q \in \Z$. If $q$ is even, then $(q^2+q+1)(q^2-1)(q-1)=3^6\cdot7$ or $3^6\cdot5$. Then it is a contradiction since Then we have a same contradiction since $q \in \Z$.

If $G/N \cong \G_2(2)'$ we see $\cod(\G_2(2)')\not\subseteq\cod(\U_{4}(3))$ since $3^3\cdot7 \notin \cod(G)$.

\item[(f)] $|\cod(G/N)|=9$.

Suppose $G/N \cong \PSL(3,q)$ with $4<q\equiv 1\ (\bmod\ 3)$. If $q$ is odd, then $\frac{1}{3}q^3(q^2+q+1)=3^6\cdot7$ or $3^6\cdot5$. It is a contradiction since $q\in\Z$. If $q$ is even, then $\frac{1}{3}(q^2+q+1)(q+1)(q-1)^2=3^6\cdot7$ or $3^6\cdot5$. This means $q\notin\Z$, a contradiction.

Suppose $G/N\cong\PSU(3,q))$ with $4<q\equiv -1\ (\bmod 3)$.
If $q$ is odd, then $\frac{1}{3}q^3(q^2-q+1)=3^6\cdot7$, it is a contradiction since $q\in\Z$. If $q$ is even, then $\frac{1}{3}(q^2-q+1)(q+1)^2(q-1)=3^6\cdot7$ or $3^6\cdot5$, which means $q$ is not an integer, it is a contradiction.

\item[(g)] $|\cod(G/N)|=10$. We note that $\cod(\M_{22})\not\subseteq\cod(\U_{4}(3))$ since $2^7\cdot3\cdot5\cdot11 \notin \cod(G)$.

\item[(h)] $|\cod(G/N)|=11$. We can checked that the codegree sets of $\PSL(4,2), \M_{12}$ and $\M_{23}$ contain $3^2 \cdot 5 \cdot 7$, $2^2 \cdot 3 \cdot 5 \cdot 11$ and $3^2 \cdot5\cdot11 \cdot 23$ respectively, none of which are elements of the codegree set of $\U_{4}(3)$.  We note that $G/N$ cannot be isomorphic to $\PSL(4,2)$, $\M_{12}$ or $\M_{23}$. It is a contradiction.

If $G/N \cong {}^2\G_2(3^{2f+1})$, then by comparing the codegree sets, we have that $3^6\cdot7$ or $3^6\cdot5$ can equal to $3^{5f+3}(q^2-q+1)$, $q^3(q+1-3m)$ or $q^3(q+1+3m)$, since they are the nontrivial odd codegrees in each set, it is a contradiction since $q\ge27$.

\item[(i)] $|\cod(G/N)|=12$. Then $G/N\cong \Sy_{4}(4)$. We note that $3^2\cdot5^2\cdot17\notin\cod(G)$. Hence $\cod(G/N)\not\subseteq\cod(G)$, it is a contradiction.

\item[(j)] $|\cod(G/N)|=13$. Then $\cod(G/N)=\cod(\U_{4}(2))$ or $\cod(\Sy_{4}(q))$, for $q=2^f>4$.

Suppose $\cod(G/N)=\cod(\U_{4}(2))$. We need only observe that $3^4\cdot5\notin\cod(G)$ which means $\cod(G/N)\not\subseteq\cod(G)$. It is a contradiction.

Suppose $\cod(G/N)=\cod(\Sy_{4}(q))$, for $q=2^f>4$. Note that $(q^2+1)(q+1)^2(q-1)^2=3^6\cdot7$ or $3^6\cdot5$, since they are nontrivial odd codegrees in each set. This has no integer solution. A contradiction.

\item[(k)]$|\cod(G/N)|=14$. It is easily checked that neither $3^3\cdot5^2\cdot13$ nor $3^5\cdot5\cdot17$ is in $\cod(G)$. Thus $\cod({^{2}\F_{4}(2){}}')\not\subseteq\cod(G)$ and $\cod(\J_{3})\not\subseteq\cod(G)$.

We have that $G/N\cong \U_{4}(3)$.
\end{proof}

\begin{lem}\label{redF42}
Let $G$ be a finite group with $\cod(G)=\cod(^{2}\F_{4}(2){}')$. If $N$ is a maximal normal subgroup, then $G/N \cong {^{2}\F_{4}(2){}}'$.
\end{lem}

\begin{proof}
Let $N$ be a maximal normal subgroup of $G$. Then $G/N$ is a non-abelian simple group and $|\cod(G/N)|$ is either $4, 5, 6, 7, 8, 9, 10, 11, 12, 13$ or $14$.
   %Since there is no nontrivial odd codegree in $\cod(M_{12})$, we have that  $G/N\cong \Alt_7, M_{22}$ or $M_{12}$. It is easily checked that   $\cod(\Alt_7)\not\subset \cod(M_{12})$ and $\cod(M_{22})\not\subset \cod(M_{12})$ .Thus $G/N \cong M_{12}$.

\item[(a)] $|\cod(G/N)|=4$. Then $G/N\cong\PSL(2,k)$ where $k=2^f\geq 4$, and $\cod (G/N)=\{1, k(k-1), k(k+1), k^2-1\}$. Thus $k^2-1=3^3\cdot5^2\cdot13$ since these are the only nontrivial odd codegrees. It is a contradiction since $k\in\Z$.

\item[(b)] $|\cod(G/N)|=5$. Then $G/N\cong\PSL(2,k)$ where $k$ is an odd prime power and $\cod(G/N)=\left\{ 1, \frac{k(k-1)}{2}, \frac{k(k+1)}{2}, \frac{k^2-1}{2}, k(k-\epsilon(k))\right\}$ where $\epsilon(k)=(-1)^{(k-1)/2}$. So $k(k-\epsilon(k))/2$, $k(k-\epsilon(k))\in\cod(G)$. This implies that $\frac{k(k\pm 1)}{2}=2^{10}\cdot3^3$, $2^{9}\cdot3^3$, or $\frac{k(k\pm 1)}{2}=2^{10}\cdot5^2$ since they are half of another codegree respectively. In either case, $k\not\in\Z$ which is a contradiction.

\item[(c)] $|\cod(G/N)|=6$. By observation, $\cod(\PSL(3,4))\not\subset\cod(^{2}\F_{4}(2){}')$ since $3^2\cdot5\cdot7\notin\cod(G)$.

Suppose $\cod(G/N)=\cod({}^2B_2(q))$. We note that $(q-1)(q^2+1)=3^3\cdot5^2\cdot13$ since these are the only nontrivial odd codegrees. We note that $3$ does not divide the order of the Suzuki groups, a contradiction.

\item[(d)] $|\cod(G/N)|=7$. Then $G/N\cong \PSL(3,3)$, $\A_{7}$, $\M_{11}$ or $\J_{1}$. We note that the nontrivial odd codegree in codegree sets of $\J_1, \M_{11}$, $\A_7$ and $\PSL(3,3)$ are $3\cdot5\cdot11\cdot19$, $3^2\cdot5\cdot 11$, $2^2\cdot3\cdot5\cdot7$ and $3^3\cdot 13$, respectively. It is a contradiction.

\item[(e)] $|\cod(G/N)|=8$.

If $G/N\cong \PSU(3,q)$ with $4<q\not\equiv -1\ (\bmod\ 3)$ then if $q$ is odd we have $q^3(q^2-q+1)=3^3\cdot5^2\cdot13$. Thus it does not yield an integer root for $q$, a contradiction. If $q$ is even we have $(q^2-q+1)(q+1)^2(q-1)=3^3\cdot5^2\cdot13$. This also does not yield an integer root for $q$. This is a contradiction.

If $G/N\cong\PSL(3,q)$ with $4<q\not\equiv 1\ (\bmod\ 3)$ and $q$ is odd, then $q^3(q^2+q+1)=3^3\cdot5^2\cdot13$. We have the same contradiction with $q^3$. If $q$ is even, then $(q^2+q+1)(q^2-1)(q-1)=3^3\cdot5^2\cdot13$. It is a contradiction since $q \in \Z$.

If $G/N\cong \G_2(2)'$ we see $\cod(\G_2(2)')\not\subseteq\cod({^{2}\F_{4}(2){}}')$ since $3^3\cdot7\notin\cod(G)$.

\item[(f)] $|\cod(G/N)|=9$. Suppose $G/N\cong \PSL(3,q)$ with $4<q\equiv 1\ (\bmod\ 3)$. If $q$ is odd, then $\frac{1}{3}q^3(q^2+q+1)=3^3\cdot5^2\cdot13$. It is a contradiction since $q$ is not an integer. If $q$ is even, then $\frac{1}{3}(q^2+q+1)(q+1)(q-1)^2=3^3\cdot5^2\cdot13$. It is a same contradiction since $q$ is an integer.

Suppose $G/N\cong\PSU(3,q))$ with $4<q\equiv -1\ (\bmod3)$. If $q$ is odd, then $\frac{1}{3}q^3(q^2-q+1)=3^3\cdot5^2\cdot13$, it is a contradiction since $q\in\Z$. If $q$ is even, then $\frac{1}{3}(q^2-q+1)(q+1)^2(q-1)=3^3\cdot5^2\cdot13$. It is a contradiction since $q$ in not an integer.

\item[(g)] $|\cod(G/N)|=10$. We note that $\cod(\M_{22})\not\subseteq\cod({^{2}\F_{4}(2){}}')$ for $2^7\cdot3\cdot5\cdot11 \notin \cod(G)$. A contradiction.

\item[(h)] $|\cod(G/N)|=11$. By observation, the codegree sets of $\PSL(4,2), \M_{12}$ and $\M_{23}$ contain $3^2 \cdot 5 \cdot 7$, $2^2 \cdot 3 \cdot 5 \cdot 11$ and $3^2 \cdot5\cdot11 \cdot 23$ respectively, none of which are elements of the codegree set of ${^{2}\F_{4}(2) {}}'$. It is easily checked that $\cod(G/N)\not\subseteq\cod(G)$. We note by observation that $G/N$ cannot be isomorphic to $\PSL(4,2)$, $\M_{12}$ or $\M_{23}$.

If $G/N\cong {}^2\G_2(3^{2f+1})$, then by comparing the nontrivial odd codegree in each set, we have a contradiction since $q\ge27$.

\item[(i)] $|\cod(G/N)|=12$. Then $G/N\cong \Sy_{4}(4)$. We note that $3^2\cdot5^2\cdot17\notin\cod(G)$. Hence $\cod(G/N)\not\subseteq\cod(G)$, it is a contradiction.

\item[(j)] $|\cod(G/N)|=13$. Then $\cod(G/N)=\cod(\U_{4}(2))$ or $\Sy_{4}(q)$, for $q=2^f>4$.

Suppose $\cod(G/N)=\cod(\U_{4}(2))$. We need only observe that $3^4\cdot5\notin\cod(G)$ which means $\cod(G/N)\not\subseteq\cod(G)$. It is a contradiction.

Suppose $\cod(G/N)=\cod(\Sy_{4}(q))$, for $q=2^f>4$. We have that $(q^2+1)(q+1)^2(q-1)^2=3^3\cdot5^2\cdot13$, as they are the only nontrivial odd codegrees in each set. It implies $q\notin\Z$, a contradiction.

\item[(k)] $|\cod(G/N)|=14$. It is easily checked that neither $3^6\cdot7$ nor $3^5\cdot5\cdot17$ is in $\cod(G)$. Thus $\cod(\U_{4}(3))\not\subseteq\cod(G)$ nor $\cod(\J_{3})\not\subseteq \cod(G)$. We have that $G/N\cong {^{2}\F_{4}(2) {}}'$.
\end{proof}

\begin{lem}\label{redJ3}
Let $G$ be a finite group with $\cod(G)=\cod(\J_{3})$. If $N$ is a maximal normal subgroup of $G$, then $G/N\cong \J_{3}$.
\end{lem}

\begin{proof}
Let $N$ be a maximal normal subgroup of $G$. Then $G/N$ is a non-abelian simple group and $|\cod(G/N)|$ is either $4, 5, 6, 7, 8, 9, 10, 11, 12, 13$ or $14$.

\item[(a)] $|\cod(G/N)|=4$. Then, $G/N\cong \PSL(2,k)$ where $k=2^f\geq4$ and $\cod(G/N)=\{1, k(k-1), k(k+1), k^2-1\}$. Hence $k^2-1=3^4\cdot 17\cdot19$ or $3^5\cdot5\cdot17$ since these are the nontrivial odd codegrees. In either case $k$ is not an integer which is a contradiction.

\item[(b)] $|\cod(G/N)|=5$. Then $G/N\cong \PSL(2,k)$ where $k$ is an odd prime power and $\cod(G/N)=\left\{ 1, \frac{k(k-1)}{2}, \frac{k(k+1)}{2}, \frac{k^2-1}{2}, k(k-\epsilon(k))\right\}$ where $\epsilon(k)=(-1)^{(k-1)/2}$. Then $k(k-\epsilon(k))/2$, $k(k-\epsilon(k))\in \cod(G)$, and this implies that $\frac{k(k\pm1)}{2}=2^6\cdot3^5\cdot5$ since $2^6\cdot 3^5\cdot5$ is the only element of $\cod(G)$ that is half of $2^7\cdot3^5\cdot5$. In either case, $k\not\in \Z$ which is a contradiction.

\item[(c)] $|\cod(G/N)|=6$. By observation, $\cod(\PSL(3,4))\not\subset \cod(\J_{3})$ for $3^2\cdot5\cdot7 \notin \cod(G)$.

Suppose $\cod(G/N)=\cod({}^2B_2(q))$. We note that $(q-1)(q^2+1)=3^4\cdot17\cdot19$ or $3^5\cdot5\cdot17$ since these are the nontrivial odd codegrees. We note that $3$ does not divide the order of the Suzuki groups, a contradiction.

\item[(d)] $|\cod(G/N)|=7$. We note that $ \M_{11}$, and $\PSL(3,3)$ can't be $3^2\cdot5\cdot11\cdot23$.
That the nontrivial odd codegree $3^2\cdot5\cdot11$, and $3^3\cdot13$ in the codegree set $\cod(\M_{11})$ and $\cod(\PSL(3,3))$ is not in $\cod(G)$. Also, $\cod(\A_{7})\not\subseteq\cod(G)$ and $\cod(\J_{1})\not\subseteq\cod(G)$ by an easy observation since $2^2\cdot3\cdot5\cdot7$ nor $3\cdot5\cdot11\cdot19 \notin \cod(G)$.

\item[(e)] $|\cod(G/N)|=8$.

If $G/N\cong \PSU(3,q)$ with $4<q\not\equiv -1\ (\bmod\ 3)$ then if $q$ is odd we have $q^3(q^2-q+1)=3^5\cdot5\cdot17$ or $3^4\cdot17\cdot19$. Since no cube of a prime divides $3^5\cdot5\cdot17$ or $3^4\cdot17\cdot19$, we have a contradiction. If $q$ is even we have $(q^2-q+1)(q+1)^2(q-1)=3^5\cdot5\cdot17$, then $(q+1)^2=3^5$. This is a contradiction since $q \in \Z$. Or $(q^2-q+1)(q+1)^2(q-1)=3^4\cdot17\cdot19$. It is a contradiction since $q$ is not an integer.

If $G/N \cong \PSL(3,q)$ with $4<q\not\equiv 1\ (\bmod 3)$ and $q$ is odd, then $q^3(q^2+q+1)=3^5\cdot5\cdot17$ or $3^4\cdot17\cdot19$. We have the same contradiction with $q$. If $q$ is even, then $(q^2+q+1)(q^2-1)(q-1)=3^5\cdot5\cdot17$, which means $q\notin \Z$, a contradiction. Or $(q^2+q+1)(q^2-1)(q-1)=3^4\cdot17\cdot19$. We can get the same contradiction as above.

If $G/N \cong \G_2(2)'$ we see $\cod(\G_2(2)')\not\subseteq\cod(\J_{3})$ since $3^3\cdot7 \notin \cod(G)$. This is a contradiction.

\item[(f)] $|\cod(G/N)|=9$.

Suppose $G/N \cong \PSL(3,q)$ with $4<q\equiv 1\ (\bmod\ 3)$. If $q$ is odd, then $\frac{1}{3}q^3(q^2+q+1)=3^5\cdot5\cdot17$. This does not an integer solution for $q$. It is a contradiction since. Or $\frac{1}{3}q^3(q^2+q+1)=3^4\cdot17\cdot19$. It is a same contradiction since $q\in \Z$. If $q$ is even. Then $\frac{1}{3}(q^2+q+1)(q+1)(q-1)^2=3^5\cdot5\cdot17$. It is a contradiction since $q \in \Z$. Or $frac{1}{3}(q^2+q+1)(q+1)(q-1)^2=3^4\cdot17\cdot19$, it is a contradiction since $q\in \Z$.

Suppose $G/N\cong \PSU(3,q))$ with $4<q \equiv-1\ (\bmod3)$. We have the same contradiction with $q$ as above. If $q$ is even, then $\frac{1}{3}(q^2-q+1)(q+1)^2(q-1)=3^5\cdot5\cdot17$. It is the same contradiction since $q \in\Z$. Or $\frac{1}{3}(q^2-q+1)(q+1)^2(q-1)=3^4\cdot17\cdot19$, it is a same contradiction as above.

\item[(g)] $|\cod(G/N)|=10$. We note that $\cod(\M_{22})\not\subseteq \cod(\J_{3})$ since $2^7\cdot3\cdot5\cdot11 \notin \cod(G)$.

\item[(h)] $|\cod(G/N)|=11$.

We note by observation that $G/N$ cannot be isomorphic to $\PSL(4,2)$, $\M_{23}$ or $\M_{12}$. we have $G/N \cong \PSL(4,2),$ $G/N \cong \M_{12},$ $G/N \cong \M_{23}$ or $G/N \cong {}^{2}\G_{2}(q)$ for $q = 3^{2f+1}, f \geq 1.$  The codegree sets of $\PSL(4,2), \M_{12}$ and $\M_{23}$ contain $3^2 \cdot 5 \cdot 7$, $2^2 \cdot 3 \cdot 5 \cdot 11$ and $3^2 \cdot5\cdot11 \cdot 23$ respectively, none of which are elements of the codegree set of $\J_3.$

If $G/N \cong {}^2\G_2(3^{2f+1})$, then by comparing the nontrivial odd codegree in each set, we can't get an integer root for $q$ which is a contradiction.

\item[(i)] $|\cod(G/N)|=12$. Then $G/N \cong \Sy_{4}(4)$. We note that $3^2\cdot5^2\cdot17\notin \cod(G)$. Thus $\cod(\Sy_{4}(4))\not\subseteq\cod(G)$.

\item[(j)] $|\cod(G/N)|=13$. Then $\cod(G/N)=\cod(\U_{4}(2))$ or $\cod(\Sy_{4}(q))$, for $q=2^f>4$.

Suppose $\cod(G/N)=\cod(\U_{4}(2))$. We have that $\cod(G/N) \not\subseteq \cod(G)$ by an easy observation since $3^4\cdot5 \notin \cod(G)$.

Suppose $\cod(G/N)=\cod(\Sy_{4}(q))$, for $q=2^f>4$. We have that $(q^2+1)(q+1)^2(q-1)^2$ can be equal to $3^5\cdot5\cdot17$ or $3^4\cdot17\cdot19$. If $(q^2+1)(q+1)^2(q-1)^2=3^5\cdot5\cdot17$, it is a contradiction since $3^5\cdot5\cdot17$ is not divisible by a square of prime. If $(q^2+1)(q+1)^2(q-1)^2=3^4\cdot17\cdot19$. We can't get an integer root for $q$. This implies a contradiction.

\item[(k)]Suppose $|\cod(G/N)|=14$. Neither $3^6\cdot7$ nor $3^3\cdot5^2\cdot13 \notin \cod(G)$. We have that $\cod(\U_{4}(3)) \not \subseteq cod(G)$ nor $\cod({^{2}\F_{4}(2){}}')$.

Thus $G/N\cong \J_{3}$.
\end{proof}

\begin{lem}\label{redG23}
Let $G$ be a finite group with $\cod(G)=\cod(\G_{2}(3))$. If $N$ is a maximal normal subgroup of $G$, then $G/N \cong \G_{2}(3)$.
\end{lem}

\begin{proof}

Let $N$ be a maximal normal subgroup of $G$. Then $G/N$ is a non-abelian simple group and $|\cod(G/N)|$ is either $4, 5, 6, 7, 8, 9, 10, 11, 12, 13, 14$ or $15$.

\item[(a)] $|\cod(G/N)|=4$. Then, $G/N \cong \PSL(2, k)$ where $k=2^f\geq 4$ and $\cod (G/N)=\{1, k(k-1), k(k+1), k^2-1\}$. Hence $k^2-1=3^6\cdot7\cdot13$, $3^6\cdot7$ or $3^6\cdot13$ since these are the nontrivial odd codegrees. In either case $k$ is not an integer which is a contradiction.

\item[(b)] $|\cod(G/N)|=5$. Then $G/N\cong \PSL(2,k)$ where $k$ is an odd prime power and $\cod(G/N)=\left\{ 1, \frac{k(k-1)}{2}, \frac{k(k+1)}{2}, \frac{k^2-1}{2}, k(k-\epsilon(k))\right\}$ where $\epsilon(k)=(-1)^{(k-1)/2}$. Then $k(k-\epsilon(k))/2$, $k(k-\epsilon(k))\in \cod(G)$, and this implies that $\frac{k(k\pm1)}{2}=2^5\cdot3^6$ or $\frac{k(k\pm1)}{2}=2^5\cdot3^5$ since these are half of the other two elements. In either case, $k\not\in \Z$ which is a contradiction.

\item[(c)] $|\cod(G/N)|=6$.

By observation, $\cod(\PSL(3,4))\not\subset \cod(\G_{2}(3))$ since $3^2\cdot5\cdot7 \notin \cod(G)$, a contradiction.

Suppose $\cod(G/N)=\cod({}^2B_2(q))$. We note that $(q-1)(q^2+1)=3^6\cdot7\cdot13$, $3^6\cdot7$ or $3^6\cdot13$ since these are the nontrivial odd codegrees. We note that $3$ does not divide the order of the Suzuki groups, a contradiction.

\item[(d)]$|\cod(G/N)|=7$. We note that $\M_{11}$, $\J_{1}$, $\A_{7}$ and $\PSL(3,3)$. None of these codegree sets are included in $\cod(G)$.

\item[(e)] $|\cod(G/N)|=8$.

If $G/N\cong \PSU(3,q)$ with $4<q \not\equiv -1\ (\bmod\ 3)$, then if $q$ is odd we have $q^3(q^2-q+1)=3^6\cdot7\cdot13$, $3^6\cdot7$ or $3^6\cdot13$. These equations can't yield an integer root for $q$. We have a contradiction. If $q$ is even, we have a contradiction for the same reason.

If $G/N \cong \PSL(3,q)$ with $4<q \not\equiv 1\ (\bmod 3)$ and $q$ is odd, then $q^3(q^2+q+1)=3^6\cdot7\cdot13$, $3^6\cdot7$ or $3^6\cdot13$. We have that these equations don't hold as above. If $q$ is even, then $(q^2+q+1)(q^2-1)(q-1)=3^6\cdot7\cdot13$, $3^6\cdot7$ or $3^6\cdot13$. We have the same contradiction.

If $G/N\cong \G_2(2)'$ we see $\cod(\G_2(2)')\not\subseteq\cod(\G_{2}(3))$ for $3^3\cdot7 \notin \cod(G)$. This is a contradiction.

\item[(f)] $|\cod(G/N)|=9$.

Suppose $G/N\cong \PSL(3,q)$ with $4<q \equiv 1\ (\bmod\ 3)$. If $q$ is odd, then $\frac{1}{3}q^3(q^2+q+1)=3^6\cdot7\cdot13$, $3^6\cdot7$ or $3^6\cdot13$. This does not yield an integer root for $q$. It is a contradiction. If $q$ is even. Then $\frac{1}{3}(q^2+q+1)(q+1)(q-1)^2=3^6\cdot7\cdot13$, $3^6\cdot7$ or $3^6\cdot13$. It is a same contradiction as above.

Suppose $G/N\cong \PSU(3,q))$ with $4<q\equiv-1\ (\bmod3)$. If $q$ is odd, then $\frac{1}{3}q^3(q^2-q+1)=3^6\cdot7\cdot13$, $3^6\cdot7$ or $3^6\cdot13$. This does not yield an integer root for $q$, it is a contradiction. If $q$ is even, then $\frac{1}{3}(q^2-q+1)(q+1)^2(q-1)=3^6\cdot7\cdot13$, $3^6\cdot7$ or $3^6\cdot13$. It is a same  contradiction as above.

\item[(g)] $|\cod(G/N)|=10$. We note that $\cod(\M_{22})\not\subseteq \cod(\G_{2}(3))$ since $2^7\cdot3\cdot5\cdot11 \notin \cod(G)$.

\item[(h)] $|\cod(G/N)|=11$. We have $G/N \cong \PSL(4,2),$ $G/N \cong \M_{12},$ $G/N \cong \M_{23}$ or $G/N \cong {}^{2}\G_{2}(q)$ for $q = 3^{2f+1}, f \geq 1$.

The codegree sets of $\PSL(4,2), \M_{12}$ and $\M_{23}$ contain $3^2 \cdot 5 \cdot 7$, $2^2 \cdot 3 \cdot 5 \cdot 11$ and $3^2 \cdot5\cdot11 \cdot 23$ respectively, none of which are elements of the codegree set of $\G_{2}(3)$.

If $G/N \cong {}^2\G_2(3^{2f+1})$, then by comparing the largest $3$-parts, this does not yield an integer root for $q$. It is a contradiction.

\item[(i)] $|\cod(G/N)|=12$. Then $G/N\cong \Sy_{4}(4)$. We note that $3^2\cdot5^2\cdot17\notin\cod(G)$. Thus $\cod(\Sy_{4}(4))\not\subseteq\cod(G)$.

\item[(j)] $|\cod(G/N)|=13$. Then $\cod(G/N)=\cod(\U_{4}(2))$ or $\cod(\Sy_{4}(q))$, for $q=2^f>4$.

 $\cod(G/N)=\cod(\U_{4}(2))$, we have that $\cod(G/N)\not\subseteq\cod(G)$ since $3^4\cdot5 \notin \cod(G)$.

 $\cod(G/N)=\cod(\Sy_{4}(q))$, for $q=2^f>4$. Again by setting an equality between the nontrivial odd codegrees of each set, we obtain no integer solution in these three cases, a contradiction.

\item[(k)] $|\cod(G/N)|=14$. Then $\cod(G/N)=\cod(\U_{4}(3))$, ${^{2}\F_{4}(2) {}}'$ or $\cod(\J_{3})$.

It is easily checked that neither $3^6\cdot7$, $3^3\cdot5^2\cdot13$ nor $3^5\cdot5\cdot17$ is in $\cod(G)$. Thus $\cod(\U_{4}(3))\not\subseteq\cod(G)$, $\cod({^{2}\F_{4}(2) {}}')\not\subseteq\cod(G)$  nor $\cod(\J_{3})\not\subseteq \cod(G)$. This is a contradiction.

\item[(l)] $|\cod(G/N)|=15$. We have that $ G/N \cong \G_{2}(3)$ by an easy observation.
\end{proof}

\begin{lem}\label{redA9}
Let $G$ be a finite group with $\cod(G) = \cod(\A_9).$ If $N$ is the maximal normal subgroup of $G,$ it we have $G/N \cong \A_9.$
\end{lem}

\begin{proof}
As $G/N$ is a non-abelian simple group, $\cod(G/N) \subseteq \cod(G)$, so we have $\cod(G/N)=4, 5, 6, 7, 8, 9, 10, 11, 12, 13, 14, 15$ or $16$.

We will show if $|\cod(G/N)|\leq 16$ and if $G/N \not \cong \A_9$, then $\cod(G/N) \not \subseteq \cod(G)$. Assume on the contrary that there exists a maximal normal subgroup $N$ such that $|\cod(G/N)|\leq16$, and $G/N \not \cong \A_9$ and $\cod(G/N) \subseteq \cod(G)$. We analyze each case separately.

\item[(a)] $|\cod(G/N)| = 4$. By Lemma \ref{list}, $G/N \cong \PSL(2,k)$ where $k = 2^f \geq 4$ so $k^2 - 1 \in \cod(G/N)$. For all even $k$, we have $k^2 - 1 \equiv 1 \pmod{2}$. However, there is no nontrivial odd codegree in $\cod(\A_9)$, contradicting $\cod(G/N) \subseteq \cod(G)$.

\item[(b)] $|\cod(G/N)| = 5$. By Lemma \ref{list}, $G/N \cong \PSL(2,k)$ where $k > 5$ is the power of an odd prime number. This implies that $\displaystyle k(k-\epsilon(k), \frac{k(k-\epsilon(k)))}{2} \in \cod(G/N),$ where $k = (-1)^{(k-1)/2}.$ We have that $\frac{k(k\pm1)}{2}=2^5\cdot3^3\cdot5$ or $\frac{k(k\pm1)}{2}=2^3\cdot3^4\cdot5$, since these are half of the other two elements. Contradiction since $k\in \Z $.

\item[(c)] $|\cod(G/N)| = 6$. By Lemma \ref{list}, $G/N \cong \PSL(3,4)$ or $G/N \cong {}^2B_2(q),$ with $q = 2^{2f+1}$.

In the first case, notice that $3^2 \cdot5\cdot7 \in \cod(G/N)$ is the only nontrivial odd element, but there are no nontrivial odd elements in $\cod(\A_9)$, contradicting $\cod(G/N) \subseteq \cod(G)$. \\
In the second case, notice that $(q-1)(q^2+1) \in \cod({}^2B_2(q))$. As $q$ is even, $(q-1)(q^2+1) \equiv 1 \pmod{2}$. However, there are no nontrivial odd elements in $\cod(\A_9)$, contradicting $\cod(G/N) \subseteq \cod(G)$.

\item[(d)] $|\cod(G/N)| = 7$. By Lemma \ref{list}, we have $G/N \cong \PSL(3,3), \A_7, \M_{11}$ or $\J_1$. Notice that $3^3 \cdot 13 \in \cod(\PSL(3,3))$, $2^2 \cdot 3 \cdot 5 \cdot 7 \in \cod(\A_7)$, $2^3 \cdot 3^2 \cdot 11 \in \cod(\M_{11})$ and $3 \cdot 5 \cdot 11 \cdot 19 \in \cod(\J_1)$. By inspection, we can see that none of these elements are contained in $\cod(\A_9)$. Thus, $\cod(G/N) \not \subseteq \cod(G)$, a contradiction.

\item[(e)] $|\cod(G/N)| = 8$. By Lemma \ref{list}, we have $G/N \cong \PSU(3,q)$ with $4 < q \not \equiv -1 \pmod{3}$, $G/N \cong \PSL(3,q)$ with $4 < q \not \equiv 1 \pmod{3}$ or $G/N \cong \G_2(2)'$. Observe that the only nontrivial element not divisible by $3$ in $\cod(\A_9)$ is $2^5 \cdot 5 \cdot 7$.

In the first case, notice that $(q^2-q+1)(q+1)^2(q-1), q^3(q^2-q+1) \in \cod(\PSU(3,q))$. If $q \equiv 1 \mod 3$, then $q^3(q^2-q+1) \not \equiv 0 \pmod{3}$, so it must hold that $q^3(q^2-q+1) = 2^5 \cdot 5 \cdot 7$. However, this does not yield an integer root for $q$. Alternatively, if $q \equiv 0 \pmod{3}$, then $(q^2-q+1)(q+1)^2(q-1) \not \equiv 0 \pmod{3}$, so it must hold that $(q^2-q+1)(q+1)^2(q-1) = 2^5 \cdot 5 \cdot 7$. This also does not yield an integer root for $q$. Thus, $\cod(G/N) \not \subseteq \cod(G) =\cod(\A_9)$, a contradiction.

In the second case, notice that $q^3(q^2+q+1) \in \cod(\PSL(3,q)).$ As $q \not \equiv 1 \pmod{3}, q^3(q^2+q+1) \not \equiv 0 \pmod{3}$, so we must have $q^3(q^2+q+1) = 2^5 \cdot 5 \cdot 7$. However, this does not an integer solution for $q$. Thus, $\cod(G/N) \not \subseteq \cod(G) =\cod(\A_9)$, a contradiction.

In the final case,  $3^3 \cdot 7 \in \cod(\G_2(2)')$ but $3^3 \cdot 7 \not \in \cod(\A_9)$, contradicting $\cod(G/N) \subseteq \cod(G) =\cod(\A_9)$.

\item[(f)] $|\cod(G/N)| = 9$. By Lemma \ref{list}, we have $G/N \cong \PSL(3,q)$ with $4 < q \equiv 1 \pmod{3}$ or $G/N \cong \PSU(3,q)$ with $4 < q \equiv -1 \pmod{3}$.

In the first case, notice that $\frac{1}{3}q^3(q^2+q+1) \in \cod(\PSL(3,q))$. It can be verified through a complete search that $q^2+q+1 \not \equiv 0 \pmod{9}$. Thus, $\frac{1}{3}q^3(q^2+q+1) \not \equiv 0 \pmod{3}$, so we must have $\frac{1}{3}q^3(q^2+q+1) = 2^5 \cdot 5 \cdot 7$. However, this yields no integer roots for $q$, contradicting $\cod(G/N) \subseteq \cod(G) =\cod(\A_9)$.

In the second case, notice that $\frac{1}{3}q^3(q^2-q+1) \in \cod(\PSU(3,q))$. It can be verified through a complete search that $q^2-q+1 \not \equiv 0 \pmod{9}$, so $\frac{1}{3}q^3(q^2-q+1) \not \equiv 0 \pmod{3}$, so we must have $\frac{1}{3} q^3(q^2-q+1) = 2^5 \cdot 5 \cdot 7$. However, there are no integer roots for $q$, contradicting $\cod(G/N) \subseteq \cod(G) =\cod(\A_9)$.

\item[(g)] $|\cod(G/N)| = 10$. By Lemma \ref{list}, we have $G/N \cong \M_{22}$. However, $2^7 \cdot 3 \cdot 5 \cdot 11$ is only present in $\cod(\M_{22})$ and not in $\cod(\A_9)$, contradicting $\cod(G/N) \subseteq \cod(G)$.

\item[(h)] $|\cod(G/N)| = 11$. By Lemma \ref{list}, we have $G/N \cong \PSL(4,2)$, $G/N \cong \M_{12}$, $G/N \cong \M_{23}$ or $G/N \cong {}^{2}\G_{2}(q)$ for $q = 3^{2f+1}, f \geq 1$. The codegree sets of $\PSL(4,2), \M_{12}$ and $\M_{23}$ contain $3^2 \cdot 5 \cdot 7$, $2^2 \cdot 3 \cdot 5 \cdot 11$ and $3^2 \cdot5\cdot11 \cdot 23$ respectively, none of which are elements of the codegree set of $\A_{9}$.

Notice additionally that $(q^2-1)(q^2-q+1) \in \cod({}^{2}G_{2}(q))$. However, $(q^2-1)(q^2-q+1) \not \equiv 0 \pmod{3}$ when $q$ is a multiple of three. This implies that we must have $(q^2-1)(q^2-q+1) = 2^5 \cdot 5 \cdot 7$, yet this does not yield an integer root for $q$. All of these cases contradicts $\cod(G/N) \subseteq \cod(G) =\cod(\A_9)$.

\item[(i)] $|\cod(G/N)| = 12$. By Lemma \ref{list}, we have $G/N \cong \Sy_4(4)$. However, the element $3^2 \cdot 5^2 \cdot 17$ is the only nontrivial odd codegree in the codegree set of $\Sy_4(4)$, it is not in the codegree set of $\A_9$,  contradicting $\cod(G/N) \subseteq \cod(G)$.

\item[(j)] $|\cod(G/N)| = 13$. By Lemma \ref{list}, we have $G/N \cong \U_4(2)$ or $G/N \cong \Sy_4(q)$ for $q = 2^f > 4$.

In the first case, the element $3^4 \cdot 5$ is only in the codegree set of $\U_4(2)$ and not the codegree set of $\A_9$.

In the second case, $2q^3(q-1)^2(q+1)^2 \in \cod(\Sy_4(q))$. However, if $\nu_2(q) > 2$, then $\nu_2(2q^3(q-1)^2(q+1)^2) > 7$, yet there does not exist an element in $\cod(\A_9)$ that is divisible by $2^7$. Both these cases contradict $\cod(G/N) \subseteq \cod(G)$.

\item[(k)] $|\cod(G/N)| = 14$. By Lemma \ref{list}. If $G/N \cong \U_4(3)$. We have that $ 3^6\cdot5$ and $3^6\cdot7$ are both the nontrivial odd elements in $\cod(\U_4(3))$, there are all the nontrivial even elements in $\cod (\A_9)$, a contradiction. If $G/N \cong {}^2\F_4(2)'$ or $G/N \cong \J_3$. the codegree sets of $\U_4(3), {}^2\F_4(2)'$ and $\J_3$ contain $3^3 \cdot5^2\cdot13$, $3^4 \cdot 17 \cdot 19$ respectively, none of which are elements of $\cod(\A_9)$. This contradicts $\cod(G/N) \subseteq \cod(G)$.

\item[(l)] $|\cod(G/N)| = 15$. By Lemma \ref{list}, we have $G/N \cong \G_2(3)$. However, there are three nontrivial odd elements in the codegree set of $\G_2(3)$ and not in the codegree set of $\A_9$. This contradicts $\cod(G/N) \subseteq \cod(G)$.

\item[(m)] $|\cod(G/N)| = 16$. By Lemma \ref{list}, $G/N \cong \A_9$ or $G/N \cong \J_2$.

Notice that $2^6 \cdot 3^3 \cdot 5^2 \in \cod(\J_2)$ but this element is not in $\cod(\A_9)$. This contradicts $\cod(G/N) \subseteq \cod(G)$.

Having ruled out all other cases for $G/N \equiv G'$ with $\cod(G') \leq 16$, we conclude that $G/N \cong \A_9$.
\end{proof}

\begin{lem}\label{redJ2}
Let $G$ be a finite group with $\cod(G) = \cod(\J_2)$. If $N$ is the maximal normal subgroup of $G$, it we have $G/N \cong \J_2$.
\end{lem}

\begin{proof}
As $G/N$ is a non-abelian simple group, $\cod(G/N) \subseteq \cod(G)$, so we have that $\cod(G/N)=4, 5, 6, 7, 8, 9, 10, 11, 12, 13, 14, 15$ or $16$.

Assume on the contrary that there exist a maximal normal subgroup $N$ such that $G/N \not \cong \J_2, \cod(G/N) \leq 16$ and $\cod(G/N) \subseteq \cod(G)$. We exhaust all possibilities to show the previous statement is false.

\item[(a)] $|\cod(G/N)| = 4$. By Lemma \ref{list}, $G/N \cong \PSL(2,k)$ where $k = 2^f \geq 4,$ so $k^2-1 \in \cod(G/N)$. For all even $k$, $k^2-1$ is odd, yet there are no nontrivial odd elements in the codegree set of $\J_2$. This contradicts $\cod(G/N) \subseteq \cod(G)$.

\item[(b)] Suppose $|\cod(G/N)| = 5$. By Lemma \ref{list}, $G/N \cong \PSL(2,k)$ where $k$ is a power of an odd prime number. This implies that $\displaystyle k(k-\epsilon(k), \frac{k(k-\epsilon(k)))}{2} \in \cod(G/N)$, where $k = (-1)^{(k-1)/2}$.
 Then $k(k-\epsilon(k))/2$, $k(k-\epsilon(k))\in \cod(G)$, and we have that $\frac{k(k\pm1)}{2}=2^6\cdot3\cdot 5^2$ since it is half of the other element. This is a contradiction since  $k\in \Z$.

\item[(c)] Suppose $|\cod(G/N)| = 6$. By Lemma \ref{list}, $G/N \cong \PSL(3,4)$ or $G/N \cong {}^2B_2(q)$, with $q = 2^{2n+1}$.

If $G/N \cong \PSL(3,4)$ then $3^2 \cdot 5 \cdot 7 \in \cod(G/N)$ is odd, but $3^2 \cdot 5 \cdot 7 \not\in \cod(G) = \cod(\J_2)$, contradicting $\cod(G/N) \subseteq \cod(G)$.

In the second case, notice that $(q-1)(q^2+1) \in \cod({}^2B_2(q))$. As $q$ is even, $(q-1)(q^2+1) \equiv 1 \pmod{2}$. However, there are no nontrivial odd elements in $\cod(\J_2)$, contradicting $\cod(G/N) \subseteq \cod(G)$.

\item[(d)] Suppose $|\cod(G/N)| = 7$. By Lemma \ref{list}, we have $G/N \cong \PSL(3,3), \A_7, \M_{11}$ or $\J_1$. We may see that $3^3 \cdot 13 \in \cod(\PSL(3,3))$, $2^2 \cdot 3 \cdot 5 \cdot 7 \in \cod(\A_7)$, $2^3 \cdot 3^2 \cdot 11 \in \cod(\M_{11})$ and $3 \cdot 5 \cdot 11 \cdot 19 \in \cod(\J_1)$. However, none of these elements are elements of $\cod(\J_2)$. Thus $\cod(G/N) \not \subseteq \cod(G)$, a contradiction.

\item[(e)] Suppose  $|\cod(G/N)| = 8$. By Lemma \ref{list}, we have $G/N \cong \PSU(3,q)$ with $4 < q \not \equiv -1 \pmod{3}$, $G/N \cong \PSL(3,q)$ with $4 < q \not \equiv 1 \pmod{3}$ or $G/N \cong \G_2(2)'$. Observe that the only nontrivial element not divisible by $3$ in $\cod(\J_2)$ is $2^7 \cdot 5^2$.

If $G/N \cong \PSU(3,q)$, notice that $(q^2-q+1)(q+1)^2(q-1), q^3(q^2-q+1) \in \cod(\PSU(3,q))$. If $q \equiv 1 \mod 3,$ then $q^3(q^2-q+1) \not \equiv 0 \pmod{3}$, so it must hold that $q^3(q^2-q+1) = 2^7\cdot5^2$. However, this does not yield an integer root for $q$. Alternatively, if $q \equiv 0 \pmod{3}$, then $(q^2-q+1)(q+1)^2(q-1) \not \equiv 0 \pmod{3}$ so it must hold that $(q^2-q+1)(q+1)^2(q-1) = 2^7 \cdot 5^2$. This also does not yield an integer root for $q$. Thus $\cod(G/N) \not \subseteq \cod(G) =\cod(\J_2)$, a contradiction.

If $G/N \cong \PSL(3,q)$, we have $q^3(q^2+q+1) \in \cod(\PSL(3,q))$. As $q \not \equiv 1 \pmod{3}, q^3(q^2+q+1) \not \equiv 0 \pmod{3}$ so it must hold that $q^3(q^2+q+1) = 2^7 \cdot 5^2$. However, this does not lead to an integer solution for $q$. Thus $\cod(G/N) \not \subseteq \cod(G) =\cod(\J_2)$, a contradiction.

In the final case,  $3^3 \cdot 7 \in \cod(\G_2(2)')$ but $3^3 \cdot 7 \not \in \cod(\J_2)$, contradicting $\cod(G/N) \subseteq \cod(G) =\cod(\J_2)$.

\item[(f)] $|\cod(G/N)| = 9$. By Lemma \ref{list}, we have $G/N \cong \PSL(3,q)$ with $4 < q \equiv 1 \pmod{3}$ or $G/N \cong \PSU(3,q)$ with $4 < q \equiv -1 \pmod{3}$.

Observe that by a complete search of remainders modulo $9$ that both $q^2 - q + 1$ and $q^2 + q + 1$ cannot be divisible by $9$.

In the first scenario, we have $\frac{1}{3}q^3(q^2+q+1) \in \cod(\PSL(3,q))$. From the above remark, notice that $\frac{1}{3}q^3(q^2+q+1) \not \equiv 0 \pmod{3}$. Therefore, we must have $\frac{1}{3}q^3(q^2+q+1) = 2^7 \cdot 5^2$, which does not yield an integer root for $q$, a contradiction.

In the second scenario, we have  $\frac{1}{3}q^3(q^2-q+1) \in \cod(\PSU(3,q))$. As with the first scenario, this element is not divisible by $3$, so we must have $\frac{1}{3}q^3(q^2-q+1) = 2^7 \cdot 5^2$. However, this does not yield an integer root for $q$, a contradiction.

\item[(g)] $|\cod(G/N)| = 10$. By Lemma \ref{list}, we have $G/N \cong \M_{22}$. However, $2^7 \cdot 3 \cdot 5 \cdot 11$ is only present in $\cod(\M_{22})$ and not in $\cod(\J_2)$, contradicting $\cod(G/N) \subseteq \cod(G)$.

\item[(h)] $|\cod(G/N)| = 11$. By Lemma \ref{list}, we have $G/N \cong \PSL(4,2)$, $G/N \cong \M_{12}$, $G/N \cong \M_{23}$ or $G/N \cong {}^{2}\G_{2}(q)$ for $q = 3^{2f+1}, f \geq 1$.

The codegree sets of $\PSL(4,2)$, $\M_{12}$ and $\M_{23}$ contain $3^2 \cdot 5 \cdot 7$, $2^2 \cdot 3 \cdot 5 \cdot 11$ and $2^7 \cdot 7 \cdot 11$ respectively, none of which are elements of the codegree set of $\J_2$.

If $G/N \cong {}^{2}\G_{2}(q)$, $(q^2-1)(q^2-q+1) \in \cod({}^{2}\G_{2}(q))$. However, $(q^2-1)(q^2-q+1) \not \equiv 0 \pmod{3}$ in this case, so we must have $(q^2-1)(q^2-q+1) = 2^5 \cdot 5 \cdot 7$, yet this does not yield an integer root for $q$. This is a contradiction to $\cod(G/N) \subseteq \cod(G) =\cod(\J_2)$.

\item[(i)] $|\cod(G/N)| = 12$. By Lemma \ref{list}, we have $G/N \cong \Sy_4(4)$. However, the element $2^7 \cdot 5^2 \cdot 17$ is only in the codegree set of $\Sy_4(4)$ and not in the codegree set of $\J_2$ ,contradicting $\cod(G/N) \subseteq \cod(G)$.

\item[(j)] $|\cod(G/N)| = 13$. By Lemma \ref{list}, we have $G/N \cong \U_4(2)$ or $G/N \cong \Sy_4(q)$ for $q = 2^f > 4$.

In the first case, the element $3^4 \cdot 5 \in \cod(\U_4(2))$ but it is not in $\cod(\J_2)$.

In the second case, $2q^3(q-1)^2(q+1)^2 \in \cod(\Sy_4(q))$. However, as $\nu_2(q) > 2$, then $\nu_2(2q^3(q-1)^2(q+1)^2) > 7$ yet there does not exist an element in $\cod(\A_9)$ that is divisible by $2^8$.
In both of these cases, $\cod(G/N) \not \subseteq \cod(G)$.

\item[(k)] $|\cod(G/N)| = 14$. By Lemma \ref{list}, we have $G/N \cong \U_4(3)$, $G/N \cong {}^2\F_4(2)'$ or $G/N \cong \J_3$. However, the codegree sets of $\U_4(3), {}^2\F_4(2)'$ and $\J_3$ contain $2^7 \cdot 3^5 \cdot 5$, $2^{10} \cdot 3^3 \cdot 5^2$ and $2^7 \cdot 3^5 \cdot 19$ respectively, none of which are elements of $\cod(\J_2)$. This contradicts $\cod(G/N) \subseteq \cod(G)$.

\item[(l)] $|\cod(G/N)| = 15$. By Lemma \ref{list}, we have $G/N \cong \G_2(3)$. However, the element $2^5 \cdot 3^6 \cdot 13$ is only present in the codegree set of $\G_2(3)$ and not in the codegree set of $\J_2$, which contradicts $\cod(G/N) \subseteq \cod(G)$.

\item[(m)] $|\cod(G/N)| = 16$. By Lemma \ref{list}, we have $G/N \cong \A_9$ or $G/N\cong \J_2$. $2^4 \cdot 3^3 \cdot 5 \in \cod(\A_9)$ but is not in $\cod(\J_2)$, contradicting the assumption.

Having ruled out all other cases, we conclude $G/N \cong \J_2$.
\end{proof}

\begin{lem}\label{redMcL}
Let $G$ be a finite group with $\cod(G) = \cod(\M^c L)$. If $N$ is the maximal normal subgroup of $G$, it we have $G/N \cong \M^c L$.
\end{lem}

\begin{proof}
As $G/N$ is a non-abelian simple group, $\cod(G/N) \subseteq \cod(G)$, so we have that $\cod(G/N)=4, 5, 6, 7, 8, 9, 10, 11, 12, 13, 14, 15$, or $16$.

Assume on the contrary that there exist a maximal normal subgroup $N$ such that $G/N \not \cong \M^c L, \cod(G/N) \leq 17$ and $\cod(G/N) \subseteq \cod(G)$. We exhaust all possibilities to show the previous statement is false.

\item [(a)] $|\cod(G/N)| = 4$. By Lemma \ref{list}, $G/N \cong \PSL(2,k)$ where $k = 2^f \geq 4$, so $k^2-1 \in \cod(G/N)$. For all even $k$, $k^2-1$ is odd. Therefore, $k^2-1$ must equal $3^6 \cdot 5^3$ or $3^6 \cdot 5^3 \cdot 11$ but this yields no integer roots for $k$, contradicting $\cod(G/N) \subseteq \cod(G)$.

\item[(b)] $|\cod(G/N)| = 5$. By Lemma \ref{list}, $G/N \cong \PSL(2,k)$ where $k$ is a power of an odd prime number. This implies that $\displaystyle k(k-\epsilon(k), \frac{k(k-\epsilon(k)))}{2} \in \cod(G/N)$, where $k = (-1)^{(k-1)/2}.$
However, there does not exist two elements $a, b \in \cod(\M^c L)$ with $a = 2b,$ contradicting $\cod(G/N) \subseteq \cod(G)$.

\item[(c)] $|\cod(G/N)| = 6$. By Lemma \ref{list}, $G/N \cong \PSL(3,4)$ or $G/N \cong {}^2B_2(q)$ with $q = 2^{2n+1}$.

If $G/N \cong \PSL(3,4)$, then $3^2 \cdot 5 \cdot 7 \in \cod(G/N)$ but $3^2 \cdot 5 \cdot 7 \not\in \cod(G) = \cod(\M^c L)$, contradicting $\cod(G/N) \subseteq \cod(G)$.

Furthermore, notice that $(q-1)(q^2+1) \in \cod({}^2B_2(q))$. As $q$ is even, $(q-1)(q^2+1) \equiv 1 \pmod{2}$. This means that $(q-1)(q^2+1) = 3^6 \cdot 5^3$ or $3^6 \cdot 5^3 \cdot 11$, which does not yield any roots. This contradicts $\cod(G/N) \subseteq \cod(G)$.

\item[(d)] $|\cod(G/N)| = 7$. This implies that $G/N \cong \PSL(3,3), \M_{11}, \A_7$ or $\J_1$. By inspection, the elements $2^2 \cdot 3^2 \cdot 13$, $2^3 \cdot 3^3 \cdot 11$, $2^2 \cdot 3 \cdot 5 \cdot 7$, $3 \cdot 5 \cdot 11 \cdot 19$ are in the codegrees of $\PSL(3,3)$, $\M_{11}$, $\A_7$ or $\J_1$, respectively. However, none of these elements are in the codegree set of $\M^c L$, a contradiction.

\item[(e)] $|\cod(G/N)| = 8$. By Lemma \ref{list}, we have $G/N \cong \PSU(3,q)$ with $4 < q \not \equiv -1 \pmod{3}$, $G/N \cong \PSL(3,q)$ with $4 < q \not \equiv 1 \pmod{3}$ or $G/N \cong \G_2(2)'$. Observe that the only nontrivial element not divisible by $3$ in $\cod(\M^c L)$ is $2^7 \cdot 5^3 \cdot 11$, $2^7 \cdot 5^3 \cdot 7$.

If $G/N \cong \PSU(3,q)$, notice that $(q^2-q+1)(q+1)^2(q-1), q^3(q^2-q+1) \in \cod(\PSU(3,q))$. If $q \equiv 1 \mod 3$, then $q^3(q^2-q+1) \not \equiv 0 \pmod{3}$, so it must hold that $q^3(q^2-q+1) = 2^7\cdot5^3 \cdot 11$ or $2^7 \cdot 5^3 \cdot 7$. However, this does not yield an integer root for $q$. Alternatively, if $q \equiv 0 \pmod{3}$, then $(q^2-q+1)(q+1)^2(q-1) \not \equiv 0 \pmod{3}$, so it must hold that $(q^2-q+1)(q+1)^2(q-1) = 2^7 \cdot 5^3 \cdot 11$ or $2^7 \cdot 5^3 \cdot 7$. This also does not yield an integer root for $q$. Thus $\cod(G/N) \not \subseteq \cod(G)$, a contradiction.

If $G/N \cong \PSL(3,q)$, we have $q^3(q^2+q+1) \in \cod(\PSL(3,q))$. As $q \not \equiv 1 \pmod{3}, q^3(q^2+q+1) \not \equiv 0 \pmod{3}$, so it must hold that $q^3(q^2+q+1) = 2^7 \cdot 5^3 \cdot 11$ or $q^3(q^2+q+1) = 2^7 \cdot 5^3 \cdot 7$. However, this does not lead to an integer solution for $q$. Thus $\cod(G/N) \not \subseteq \cod(G),$, a contradiction.

In the final case, $3^3 \cdot 7 \in \cod(\G_2(2)')$ but $3^3 \cdot 7 \not \in \cod(\M^c L)$, contradicting $\cod(G/N) \subseteq \cod(G)$.
    %TODO: rewrite this so it's not as verbatim

\item[(f)] $|\cod(G/N)| = 9$. By Lemma \ref{list}, we have $G/N \cong \PSL(3,q)$ with $4 < q \equiv 1 \pmod{3}$ or $G/N \cong \PSU(3,q)$ with $4 < q \equiv -1 \pmod{3}$. Both $q^2 - q + 1$ and $q^2 + q + 1$ are not divisible by $9$.

In the first case, we have $\frac{1}{3}q^3(q^2+q+1) \in \cod(\PSL(3,q))$. Thus, $\frac{1}{3}q^3(q^2+q+1) \not \equiv 0 \pmod{3}$, so we must have $\frac{1}{3}q^3(q^2+q+1) = 2^7 \cdot 5^3 \cdot 11$ or $2^7 \cdot 5^3 \cdot 7$, which does not yield an integer root for $q$, a contradiction.

In the second scenario, we have  $\frac{1}{3}q^3(q^2-q+1) \in \cod(\PSU(3,q))$. As with the first scenario, this element is not divisible by $3$, so we must have $\frac{1}{3}q^3(q^2-q+1) = 2^7 \cdot 5^3 \cdot 11$ or $2^7 \cdot 5^3 \cdot 7$. However, this does not yield an integer root for $q$, a contradiction.

\item[(g)] $|\cod(G/N)| = 10$. By Lemma \ref{list}, we have $G/N \cong \M_{22}$. However, $2^7 \cdot 3 \cdot 5 \cdot 11$ is only present in $\cod(\M_{22})$ and not in $\cod(\M^c L)$, contradicting $\cod(G/N) \subseteq \cod(G)$.

\item[(h)] $|\cod(G/N)| = 11$. By Lemma \ref{list}, we have $G/N \cong \PSL(4,2)$, $G/N \cong \M_{12}$, $G/N \cong \M_{23}$ or $G/N \cong {}^{2}\G_{2}(q)$ for $q = 3^{2f+1}, f \geq 1$. The codegree sets of $\PSL(4,2), \M_{12}$ and $\M_{23}$ contain $3^2 \cdot 5 \cdot 7$, $2^2 \cdot 3 \cdot 5 \cdot 11$ and $2^7 \cdot 7 \cdot 11$ respectively, none of which are elements of the codegree set of $\M^c L$.

In the other case, note that $(q^2-1)(q^2-q+1) \in \cod({}^{2}\G_{2}(q))$. However, $(q^2-1)(q^2-q+1) \not \equiv 0 \pmod{3}$ in this case, so we must have $(q^2-1)(q^2-q+1) = 2^7 \cdot 5^3 \cdot 7$, $2^7 \cdot 5^3 \cdot 11$ yet this does not yield an integer root for $q$. This is a contradiction to $\cod(G/N) \subseteq \cod(G)$.

\item[(i)] $|\cod(G/N)| = 12$. By Lemma \ref{list}, we have $G/N \cong \Sy_4(4)$. However, the element $3^2 \cdot 5^2 \cdot 17$ is in the codegree set of $\Sy_4(4)$ but not in the codegree set of $\M^c L$, contradicting $\cod(G/N) \subseteq \cod(G)$.

\item[(j)] $|\cod(G/N)| = 13$. By Lemma \ref{list}, we have $G/N \cong \U_4(2)$ or $G/N \cong \Sy_4(q)$ for $q = 2^f > 4$.
In the first case, the element $3^4 \cdot 5 \in \cod(\U_4(2))$ but it is not in $\cod(\M^c L)$.

In the second case, $2q^3(q-1)^2(q+1)^2 \in \cod(\Sy_4(q))$. However, as $\nu_2(q) > 2$, then $\nu_2(2q^3(q-1)^2(q+1)^2) > 7$ yet there does not exist an element in $\cod(\M^c L)$ that is divisible by $2^8$, a contradiction.

\item[(k)] $|\cod(G/N)| = 14$. By Lemma \ref{list}, we have $G/N \cong \U_4(3)$, $G/N \cong {}^2\F_4(2)'$ or $G/N \cong \J_3$. However, the codegree sets of $\U_4(3)$, ${}^2\F_4(2)'$ and $\J_3$ contain $2^7 \cdot 3^5 \cdot 5$, $2^{10} \cdot 3^3 \cdot 5^2$ and $2^7 \cdot 3^5 \cdot 19$ respectively, but these are not in $\cod(\M^c L)$. This contradicts $\cod(G/N) \subseteq \cod(G)$.

\item[(l)] $|\cod(G/N)| = 15$. By Lemma \ref{list}, we have $G/N \cong \G_2(3)$. However, the element $2^5 \cdot 3^6 \cdot 13$ is only present in the codegree set of $\G_2(3)$ and not in the codegree set of $\M^c L$, which contradicts $\cod(G/N) \subseteq \cod(G)$.

\item[(m)] $|\cod(G/N)| = 16$. By Lemma \ref{list}, we have $G/N \cong \A_9$ or $G/N \cong \J_2$. In the first scenario, $2^4 \cdot 3^3 \cdot 5 \in \cod(\A_9)$ but is not in $\cod(\M^c L)$, while in the second $2^6 \cdot 3^3 \cdot 5^2 \in \cod(\J_2)$, but this element is not in $\cod(\M^c L)$. This contradicts $\cod(G/N) \subseteq \cod(G)$.

\item[(n)] $|\cod(G/N)| = 17$. Then $G/N \cong \PSL(4,3)$. However, the element $2^6 \cdot 3^6 \cdot 5$ is only in $\cod(\PSL(4,3))$ and not the codegree of $\M^c L$, a contradiction.

Having ruled out all other cases, we conclude $G/N \cong \M^c L$.
\end{proof}

\begin{lem}\label{redL43}
Let $G$ be a finite group with $\cod(G) = \cod(\PSL(4,3))$. If $N$ is the maximal normal subgroup of $G$, it we have $G/N \cong \PSL(4,3)$.
\end{lem}

\begin{proof}
 As $G/N$ is a non-abelian simple group, $\cod(G/N) \subseteq \cod(G)$, so we have that $\cod(G/N)=4, 5, 6, 7, 8, 9, 10, 11, 12, 13, 14, 15, 16$, or $17$.

Assume on the contrary that there exist a maximal normal subgroup $N$ such that $G/N \not \cong \PSL(4,3), \cod(G/N) \leq 17$ and $\cod(G/N) \subseteq \cod(G)$. We exhaust all possibilities to show the previous statement is false.

\item [(a)] $|\cod(G/N)| = 4$. By Lemma \ref{list}, $G/N \cong \PSL(2,k)$ where $k = 2^f \geq 4$, so $k^2-1 \in \cod(G/N)$. For all even $k$, $k^2-1$ is odd. Therefore, $k^2-1$ must equal $3^6 \cdot 13$ but this yields no integer roots for $k$, contradicting $\cod(G/N) \subseteq \cod(G)$.

\item[(b)] $|\cod(G/N)| = 5$. By Lemma \ref{list}, $G/N \cong \PSL(2,k)$ where $k$ is a power of an odd prime number. This implies that $\displaystyle k(k-\epsilon(k), \frac{k(k-\epsilon(k)))}{2} \in \cod(G/N)$, where $k = (-1)^{(k-1)/2}$.
Then $k(k-\epsilon(k))/2$, $k(k-\epsilon(k))\in \cod(G)$. By a complete search, and we have that $\frac{k(k\pm1)}{2}=2^5\cdot3^6\cdot 5$, $2^5\cdot3^5$ or $2^5\cdot3^4\cdot5$ since they are half of the other three elements. This is a contradiction since $k\in \Z$.

\item[(c)] $|\cod(G/N)| = 6$. By Lemma \ref{list}, $G/N \cong \PSL(3,4)$ or $G/N \cong {}^2B_2(q)$, with $q = 2^{2n+1}$.

If $G/N \cong \PSL(3,4)$ then $3^2 \cdot 5 \cdot 7 \in \cod(\PSL(3,4))$ but $3^2 \cdot 5 \cdot 7 \not\in \cod(G) = \cod(\PSL(4,3))$, contradicting $\cod(G/N) \subseteq \cod(G)$.

Furthermore, notice that $(q-1)(q^2+1) \in \cod({}^2B_2(q))$. As $q$ is even, $(q-1)(q^2+1) \equiv 1 \pmod{2}$. This means that $(q-1)(q^2+1) = 3^6 \cdot 13$ which does not yield any roots. This contradicts $\cod(G/N) \subseteq \cod(G)$.

\item[(d)] $|\cod(G/N)| = 7$. By Lemma \ref{list}, we have $G/N \cong \PSL(3,3)$, $\M_{11}$, $\A_7$ or $\J_1$. By inspection, the elements $2^2 \cdot 3^2 \cdot 13$, $2^3 \cdot 3^3 \cdot 11$, $2^2 \cdot 3 \cdot 5 \cdot 7$, $3 \cdot 5 \cdot 11 \cdot 19$ are in the codegrees of $\PSL(3,3)$, $\M_{11}$, $\A_7$ or $\J_1$ respectively. However, none of these elements are in the codegree set of $\PSL(4,3)$, a contradiction.

\item[(e)] $|\cod(G/N)| = 8$. By Lemma \ref{list}, we have $G/N \cong \PSU(3,q)$ with $4 < q \not \equiv -1 \pmod{3}$, $G/N \cong \PSL(3,q)$ with $4 < q \not \equiv 1 \pmod{3}$ or $G/N \cong \G_2(2)'$. There are no nontrivial elements in the codegree set of $\PSL(4,3)$ that are not divisible by $3$ except $2^7\cdot5\cdot13$.

If $G/N \cong \PSU(3,q)$, notice that $(q^2-q+1)(q+1)^2(q-1), q^3(q^2-q+1) \in \cod(\PSU(3,q))$. If $q \equiv 1 \mod 3$, in this case, we must have $q^3(q^2-q+1)$ yet this is the nontrivial odd codegree, this does not yield an integer root for $q$. This is a contradiction to $\cod(G/N) \subseteq \cod(G)$. Alternatively, if $q \equiv 0 \pmod{3}$, then $(q^2-q+1)(q+1)^2(q-1) \not \equiv 0 \pmod{3}$. We have that $(q^2-q+1)(q+1)^2(q-1)=2^7\cdot5\cdot13$, this does not yield an integer root for $q$. Thus $\cod(G/N) \not \subseteq \cod(G)$, a contradiction.

If $G/N \cong \PSL(3,q)$, we have $q^3(q^2+q+1) \in \cod(\PSL(3,q))$. As $q \not \equiv 1 \pmod{3}, q^3(q^2+q+1) \not \equiv 0 \pmod{3}$. We can also get the same contradiction as above. Thus $\cod(G/N) \not \subseteq \cod(G)$.

In the final case, $3^3 \cdot 7 \in \cod(\G_2(2)')$ but $3^3 \cdot 7 \not \in \cod(\PSL(4,3))$, contradicting $\cod(G/N) \subseteq \cod(G)$.

\item[(f)] $|\cod(G/N)| = 9$. By Lemma \ref{list}, we have $G/N \cong \PSL(3,q)$ with $4 < q \equiv 1 \pmod{3}$ or $G/N \cong \PSU(3,q)$ with $4 < q \equiv -1 \pmod{3}$. Both $q^2 - q + 1$ and $q^2 + q + 1$ are not divisible by $9$.

In the first case, we have $\frac{1}{3}q^3(q^2+q+1) \in \cod(\PSL(3,q))$. Thus $\frac{1}{3}q^3(q^2+q+1) \not \equiv 0 \pmod{3}$, we must have $\frac{1}{3}q^3(q^2+q+1)$ is the odd codegree, yet this does not yield an integer root for $q$, contradiction.

In the second scenario, we have  $\frac{1}{3}q^3(q^2-q+1) \in \cod(\PSU(3,q))$. This is not divisible by $3$, we can get the same contradiction as above.

\item[(g)] $|\cod(G/N)| = 10$. By Lemma \ref{list}, we have $G/N \cong \M_{22}.$ However, $2^7 \cdot 3 \cdot 5 \cdot 11$ is only present in $\cod(\M_{22})$ and not in $\cod(\PSL(4,3))$, contradicting $\cod(G/N) \subseteq \cod(G)$.

\item[(h)] $|\cod(G/N)| = 11$. By Lemma \ref{list}, we have $G/N \cong \PSL(4,2)$, $G/N \cong \M_{12}$, $G/N \cong \M_{23}$ or $G/N \cong {}^{2}\G_{2}(q)$ for $q = 3^{2f+1}, f \geq 1$. The codegree sets of $\PSL(4,2)$, $\M_{12}$ and $\M_{23}$ contain $3^2 \cdot 5 \cdot 7$, $2^2 \cdot 3 \cdot 5 \cdot 11$ and $2^7 \cdot 7 \cdot 11$ respectively, none of which are elements of the codegree set of $\PSL(4,3)$.

If $G/N \cong {}^{2}\G_{2}(q)$,  $(q^2-1)(q^2-q+1) \in \cod({}^{2}\G_{2}(q))$. However, $(q^2-1)(q^2-q+1) \not \equiv 0 \pmod{3}$ in this case, so we must have $(q^2-1)(q^2-q+1) = 2^7 \cdot 5^3 \cdot 7, 2^7 \cdot 5^3 \cdot 11$ yet this does not yield an integer root for $q$. This is a contradiction to $\cod(G/N) \subseteq \cod(G)$.

\item[(i)] $|\cod(G/N)| = 12$. By Lemma \ref{list}, we have $G/N \cong \Sy_4(4)$. However, $17$ is a factor of $|\cod(\Sy_4(4))|$ and not $|\cod(\PSL(4,3))|$, contradicting $\cod(G/N) \subseteq \cod(G)$.

\item[(j)] $|\cod(G/N)| = 13$. By Lemma \ref{list}, we have $G/N \cong \U_4(2)$ or $G/N \cong \Sy_4(q)$ for $q = 2^f > 4$.

In the first case, the element $3^4 \cdot 5 \in \cod(\U_4(2))$ but it is not in $\cod(\PSL(4,3))$.

In the second case, $2q^3(q-1)^2(q+1)^2 \in \cod(\Sy_4(q))$. However, as $\nu_2(q) > 2$, then $\nu_2(2q^3(q-1)^2(q+1)^2) > 7$ yet there does not exist an element in $\cod(\M^c L)$ that is divisible by $2^8$, a contradiction.

\item[(k)] $|\cod(G/N)| = 14$. By Lemma \ref{list}, we have $G/N \cong \U_4(3)$, $G/N \cong {}^2\F_4(2)'$ or $G/N \cong \J_3$. However, the codegree sets of $\U_4(3)$, ${}^2\F_4(2)'$ and $\J_3$ have $2^4 \cdot 3^6$, $2^{10} \cdot 3^3 \cdot 5^2$ and $2^7 \cdot 3^5 \cdot 19$ respectively, but these are not in $\cod(\PSL(4,3))$. This contradicts $\cod(G/N) \subseteq \cod(G)$.

\item[(l)] $|\cod(G/N)| = 15$. By Lemma \ref{list}, we have $G/N \cong \G_2(3)$. However, the element $2^5 \cdot 3^6 \cdot 13$ is only present in the codegree set of $\G_2(3)$ and not in the codegree set of $\PSL(4,3)$, which contradicts $\cod(G/N) \subseteq \cod(G)$.

\item[(m)] $|\cod(G/N)| = 16$. By Lemma \ref{list}, we have $G/N \cong \A_9$ or $G/N \cong \J_2$. In the $\A_9$ case, $2^4 \cdot 3^3 \cdot 5 \in \cod(\A_9)$ but is not in $\cod(\PSL(4,3))$, while in the $\J_2$ case $2^6 \cdot 3^3 \cdot 5^2 \in \cod(\J_2)$ but this element is not in $\cod(\PSL(4,3))$. This contradicts $\cod(G/N) \subseteq \cod(G)$.

\item[(n)] $|\cod(G/N)| = 17$. Then $G/N \cong \M^c L$. However, the element $2^7 \cdot 5^3 \cdot 11$ is only in $\cod(\M^c L)$ and not the codegree of $\PSL(4,3)$, a contradiction.

Having ruled out all other cases, we conclude $G/N \cong \PSL(4,3)$.
\end{proof}

\begin{lem} \label{redS45}
Let $G$ be a finite group with $\cod(G) = \cod(\Sy_4(5))$. If $N$ is the maximal normal subgroup of $G$, it we have $G/N \cong \Sy_4(5)$.
\end{lem}

\begin{proof}
As $G/N$ is a non-abelian simple group, $\cod(G/N) \subseteq \cod(G)$, so we have that $\cod(G/N)=4, 5, 6, 7, 8, 9, 10, 11, 12, 13, 14, 15, 16, 17$, or $18$.

Assume on the contrary that there exist a maximal normal subgroup $N$ such that $G/N \not \cong \Sy_4(5), \cod(G/N) \leq 18$ and $\cod(G/N) \subseteq \cod(G)$. We exhaust all possibilities.

\item[(a)] Suppose $|\cod(G/N)| = 4$. By Lemma \ref{list}, $G/N \cong \PSL(2,k)$ where $k = 2^f \geq 4$, so $k^2-1 \in \cod(G/N)$. For all even $k$, $k^2-1$ is odd. $k^2-1$ must equal $5^4 \cdot13$ but this yields no integer roots for $k$, contradicting $\cod(G/N) \subseteq \cod(G)$.

\item[(b)] Suppose $|\cod(G/N)| = 5$. By Lemma \ref{list}, $G/N \cong \PSL(2,k)$ where $k$ is a power of an odd prime number. This implies that $\displaystyle k(k-\epsilon(k), \frac{k(k-\epsilon(k)))}{2} \in \cod(G/N)$, where $k = (-1)^{(k-1)/2}$. By a complete search, and we have that $\frac{k(k\pm1)}{2}=2^4\cdot3\cdot 5^4$, $2^4\cdot3\cdot5^3$ or $2^2\cdot3^2\cdot5^4$ since they are half of the other three elements. This is a contradiction since $k\in \Z$.

\item[(c)] $|\cod(G/N)|=6$. By Lemma \ref{list}, $G/N \cong \PSL(3,4)$ or $G/N \cong {}^2B_2(q)$, with $q = 2^{2n+1}$.

If $G/N \cong \PSL(3,4)$ then $3^2 \cdot 5 \cdot 7 \in \cod(\PSL(3,4))$, but $3^2 \cdot 5 \cdot 7 \not\in \cod(\Sy_4(5))$, contradicting $\cod(G/N) \subseteq \cod(G)$.
Furthermore, notice that $(q-1)(q^2+1) \in \cod({}^2B_2(q))$. As $q$ is even, $(q-1)(q^2+1) \equiv 1 \pmod{2}$. This means that $(q-1)(q^2+1) = 5^4 \cdot 13$ which does not yield any roots, contradicting $\cod(G/N) \subseteq \cod(G)$.

\item[(d)] $|\cod(G/N)| = 7$. By Lemma \ref{list}, we have $G/N \cong \PSL(3,3)$, $\M_{11}$, $\A_7$ or $\J_1$. The elements $2^2 \cdot 3^2 \cdot 13$, $2^3 \cdot 3^3 \cdot 11$, $2^2 \cdot 3 \cdot 5 \cdot 7$, $3 \cdot 5 \cdot 11 \cdot 19$ are in the codegrees of $\PSL(3,3)$, $\M_{11}$, $\A_7$ or $\J_1$ respectively. However, none of these elements are in the codegree set of $\Sy_4(5)$, a contradiction.

\item[(e)] $|\cod(G/N)| = 8$. By Lemma \ref{list}, we have $G/N \cong \PSU(3,q)$ with $4 < q \not \equiv -1 \pmod{3}$, $G/N \cong \PSL(3,q)$ with $4 < q \not \equiv 1 \pmod{3}$ or $G/N \cong \G_2(2)'$.

The only nontrivial element of $\Sy_4(5)$ not divisible by $3$ is $5^4 \cdot 13$.

If $G/N \cong \PSU(3,q)$, notice that $(q^2-q+1)(q+1)^2(q-1), q^3(q^2-q+1) \in \cod(\PSU(3,q))$. If $q \equiv 1 \mod 3$, then $q^3(q^2-q+1) \not \equiv 0 \pmod{3}$. Alternatively, if $q \equiv 0 \pmod{3}$, then $(q^2-q+1)(q+1)^2(q-1) \not \equiv 0 \pmod{3}$. However,$(q^2-q+1)(q+1)^2(q-1)=5^4 \cdot 13$ yield no integer roots for $q$. Thus, $\cod(G/N) \not \subseteq \cod(G)$, a contradiction.

If $G/N \cong \PSL(3,q)$, we have $q^3(q^2+q+1) \in \cod(\PSL(3,q))$. As $q \not \equiv 1 \pmod{3}, q^3(q^2+q+1) \not \equiv 0 \pmod{3}$. However, $q^3(q^2+q+1) = 5^4 \cdot 13$ has no integer roots. Thus, $\cod(G/N) \not \subseteq \cod(G)$, a contradiction.

In the final case, $3^3 \cdot 7 \in \cod(\G_2(2)')$ but $3^3 \cdot 7 \not \in \cod(\Sy_4(5))$, contradicting $\cod(G/N) \subseteq \cod(G)$.

\item[(f)] $|\cod(G/N)| = 9$. By Lemma \ref{list}, we have $G/N \cong \PSL(3,q)$ with $4 < q \equiv 1 \pmod{3}$ or $G/N \cong \PSU(3,q)$ with $4 < q \equiv -1 \pmod{3}$.

Both $q^2 - q + 1$ and $q^2 + q + 1$ are not divisible by $9$.

In the first case, we have $\frac{1}{3}q^3(q^2+q+1) \in \cod(\PSL(3,q))$. Thus, $\frac{1}{3}q^3(q^2+q+1) \not \equiv 0 \pmod{3}$. Furthermore, we have $\frac{1}{3}q^3(q^2-q+1) \in \cod(\PSU(3,q))$. This is not divisible by $3$.

The equation $\frac{1}{3}q^3(q^2\pm q+1) = 5^4 \cdot 13$ has no integer roots, contradicting $\cod(G/N) \subseteq \cod(G)$.

\item[(g)] $|\cod(G/N)| = 10$. By Lemma \ref{list}, we have $G/N \cong \M_{22}$. However, $2^7 \cdot 3 \cdot 5 \cdot 11$ is only present in $\cod(\M_{22})$ and not in $\cod(\Sy_4(5))$, contradicting $\cod(G/N) \subseteq \cod(G)$.

\item[(h)] $|\cod(G/N)| = 11$. By Lemma \ref{list}, we have $G/N \cong \PSL(4,2)$, $G/N \cong \M_{12}$, $G/N \cong \M_{23}$ or $G/N \cong {}^{2}\G_{2}(q)$ for $q = 3^{2f+1}, f \geq 1$. The codegree sets of $\PSL(4,2)$, $\M_{12}$ and $\M_{23}$ contain $3^2 \cdot 5 \cdot 7$, $2^2 \cdot 3 \cdot 5 \cdot 11$ and $2^7 \cdot 7 \cdot 11$, respectively, none of which are elements of the codegree set of $\Sy_4(5)$.

If $G/N \cong  {}^{2}\G_{2}(q)$, $(q^2-1)(q^2-q+1) \in \cod({}^{2}\G_{2}(q))$. However, $(q^2-1)(q^2-q+1) \not \equiv 0 \pmod{3}$. In this case, so we must have $(q^2-1)(q^2-q+1) = 5^4 \cdot 13$ yet this does not yield an integer root for $q$. This is a contradiction to $\cod(G/N) \subseteq \cod(G)$.

\item[(i)] $|\cod(G/N)| = 12$. By Lemma \ref{list}, we have $G/N \cong \Sy_4(4)$. However, $17$ is a factor of $|\cod(\Sy_4(4))|$ and not  $|\cod(\Sy_4(5))|$, contradicting $\cod(G/N) \subseteq \cod(G)$.

\item[(j)] $|\cod(G/N)| = 13$. By Lemma \ref{list}, we have $G/N \cong \U_4(2)$ or $G/N \cong \Sy_4(q)$ for $q = 2^f > 4$.

In the first case, the element $3^4 \cdot 5 \in \cod(\Sy_4(q))$ but it is not in $\cod(\Sy_4(5))$.

In the second case, $2q^3(q-1)^2(q+1)^2 \in \cod(\U_4(2))$. However, as $\nu_2(q) > 2$, then $\nu_2(2q^3(q-1)^2(q+1)^2) > 7$ yet there does not exist an element in $\cod(\M^c L)$ that is divisible by $2^8$, a contradiction.

\item[(k)] $|\cod(G/N)| = 14$. By Lemma \ref{list}, we have $G/N \cong \U_4(3)$, $G/N \cong {}^2\F_4(2)'$ or $G/N \cong \J_3$. However, the codegree sets of $\U_4(3)$, ${}^2\F_4(2)'$ and $\J_3$ have $2^4 \cdot 3^6$, $2^{10} \cdot 3^3 \cdot 5^2$ and $2^7 \cdot 3^5 \cdot 19$, respectively, but these are not in $\cod(\Sy_4(5))$. This contradicts $\cod(G/N) \subseteq \cod(G)$.

\item[(l)] $|\cod(G/N)| = 15$. By Lemma \ref{list}, we have $G/N \cong \G_2(3)$. However, the element $2^5 \cdot 3^6 \cdot 13$ is only present in the codegree set of $\G_2(3)$ and not in the codegree set of $\Sy_4(5)$, which contradicts $\cod(G/N) \subseteq \cod(G)$.

\item[(m)] $|\cod(G/N)| = 16$. By Lemma \ref{list}, we have $G/N \cong \A_9$ or $G/N \cong \J_2$.

In the $\A_9$ case, $7$ is a factor of $|\cod(\A_9)|$ and not $|\cod(\Sy_4(5))|$, contradicting $\cod(G/N) \subseteq \cod(G)$.

while in the $\J_2$ case, this situation is the same as the contradiction above. $7$ is not the factor of $|\cod(\Sy_4(5))|$.

\item[(n)] $|\cod(G/N)| = 17$. Then $G/N \cong \M^c L$ or $\PSL(4,3)$ However, there are three nontrivial odd codegrees in $\cod \M^c L$, there is only a nontrivial odd element in the set of $\Sy_4(5)$, a contradiction. Additionally, the element $3^6\cdot 13$ is only in $\cod(\PSL(4,3))$ and not $\cod(\Sy_4(5))$, a contradiction.

\item[(o)] $|\cod(G/N)| = 18$. We have $G/N \cong \G_2(4)$ or $\HS$. If $G/N \cong \G_2(4)$, then $2^{12} \cdot 3^3 \cdot 5 \cdot 7 \in \cod(G/N)$. If $G/N \in \HS$, then $11$ is not the factor of $|\cod(\Sy_4(5))|$, a contradiction.

Having ruled out all other cases, we conclude $G/N \cong \Sy_4(5)$.
\end{proof}

\begin{lem}\label{redG24}
Let $G$ be a finite group with $\cod(G) = \cod(\G_2(4))$. If $N$ is the maximal normal subgroup of $G$, it we have $G/N \cong \G_2(4)$.
\end{lem}

\begin{proof}
As $G/N$ is a non-abelian simple group, $\cod(G/N) \subseteq \cod(G)$, so we have that $\cod(G/N)=4, 5, 6, 7, 8, 9, 10, 11, 12, 13, 14, 15, 16, 17,$ or $18$.

Assume on the contrary that there exist a maximal normal subgroup $N$ such that $G/N \not \cong \G_2(4), \cod(G/N) \leq 18$ and $\cod(G/N) \subseteq \cod(G)$. We exhaust all possibilities.

\item [(a)] $|\cod(G/N)| = 4$. By Lemma \ref{list}, $G/N \cong \PSL(2,k)$ where $k = 2^f \geq 4$, so $k^2-1 \in \cod(G/N)$. For all even $k$, $k^2-1$ is odd. $k^2-1$ must equal $3^3 \cdot 5^2 \cdot 7 \cdot 13$ but this yields no integer roots for $k$, contradicting $\cod(G/N) \subseteq \cod(G)$.

\item[(b)] $|\cod(G/N)| = 5$. By Lemma \ref{list}, $G/N \cong \PSL(2,k)$ where $k$ is a power of an odd prime number. This implies that $\displaystyle k(k-\epsilon(k)\in \cod(G/N)$ ,where $k = (-1)^{(k-1)/2}$. we have that $\frac{k(k\pm1)}{2}=2^{10}\cdot3^3\cdot 7$, since it is half of $2^{11}\cdot3^3\cdot 7$. This is a contradiction since $k\in \Z$.

\item[(c)] $|\cod(G/N)| = 6$. By Lemma \ref{list}, $G/N \cong \PSL(3,4)$ or $G/N \cong {}^2B_2(q)$, with $q = 2^{2n+1}$.

If $G/N \cong \PSL(3,4)$ then $3^2 \cdot 5 \cdot 7 \in \cod(\PSL(3,4))$ but $3^2 \cdot 5 \cdot 7 \not\in \cod(\G_2(4))$, contradicting $\cod(G/N) \subseteq \cod(G)$.

If $G/N \cong {}^2B_2(q)$, notice that $(q-1)(q^2+1) \in \cod({}^2B_2(q))$. As $q$ is even, $(q-1)(q^2+1) \equiv 1 \pmod{2}$. This means that $(q-1)(q^2+1) = 3^3 \cdot 5^2 \cdot 7 \cdot 13$ which does not yield any roots, contradicting $\cod(G/N) \subseteq \cod(G)$.

\item[(c)] $|\cod(G/N)| = 7$. By Lemma \ref{list}, we have $G/N \cong \PSL(3,3)$, $\M_{11}$, $\A_7$ or $\J_1$. The elements $2^2 \cdot 3^2 \cdot 13$, $2^3 \cdot 3^3 \cdot 11$, $2^2 \cdot 3 \cdot 5 \cdot 7$, $3 \cdot 5 \cdot 11 \cdot 19$ are in the codegrees of $\PSL(3,3), \M_{11}, \A_7$ or $\J_1$, respectively. However, none of these elements are in $\cod(\G_2(4))$, a contradiction.

\item[(d)] $|\cod(G/N)| = 8$. By Lemma \ref{list}, we have $G/N \cong \PSU(3,q)$ with $4 < q \not \equiv -1 \pmod{3}$, $G/N \cong \PSL(3,q)$ with $4 < q \not \equiv 1 \pmod{3}$ or $G/N \cong \G_2(2)'$.

The only nontrivial elements of $\G_2(4)$ not divisible by $3$ are $2^{11} \cdot 5^2 \cdot 13$ and $2^{12} \cdot 13$.

If $G/N \cong \PSU(3,q)$, then $(q^2-q+1)(q+1)^2(q-1), q^3(q^2-q+1) \in \cod(\PSU(3,q))$. If $q \equiv 1 \mod 3,$ then $q^3(q^2-q+1) \not \equiv 0 \pmod{3}$. Alternatively, if $q \equiv 0 \pmod{3}$, then $(q^2-q+1)(q+1)^2(q-1) \not \equiv 0 \pmod{3}$. However, setting either to be equivalent to $2^{11} \cdot 5^2 \cdot 13$ or $2^{12} \cdot 13$ yield no integer roots for $q$. Thus, $\cod(G/N) \not \subseteq \cod(G)$, a contradiction.

If $G/N \cong \PSL(3,q)$, we have $q^3(q^2+q+1) \in \cod(\PSL(3,q))$. As $q \not \equiv 1 \pmod{3}, q^3(q^2+q+1) \not \equiv 0 \pmod{3}$. However, $q^3(q^2+q+1) = 2^{11} \cdot 5^2 \cdot 13$ or $2^{12} \cdot 13$ has no integer roots. Thus, $\cod(G/N) \not \subseteq \cod(G)$, a contradiction.

In the final case, $3^3 \cdot 7 \in \cod(\G_2(2)')$ but $3^3 \cdot 7 \not \in \cod(\G_2(4))$, contradicting $\cod(G/N) \subseteq \cod(G)$.

\item[(d)] $|\cod(G/N)| = 9$. By Lemma \ref{list}, we have $G/N \cong \PSL(3,q)$ with $4 < q \equiv 1 \pmod{3}$ or $G/N \cong \PSU(3,q)$ with $4 < q \equiv -1 \pmod{3}$.

Both $q^2 - q + 1$ and $q^2 + q + 1$ are not divisible by $9$.

In the first case, we have $\frac{1}{3}q^3(q^2+q+1) \in \cod(\PSL(3,q))$. Thus, $\frac{1}{3}q^3(q^2+q+1) \not \equiv 0 \pmod{3}$. Furthermore, we have $\frac{1}{3}q^3(q^2-q+1) \in \cod(\PSU(3,q))$. This is not divisible by $3$. The equations $\frac{1}{3}q^3(q^2\pm q+1) = 2^{11} \cdot 5^2 \cdot 13$ and $\frac{1}{3}q^3(q^2\pm q+1)=2^{12} \cdot 13$ have no integer roots, contradicting $\cod(G/N) \subseteq \cod(G)$.

\item[(g)] $|\cod(G/N)| = 10$. By Lemma \ref{list}, we have $G/N \cong \M_{22}$. However, $2^7 \cdot 3 \cdot 5 \cdot 11$ is only present in $\cod(\M_{22})$ and not in $\cod(\G_2(4))$, contradicting $\cod(G/N) \subseteq \cod(G)$.

\item[(h)] $|\cod(G/N)| = 11$. By Lemma \ref{list}, we have $G/N \cong \PSL(4,2)$, $G/N \cong \M_{12}$, $G/N \cong \M_{23}$ or $G/N \cong {}^{2}\G_{2}(q)$ for $q = 3^{2f+1}, f \geq 1$.

The codegree sets of $\PSL(4,2)$, $\M_{12}$ and $\M_{23}$ contain $3^2 \cdot 5 \cdot 7$, $2^2 \cdot 3 \cdot 5 \cdot 11$ and $2^7 \cdot 7 \cdot 11$, respectively, none of which are elements of the codegree set of $\G_2(4)$.

Furthermore, $(q^2-1)(q^2-q+1) \in \cod({}^{2}\G_{2}(q))$. However, $(q^2-1)(q^2-q+1) \not \equiv 0 \pmod{3}$ in this case, so we must have $(q^2-1)(q^2-q+1) = 2^{11} \cdot 5^2 \cdot 13$ or $2^{12} \cdot 13$, yet this does not yield an integer root for $q$, contradicting $\cod(G/N) \subseteq \cod(G)$.

\item[(i)] $|\cod(G/N)| = 12$. By Lemma \ref{list}, we have $G/N \cong \Sy_4(4)$. However, $17$ is a factor of $|\cod(\Sy_4(4))|$ and not  $|\cod(\G_2(4)|$, contradicting $\cod(G/N) \subseteq \cod(G)$.

\item[(j)] $|\cod(G/N)| = 13$. By Lemma \ref{list}, we have $G/N \cong \U_4(2)$ or $G/N \cong \Sy_4(q)$ for $q = 2^f > 4$.

In the first case, the element $3^4 \cdot 5 \in \cod(\U_4(2))$ but it is not in $\cod(\G_2(4))$.

In the second case, $q^4(q-1)^2 \in \cod(\Sy_4(q))$. However, by a complete search there is no integer $q$ such that this element becomes equal to an element in $\cod(\G_2(4))$, a contradiction.

\item[(k)] $|\cod(G/N)| = 14$. By Lemma \ref{list}, we have $G/N \cong \U_4(3)$, $G/N \cong {}^2\F_4(2)'$ or $G/N \cong \J_3$.

However, the codegree sets of $\U_4(3)$, ${}^2\F_4(2)'$ and $\J_3$ have $2^4 \cdot 3^6$, $3^3 \cdot 5^2 \cdot 13$ and $2^7 \cdot 3^5 \cdot 19$, respectively, but these are not in $\cod(\G_2(4))$. This contradicts $\cod(G/N) \subseteq \cod(G)$.

\item[(l)] $|\cod(G/N)| = 15$. By Lemma \ref{list}, we have $G/N \cong \G_2(3)$. However, the element $2^5 \cdot 3^6 \cdot 13$ is only present in the codegree set of $\G_2(3)$ and not in the codegree set of $\G_2(4)$, which contradicts $\cod(G/N) \subseteq \cod(G)$.

\item[(m)] $|\cod(G/N)| = 16$. By Lemma \ref{list}, we have $G/N \cong \A_9$ or $G/N \cong \J_2$. In the $\A_9$ case, $2^4 \cdot 3^3 \cdot 5 \in \cod(\A_9)$ but is not in $\cod(\G_2(4))$, while in the $\J_2$ case $2^6 \cdot 3^3 \cdot 5^2 \in \cod(\J_2)$ but this element is not in $\cod(\G_2(4))$. This contradicts $\cod(G/N) \subseteq \cod(G)$.

\item[(n)] $|\cod(G/N)| = 17$. Then $G/N \cong \M^c L$ or $\PSL(4,3)$.

However, there are three nontrivial odd elements in $\cod(\M^c L)$ and only a nontrivial odd codegree of $\cod(\G_2(4))$, a contradiction.

If $G/N = \PSL(4,3)$, the element $3^6\cdot 13 \in \cod(\PSL(4,3))$ but it is not in $\cod(\G_2(4))$, a contradiction.

\item[(o)] $|\cod(G/N)| = 18$. We have $G/N \cong \G_2(4)$, $\Sy_4(5)$ or $\HS$. If $G/N \cong \Sy_4(5)$, then $5^4 \cdot 13 \in \cod(G/N)$. If $G/N \in \HS$, then $2^2 \cdot 3^2 \cdot 5 \cdot 7 \cdot 11 \in \cod(G/N)$. Neither is in $\cod(\G_2(4))$, a contradiction.
Having ruled out all other cases, we conclude $G/N \cong \G_2(4)$.
\end{proof}

\begin{lem}\label{redHS}
Let $G$ be a finite group with $\cod(G) = \cod(\HS)$. If $N$ is the maximal normal subgroup of $G$, we have $G/N \cong \HS$.
\end{lem}

\begin{proof}
As $G/N$ is a non-abelian simple group, $\cod(G/N) \subseteq \cod(G)$, so we have that $\cod(G/N)=4, 5, 6, 7, 8, 9, 10, 11, 12, 13, 14, 15, 16, 17, 18$, or $19$.

\item [(a)] $|\cod(G/N)| = 4$. By Lemma \ref{list}, $G/N \cong \PSL(2,k)$ where $k = 2^f \geq 4$, so $k^2-1 \in \cod(G/N)$. For all even $k$, $k^2-1$ is odd. $k^2-1$ is the only nontrivial odd codegree in $\cod(G/N)$, but None of the elements in the $\cod(\HS)$ is odd. Contradicting $\cod(G/N) \subseteq \cod(G)$.

\item[(b)] $|\cod(G/N)| = 5$. By Lemma \ref{list}, $G/N \cong \PSL(2,k)$ where $k$ is a power of an odd prime number. This implies that $\displaystyle k(k-\epsilon(k)\in \cod(G/N)$, where $k = (-1)^{(k-1)/2}$. By a complete search, $k(k + 1)$ cannot be equal to $2^8\cdot3^3\cdot5^3$ or $2^8\cdot5^3$ in $\cod(\HS)$ since it is half of another nontrivial codegrees, a contradiction since $k\in \Z$.

\item[(c)] $|\cod(G/N)| = 6$. By Lemma \ref{list}, $G/N \cong \PSL(3,4)$ or $G/N \cong {}^2B_2(q)$, with $q = 2^{2n+1}$.

If $G/N \cong \PSL(3,4)$, then $3^2 \cdot 5 \cdot 7 \in \cod(\PSL(3,4))$ but $3^2 \cdot 5 \cdot 7 \not\in \cod(\HS)$, contradicting $\cod(G/N) \subseteq \cod(G)$.

If $G/N \cong {}^2B_2(q)$, notice that $(q-1)(q^2+1) \in \cod({}^2B_2(q))$. As $q$ is even, $(q-1)(q^2+1) \equiv 1 \pmod{2}$. This means that $(q-1)(q^2+1) \notin \cod(HS)$ since none of these are nontrivial codegrees. Contradicting $\cod(G/N) \subseteq \cod(G)$.

\item[(d)] $|\cod(G/N)| = 7$. By Lemma \ref{list}, we have $G/N \cong \PSL(3,3)$, $\M_{11}$, $\A_7$ or $\J_1$. The elements $2^2 \cdot 3^2 \cdot 13$, $2^3 \cdot 3^3 \cdot 11$, $2^2 \cdot 3 \cdot 5 \cdot 7$, $3 \cdot 5 \cdot 11 \cdot 19$ are in the codegrees of $\PSL(3,3)$, $\M_{11}$, $\A_7$ or $\J_1$, respectively. However, none of these elements are in $\cod(\G_2(4))$, a contradiction.

\item[(e)] $|\cod(G/N)| = 8$. By Lemma \ref{list}, we have $G/N \cong \PSU(3,q)$ with $4 < q \not \equiv -1 \pmod{3}$, $G/N \cong \PSL(3,q)$ with $4 < q \not \equiv 1 \pmod{3}$ or $G/N \cong \G_2(2)'$. There are three nontrivial elements of $\G_2(4)$ not divisible by $3$ are $2^9 \cdot 5^3$, $2^6\cdot5^2\cdot11$ and $2^8 \cdot 5^3$.

If $G/N \cong \PSU(3,q)$, then $(q^2-q+1)(q+1)^2(q-1), q^3(q^2-q+1) \in \cod(\PSU(3,q))$. If $q \equiv 1 \mod 3,$ then $q^3(q^2-q+1) \not \equiv 0 \pmod{3}$. Alternatively, if $q \equiv 0 \pmod{3}$, then $(q^2-q+1)(q+1)^2(q-1) \not \equiv 0 \pmod{3}$. However, setting either to be equivalent to $2^9 \cdot 5^3$, $2^6\cdot5^2\cdot11$ or $2^8 \cdot 5^3$ yield no integer roots for $q$. Thus, $\cod(G/N) \not \subseteq \cod(G)$, a contradiction.

If $G/N \cong \PSL(3,q)$, we have $q^3(q^2+q+1) \in \cod(\PSL(3,q))$. As $q \not \equiv 1 \pmod{3}, q^3(q^2+q+1) \not \equiv 0 \pmod{3}$. However, $q^3(q^2+q+1) = 2^9 \cdot 5^3$, $2^6\cdot5^2\cdot11$ or $2^8 \cdot 5^3$ has no integer roots. Thus, $\cod(G/N) \not \subseteq \cod(G)$, a contradiction.

In the final case, $3^3 \cdot 7 \in \cod(\G_2(2)')$ but $3^3 \cdot 7 \not \in \cod(\HS)$, contradicting $\cod(G/N) \subseteq \cod(G)$.

\item[(f)] $|\cod(G/N)| = 9$. By Lemma \ref{list}, we have $G/N \cong \PSL(3,q)$ with $4 < q \equiv 1 \pmod{3}$ or $G/N \cong \PSU(3,q)$ with $4 < q \equiv -1 \pmod{3}$. Both $q^2 - q + 1$ and $q^2 + q + 1$ are not divisible by $9$.

In the first case, we have $\frac{1}{3}q^3(q^2+q+1) \in \cod(\PSL(3,q))$. Thus, $\frac{1}{3}q^3(q^2+q+1) \not \equiv 0 \pmod{3}$. Furthermore, we have $\frac{1}{3}q^3(q^2-q+1) \in \cod(\PSU(3,q))$. This is not divisible by $3$.

The equations $\frac{1}{3}q^3(q^2\pm q+1) = 2^9 \cdot 5^3$, $2^6\cdot5^2\cdot11$ or $2^8 \cdot 5^3$, and $\frac{1}{3}q^3(q^2\pm q+1)=2^9 \cdot 5^3$, $2^6\cdot5^2\cdot11$ or $2^8 \cdot 5^3$ have no integer roots, contradicting $\cod(G/N) \subseteq \cod(G)$.

\item[(g)] $|\cod(G/N)| = 10$. By Lemma \ref{list}, we have $G/N \cong \M_{22}$. However, $2^7 \cdot 3 \cdot 5 \cdot 11$ is only present in $\cod(\M_{22})$ and not in $\cod(\HS))$, contradicting $\cod(G/N) \subseteq \cod(G)$.

\item[(h)] $|\cod(G/N)| = 11$. Lemma \ref{list}, we have $G/N \cong \PSL(4,2)$, $G/N \cong \M_{12}$, $G/N \cong \M_{23}$ or $G/N \cong {}^{2}\G_{2}(q)$ for $q = 3^{2f+1}, f \geq 1$.

$\cod(\PSL(4,2))$, $\cod(\M_{12})$ and $\cod(\M_{23})$ contain $3^2 \cdot 5 \cdot 7$, $2^2 \cdot 3 \cdot 5 \cdot 11$ and $2^7 \cdot 7 \cdot 11,$ respectively, none of which are elements of the codegree set of $\HS$.

Furthermore, $(q^2-1)(q^2-q+1) \in \cod({}^{2}\G_{2}(q))$. However, $(q^2-1)(q^2-q+1) \not \equiv 0 \pmod{3}$ in this case, so we must have $(q^2-1)(q^2-q+1) = 2^8 \cdot 5^3$, $2^6\cdot5^2\cdot11$ or $2^9 \cdot 5^3$, yet this does not yield an integer root for $q$, contradicting $\cod(G/N) \subseteq \cod(G)$.

\item[(i)] $|\cod(G/N)| = 12$. By Lemma \ref{list}, we have $G/N \cong \Sy_4(4)$. However, $17$ is a factor of $|\cod(\Sy_4(4))|$ and not $|\cod(\HS)|$, contradicting $\cod(G/N) \subseteq \cod(G)$.

\item[(j)] $|\cod(G/N)| = 13$. By Lemma \ref{list}, we have $G/N \cong \U_4(2)$ or $G/N \cong \Sy_4(q)$ for $q = 2^f > 4$.

In the first case, the element $3^4 \cdot 5 \in \cod(\U_4(2))$ but it is not in $\cod(\HS).$

In the second case, $q^4(q-1)^2 \in \cod(\Sy_4(q))$. However, by a complete search there is no integer $q$ such that $q^4(q-1)^2$ equals an element in $\cod(\HS)$, a contradiction.

\item[(k)] $|\cod(G/N)| = 14$. By Lemma \ref{list}, we have $G/N \cong \U_4(3)$, $G/N \cong {}^2\F_4(2)'$ or $G/N \cong \J_3$. However, the codegree sets of $\U_4(3)$, ${}^2\F_4(2)'$ and $\J_3$ have $2^4 \cdot 3^6$, $3^3 \cdot 5^2 \cdot 13$ and $2^7 \cdot 3^5 \cdot 19$, respectively, but these are not in $\cod(\HS)$. This contradicts $\cod(G/N) \subseteq \cod(G)$.

\item[(l)] $|\cod(G/N)| = 15$. By Lemma \ref{list}, we have $G/N \cong \G_2(3)$. However, the element $2^5 \cdot 3^6 \cdot 13$ is only present in the codegree set of $\G_2(3)$ and not in the codegree set of $\HS$, which contradicts $\cod(G/N) \subseteq \cod(G)$.

\item[(m)] $|\cod(G/N)| = 16$. By Lemma \ref{list}, we have $G/N \cong \A_9$ or $G/N \cong \J_2$. In the $\A_9$ case, $2^4 \cdot 3^3 \cdot 5 \in \cod(\A_9)$ but is not in $\cod(\HS)$, while in the $\J_2$ case $2^6 \cdot 3^3 \cdot 5^2 \in \cod(\J_2)$ but $2^6 \cdot 3^3 \cdot 5^2 \not \in \cod(\HS)$. This contradicts $\cod(G/N) \subseteq \cod(G)$.

\item[(n)] $|\cod(G/N)| = 17$. Then $G/N \cong \M^c L$ or $\PSL(4,3)$.

Notice that the element $2^7 \cdot 5^3 \cdot 11$ is only in $\cod(\M^c L)$ and not the codegree of $\cod(\HS)$, a contradiction.

If $G/N = \PSL(4,3)$, the element $2^3 \cdot 3^6 \in \cod(\PSL(4,3))$ but it is not in $\cod(\HS)$.\\

\item[(o)] $|\cod(G/N)| = 18$. We have $G/N \cong \G_2(4)$ or $\Sy_4(5)$. If $G/N \cong \Sy_4(5)$, then $5^4 \cdot 13 \in \cod(G/N)$. If $G/N \in \G_2(4)$, then $2^{12} \cdot 3^3 \cdot 5 \cdot 7 \in \cod(G/N)$. Neither is in $\cod(\HS)$, a contradiction.

Having ruled out all other cases, we conclude $G/N \cong \HS$.
\end{proof}

\begin{lem}\label{redON}
Let $G$ be a finite group with $\cod(G) = \cod(\ON)$. If $N$ is the maximal normal subgroup of $G$, it we have $G/N \cong \ON$.
\end{lem}

\begin{proof}
As $G/N$ is a non-abelian simple group, $\cod(G/N) \subseteq \cod(G)$, so we have that $\cod(G/N)=4, 5, 6, 7, 8, 9, 10, 11, 12, 13, 14, 15, 16, 17$ or $18$.

\item [(a)] $|\cod(G/N)| = 4$. By Lemma \ref{list}, $G/N \cong \PSL(2,k)$ where $k = 2^f \geq 4$. The element $k^2-1 \in \cod(G/N)$. For all even $k$, $k^2-1$ is odd. $k^2-1$ must equal $7^3 \cdot 11 \cdot 19 \cdot 31$ or $3^4 \cdot 5 \cdot 11 \cdot 19 \cdot 31$ but this yields no integer roots for $k$, contradicting $\cod(G/N) \subseteq \cod(G)$.

\item[(b)] $|\cod(G/N)| = 5$. By Lemma \ref{list}, $G/N \cong \PSL(2,k)$ where $k$ is a power of an odd prime number. We have $\displaystyle k(k-\epsilon(k)\in \cod(G/N)$, where $k = (-1)^{(k-1)/2}$.
By a complete search, $k(k + 1)$ cannot be equal to any of the nontrivial elements in $\cod(\ON)$ for any integer $k$, a contradiction.

\item[(c)] $|\cod(G/N)| = 6$. By Lemma \ref{list}, $G/N \cong \PSL(3,4)$ or $G/N \cong {}^2B_2(q)$, with $q = 2^{2n+1}$.

If $G/N \cong \PSL(3,4)$ then $3^2 \cdot 5 \cdot 7 \in \cod(\PSL(3,4))$ but $3^2 \cdot 5 \cdot 7 \not\in \cod(\ON)$, contradicting $\cod(G/N) \subseteq \cod(G)$.

If $G/N \cong {}^2B_2(q)$, notice that $(q-1)(q^2+1) \in \cod({}^2B_2(q))$. As $q$ is even, $(q-1)(q^2+1) \equiv 1 \pmod{2}$. This means that $(q-1)(q^2+1) = 7^3\cdot11\cdot19\cdot31$ or $3^4\cdot5\cdot11\cdot19\cdot31$ which does not yield any roots, contradicting $\cod(G/N) \subseteq \cod(G)$.

\item[(d)] $|\cod(G/N)| = 7$. By Lemma \ref{list}, we have $G/N \cong \PSL(3,3)$, $\M_{11}$, $\A_7$ or $\J_1$.

The elements $2^2 \cdot 3^2 \cdot 13$, $2^3 \cdot 3^3 \cdot 11$, $2^2 \cdot 3 \cdot 5 \cdot 7$, $3 \cdot 5 \cdot 11 \cdot 19$ are in the codegrees of $\PSL(3,3)$, $\M_{11}$, $\A_7$ or $\J_1$, respectively. However, none of these are elements of $\cod(\ON)$, a contradiction.

\item[(e)] $|\cod(G/N)| = 8$. By Lemma \ref{list}, we have $G/N \cong \PSU(3,q)$ with $4 < q \not \equiv -1 \pmod{3}$, $G/N \cong \PSL(3,q)$ with $4 < q \not \equiv 1 \pmod{3}$ or $G/N \cong \G_2(2)'$.

If $G/N \cong \PSU(3,q)$, then $q^3(q^2-q+1) \in \cod(\PSU(3,q))$. By a complete search, $q^3(q^2-q+1)$ does not equal an element of $\cod(\ON)$ for any integer $q$.

If $G/N \cong \PSL(3,q)$, we have $q^3(q^2+q+1) \in \cod(\PSL(3,q))$. By a complete search, $q^3(q^2+q+1)$ does not equal an element of $\cod(\ON)$ for any integer $q$. Thus, $\cod(G/N) \not \subseteq \cod(G)$, a contradiction.

In the final case, $3^3 \cdot 7 \in \cod(\G_2(2)')$ but $3^3 \cdot 7 \not \in \cod(\ON)$, contradicting $\cod(G/N) \subseteq \cod(G)$.

\item[(f)] $|\cod(G/N)| = 9$. By Lemma \ref{list}, we have $G/N \cong \PSL(3,q)$ with $4 < q \equiv 1 \pmod{3}$ or $G/N \cong \PSU(3,q)$ with $4 < q \equiv -1 \pmod{3}$.

In the first case, we have $\frac{1}{3}q^3(q^2+q+1) \in \cod(\PSL(3,q))$. Furthermore, we have $\frac{1}{3}q^3(q^2-q+1) \in \cod(\PSU(3,q))$. Regardless, by a complete search, these expressions do not equal an element of $\cod(\ON)$ for any integer $q$, contradicting $\cod(G/N) \subseteq \cod(G)$.

\item[(g)] $|\cod(G/N)| = 10$. By Lemma \ref{list}, we have $G/N \cong \M_{22}$. However, $2^7 \cdot 3 \cdot 5 \cdot 11$ is only present in $\cod(\M_{22})$ and not in $\cod(\ON)$, contradicting $\cod(G/N) \subseteq \cod(G)$.

\item[(h)] $|\cod(G/N)| = 11$. Lemma \ref{list}, we have $G/N \cong \PSL(4,2)$, $G/N \cong \M_{12}$, $G/N \cong \M_{23}$ or $G/N \cong {}^{2}\G_{2}(q)$ for $q = 3^{2f+1}, f \geq 1$.

$\cod(\PSL(4,2))$, $\cod(\M_{12})$ and $\cod(\M_{23})$ contain $3^2 \cdot 5 \cdot 7$, $2^2 \cdot 3 \cdot 5 \cdot 11$ and $2^7 \cdot 7 \cdot 11$, respectively, none of which are elements of the codegree set of $\ON$.

Furthermore, $(q^2-1)(q^2-q+1) \in \cod({}^{2}\G_{2}(q))$. By a complete search, this element does not equal an element of $\cod(\ON)$ for any integer root for $q$, contradicting $\cod(G/N) \subseteq \cod(G)$.

\item[(i)] $|\cod(G/N)| = 12$. By Lemma \ref{list}, we have $G/N \cong \Sy_4(4)$. However, $2^7 \cdot 5^2 \cdot 17$ is in of $\cod(\Sy_4(4))$ and not in $\cod(\ON)$, contradicting $\cod(G/N) \subseteq \cod(G)$.

\item[(j)] $|\cod(G/N)| = 13$. By Lemma \ref{list}, we have $G/N \cong \U_4(2)$ or $G/N \cong \Sy_4(q)$ for $q = 2^f > 4$.

In the first case, the element $3^4 \cdot 5 \in \cod(\U_4(2))$ but it is not in $\cod(\ON)$.

In the second case, $q^4(q-1)^2 \in \cod(\Sy_4(q))$. However, by a complete search there is no integer $q$ such that $q^4(q-1)^2$ equals an element in $\cod(\ON)$, a contradiction.

\item[(k)] $|\cod(G/N)| = 14$. By Lemma \ref{list}, we have $G/N \cong \U_4(3)$, $G/N \cong {}^2\F_4(2)'$ or $G/N \cong \J_3$. However, the codegree sets of $\U_4(3)$, ${}^2\F_4(2)'$ and $\J_3$ have $2^4 \cdot 3^6$, $3^3 \cdot 5^2 \cdot 13$ and $2^7 \cdot 3^5 \cdot 19$, respectively, but none of these numbers are in $\cod(\ON)$. This contradicts $\cod(G/N) \subseteq \cod(G)$.

\item[(l)] $|\cod(G/N)| = 15$. By Lemma \ref{list}, we have $G/N \cong \G_2(3)$. However, the element $2^5 \cdot 3^6 \cdot 13$ is only present in the codegree set of $\G_2(3)$ and not in the codegree set of $\ON$, which contradicts $\cod(G/N) \subseteq \cod(G)$.

\item[(m)] $|\cod(G/N)| = 16$. By Lemma \ref{list}, we have $G/N \cong \A_9$ or $G/N \cong \J_2$. If $G/N \cong \A_9$ case, $2^4 \cdot 3^3 \cdot 5 \in \cod(\A_9)$ but is not in $\cod(\ON)$, while in the $\J_2$ case $2^6 \cdot 3^3 \cdot 5^2 \in \cod(\J_2)$ but $2^6 \cdot 3^3 \cdot 5^2 \not \in \cod(\ON)$. This contradicts $\cod(G/N) \subseteq \cod(G)$.

\item[(n)] $|\cod(G/N)| = 17$. Then $G/N \cong \M^c L$ or $\PSL(4,3)$. However,  $2^7 \cdot 5^3 \cdot 11$ is only in $\cod(\M^c L)$ and not $\cod(\ON)$, a contradiction. If $G/N = \PSL(4,3)$, the element $2^3 \cdot 3^6 \in \cod(\PSL(4,3))$ but it is not in $\cod(\ON)$.

\item[(o)] $|\cod(G/N)| = 18$. We have $G/N \cong \G_2(4), \Sy_4(5)$ or $\HS$.

If $G/N \cong \Sy_4(5)$, then $5^4 \cdot 13 \in \cod(G/N)$. If $G/N \in \G_2(4)$, then $2^{12} \cdot 3^3 \cdot 5 \cdot 7 \in \cod(G/N)$. If $G/N \cong \HS$, $2^8 \cdot 3^2 \cdot 11 \in \cod(G/N)$. None of these elements are in $\cod(\ON)$, a contradiction.

Having ruled out all other cases, we conclude $G/N \cong \ON$.
\end{proof}

\begin{lem}\label{redM24}
Let $G$ be a finite group with $\cod(G) = \cod(\M_{24})$. If $N$ is the maximal normal subgroup of $G$, it we have $G/N \cong \M_{24}$.
\end{lem}

\begin{proof}
As $G/N$ is a non-abelian simple group, $\cod(G/N) \subseteq \cod(G)$, so we have that $\cod(G/N)=4, 5, 6, 7, 8, 9, 10, 11, 12, 13, 14, 15, 16, 17, 18, 19$ or $G/N \cong \G_2(q)$.

\item [(a)] $|\cod(G/N)| = 4$. By Lemma \ref{list}, $G/N \cong \PSL(2,k)$ where $k = 2^f \geq 4$, with $k^2-1 \in \cod(G/N)$. By a complete search $k^2-1$ does not equal an element of $\cod(\M_{24})$ since none of the nontrivial codegrees are odd. This is a contradiction.

\item[(b)] $|\cod(G/N)| = 5$. By Lemma \ref{list}, $G/N \cong \PSL(2,k)$ where $k$ is a power of an odd prime number. We have $\displaystyle k(k-\epsilon(k)\in \cod(G/N)$, where $k = (-1)^{(k-1)/2}$. By a complete search, $k(k + 1)$ cannot be equal to any of the nontrivial elements in $\cod(\M_{24})$ for any integer $k$, a contradiction.

\item[(c)] $|\cod(G/N)| = 6$. By Lemma \ref{list}, $G/N \cong \PSL(3,4)$ or $G/N \cong {}^2B_2(q)$, with $q = 2^{2n+1}$.

If $G/N \cong \PSL(3,4)$ then $3^2 \cdot 5 \cdot 7 \in \cod(\PSL(3,4))$ but $3^2 \cdot 5 \cdot 7 \not\in \cod(\M_{24})$, contradicting $\cod(G/N) \subseteq \cod(G)$.

If $G/N \cong {}^2B_2(q)$, notice that $(q-1)(q^2+1) \in \cod({}^2B_2(q))$. By a complete search $(q-1)(q^2+1)$ does not equal an element of $\cod(\M_{24})$ for any integer $q$, a contradiction.

\item[(d)] $|\cod(G/N)| = 7$. By Lemma \ref{list}, we have $G/N \cong \PSL(3,3)$, $\M_{11}$, $\A_7$ or $\J_1$. The elements $2^2 \cdot 3^2 \cdot 13$, $2^3 \cdot 3^3 \cdot 11$, $2^2 \cdot 3 \cdot 5 \cdot 7$, $3 \cdot 5 \cdot 11 \cdot 19$ are in the codegrees of $\PSL(3,3)$, $\M_{11}$, $\A_7$ or $\J_1$, respectively. However, none of these are elements of $\cod(\M_{24}$, a contradiction.

\item[(e)] $|\cod(G/N)| = 8$. By Lemma \ref{list}, we have $G/N \cong \PSU(3,q)$ with $4 < q \not \equiv -1 \pmod{3}$, $G/N \cong \PSL(3,q)$ with $4 < q \not \equiv 1 \pmod{3}$ or $G/N \cong \G_2(2)'$.

If $G/N \cong \PSU(3,q)$, then $q^3(q^2-q+1) \in \cod(\PSU(3,q))$. By a complete search, $q^3(q^2-q+1)$ does not equal an element of $\cod(\M_{24})$ for any integer $q$, since none of these elements are odd in $\cod(\M_{24}$.

If $G/N \cong \PSL(3,q),$ we have $q^3(q^2+q+1) \in \cod(\PSL(3,q))$. By a complete search, $q^3(q^2+q+1)$ does not equal an element of $\cod(\M_{24})$ for any integer $q$, a contradiction.

Furthermore, $3^3 \cdot 7 \in\cod(\G_2(2)')$ but $3^3 \cdot 7 \not \in \cod(\M_{24})$, a contradiction.

\item[(f)] $|\cod(G/N)| = 9$. By Lemma \ref{list}, we have $G/N \cong \PSL(3,q)$ with $4 < q \equiv 1 \pmod{3}$ or $G/N \cong \PSU(3,q)$ with $4 < q \equiv -1 \pmod{3}$.

 By a complete search, these expressions do not equal any element of $\cod(\M_{24})$ for any nontrivial codegree is even in $\cod(G)$.

\item[(g)] $|\cod(G/N)| = 10$. By Lemma \ref{list}, we have $G/N \cong \M_{22}$. However, $2^7 \cdot 3 \cdot 5 \cdot 11$ is only present in $\cod(\M_{22})$ and not in $\cod(\M_{24})$, contradicting $\cod(G/N) \subseteq \cod(G)$.

\item[(h)] $|\cod(G/N)| = 11$. Lemma \ref{list}, we have $G/N \cong \PSL(4,2)$, $G/N \cong \M_{12}$, $G/N \cong \M_{23}$ or $G/N \cong {}^{2}\G_{2}(q)$ for $q = 3^{2f+1}, f \geq 1$. $\cod(\PSL(4,2)), \cod(\M_{12})$ and $\cod(\M_{23})$ contain $3^2 \cdot 5 \cdot 7, 2^2 \cdot 3 \cdot 5 \cdot 11$ and $2^7 \cdot 7 \cdot 11$, respectively, none of which are elements of $\cod(\M_{24})$.

Furthermore, $(q^2-1)(q^2-q+1) \in \cod({}^{2}\G_{2}(q))$. By a complete search, this element does not equal an element of $\cod(ON)$ for any integer root for $q$, contradicting $\cod(G/N) \subseteq \cod(G)$.

\item[(i)] $|\cod(G/N)| = 12$. By Lemma \ref{list}, we have $G/N \cong \Sy_4(4)$. However, $2^7 \cdot 5^2 \cdot 17$ is in of $\cod(\Sy_4(4))$ and not in $\cod(\ON)$, contradicting $\cod(G/N) \subseteq \cod(G)$.

\item[(j)] $|\cod(G/N)| = 13$. By Lemma \ref{list}, we have $G/N \cong \U_4(2)$ or $G/N \cong \Sy_4(q)$ for $q = 2^f > 4$.
If $G/N \cong \U_4(2)$, the element $3^4 \cdot 5 \in \cod(\U_4(2))$ but it is not in $\cod(\ON)$.

Furthermore, $q^4(q-1)^2 \in \cod(\Sy_4(q))$. However, by a complete search there is no integer $q$ such that $q^4(q-1)^2$ equals an element in $\cod(\M_{24})$, a contradiction.

\item[(k)] $|\cod(G/N)| = 14$. By Lemma \ref{list}, we have $G/N \cong \U_4(3)$, $G/N \cong {}^2\F_4(2)'$ or $G/N \cong \J_3$. However, the codegree sets of $\U_4(3)$, ${}^2\F_4(2)'$ and $\J_3$ have $2^4 \cdot 3^6$, $3^3 \cdot 5^2 \cdot 13$ and $2^7 \cdot 3^5 \cdot 19$, respectively, but none of these numbers are in $\cod(\M_{24})$, a contradiction.

\item[(l)] $|\cod(G/N)| = 15$. By Lemma \ref{list}, we have $G/N \cong \G_2(3)$. However, the element $2^5 \cdot 3^6 \cdot 13$ is only present in the codegree set of $\G_2(3)$ and not in the codegree set of $\M_{24}$, a contradiction.

\item[(m)] $|\cod(G/N)| = 16$. By Lemma \ref{list}, we have $G/N \cong \A_9$ or $G/N \cong \J_2$. If $G/N \cong \A_9$, $2^4 \cdot 3^3 \cdot 5 \in \cod(\A_9)$ but is not in $\cod(\M_{24})$, while in the $\J_2$ case $2^6 \cdot 3^3 \cdot 5^2 \in \cod(\J_2)$ but $2^6 \cdot 3^3 \cdot 5^2 \not \in \cod(\M_{24})$. This contradicts $\cod(G/N) \subseteq \cod(G)$.

\item[(n)] $|\cod(G/N)| = 17$. Then $G/N \cong \M^c L$ or $\PSL(4,3)$. However,  $2^7 \cdot 5^3 \cdot 11$ is only in $\cod(\M^c L)$ and not $\cod(\M_{24})$ contradiction. If $G/N = \PSL(4,3)$, the element $2^3 \cdot 3^6 \in \cod(\PSL(4,3))$ but it is not in $\cod(\M_{24})$.

\item[(o)] $|\cod(G/N)| = 18$. We have $G/N \cong \G_2(4)$, $\Sy_4(5)$ or $\HS$. If $G/N \cong \Sy_4(5)$, then $5^4 \cdot 13 \in \cod(G/N)$. If $G/N \in \G_2(4)$. then $2^{12} \cdot 3^3 \cdot 5 \cdot 7 \in \cod(G/N)$. If $G/N \cong \HS$, $2^8 \cdot 3^2 \cdot 11 \in \cod(G/N)$. None of these elements are $\cod(\M_{24})$, a contradiction.

\item[(p)] $|\cod(G/N)| = 19$. $G/N$ must be isomorphic to $\ON$. However, $2^4 \cdot 3^4 \cdot 7^2 \cdot 31 \in \cod(\ON)$ but it is not an element of $\cod(\M_{24})$, a contradiction.

\item[(q)] $|\cod(G/N)| = 20$. $G/N$ must be isomorphic to $\G_2(q)$. We may find from \cite{Tong-Viet} that $(q^6-1)(q^2-1)$ is a codegree of $\G_2(q)$. However, from a complete search, we may find that there does not exist an integer $q$ such that $(q^6-1)(q^2-1)$ is equal to an element of $\cod(\M_{24})$, a contradiction.

Having ruled out all other cases, we conclude $G/N \cong \M_{24}$.
\end{proof}

Now we give the proof of Theorem ~\ref{main}.

%\textbf{Proof.}
\begin{proof}
Suppose that the theorem is not true and let $G$ be a minimal counterexample.  Then $G$ is perfect by Remark \ref{perfectgroup}. Let $N$ be a maximal normal subgroup of $G$.
We have $G/N \cong H$ by Lemmas ~\ref{redS44}-\ref{redG23} and $N>1$. By the choice of $G$,  $N$ is a minimal normal subgroup of $G$. Otherwise there exists a nontrivial normal subgroup of $L$ of $G$ such that $L$ is included in $N$. Then  $\cod(G/L)=\cod(H)$ for  $\cod(H)=\cod(G/N)\subseteq \cod(G/L) \subseteq\cod(G)=\cod(H)$  and $G/L\cong H$ for $G$ is a minimal counterexample, a contradiction.\\

{\bf Step 1:} $N$ is the unique minimal normal subgroup of $G$.

Otherwise we assume $M$ is another minimal normal subgroup of $G$. Then $G=N\times M$ for $G/N$ is simple and $N\cong M\cong H$ for $M$ is also a maximal normal subgroup of $G$. If $H\cong \Sy_{4}(4)$, choose $\psi_1\in \irr(N)$ and $\psi_2\in \irr(M)$ such that $\cod(\psi_1)=\cod(\psi_2)=2^8\cdot 5^2$. Set $\chi=\psi_1\cdot\psi_2\in \irr(G)$. Then $\cod(\chi)=(2^8\cdot5^2)^2$, a contradiction. In other cases, we can also obtain a contradiction similarly.\\

{\bf Step 2:} $\chi$ is faithful for each $\chi \in \irr(G|N)$.

It can be easily checked that $N$ is not contained in the kernel of $\chi$ for each $\chi \in \irr(G|N)$. Then the kernel of $\chi$ is trivial by Step 1.\\

{\bf Step 3:} $N$ is elementary abelian.

Assume to the contrary that $N$ is not abelian. Thus $N=S^n$ where $S$ is a non-abelian simple group.
By Lemmas ~\ref{Mal2} and ~\ref{Liu2.5}, we see that there exists a non-principal character $\psi\in \irr(N)$ that extends to some $\chi\in \irr(G)$. Then $\ker(\chi)=1$ by Step 2 and $\cod(\chi)=|G|/\chi(1)=|N|/\psi(1)\cdot|G/N|$. This is a contradiction since $|G/N|$ is divisible by $\cod(\chi)$. \\

{\bf Step 4:} We show that $C_G(N)=N$.

First, ${\bf{C}}_G(N) \unlhd G$. Since $N$ is abelian by Step 3, there are two cases: either ${\bf{C}}_G(N)=G$ or ${\bf{C}}_G(N)=N$. If ${\bf{C}}_G(N)=N$, we are done.

Suppose ${\bf{C}}_G(N)=G$. Therefore $N$ must be in the center of $G$ and $|N|$ is a prime by Step 1. Since $G$ is perfect, we must have ${\bf Z}(G)=N$ and $N$ is isomorphic to a subgroup of the Schur multiplier of $G/N$ \cite[Corollary 11.20]{Isaacs}.

If $H$ is isomorphic to $\Sy_{4}(4)$, $\Sy_{4}(q)$, $q=2^f>4$ or ${}^2\F_4(2)'$, then by \cite{Conway} the Schur multiplier of $H$ is trivial which means $N=1$, a contradiction.

If $H$ is isomorphic to $\U_{4}(2)$, then $G$ is isomorphic to $2.\U_{4}(2)$ by  \cite{Conway}.
We note that $2.\U_{4}(2)$ has a character of degree
$16$ which means $3240 \in \cod(G)$, and this is a contradiction.

If $H$ is isomorphic to $\U_{4}(3)$, then $G$ is isomorphic to $2.\M_{22}$ or $3.\M_{22}$ by \cite{Conway}.
For $2.\M_{22}$ we have a character of degree $26,892$. This gives a codegree of $840 \in \cod(G)$, a contradiction. For $3.\M_{22}$ we have a character of degree $729$. So $2^3\cdot3^7\cdot5\cdot7 \in \cod(G)$, a contradiction.

If $H$ is isomorphic to $\J_{3}$, then $G$ is isomorphic to $3.\J_{3}$ by  \cite{Conway}.
 We note that $3.\J_{3}$ has a character of degree $26,163$ which means $1920 \in \cod(G)$, and this is a contradiction.
%procedure: character degrees of schur groups can be found by looking starting from bottom column number till right row number. th3n codegree is order/character degree

If $H$ is isomorphic to $\G_{2}(3)$, then $G$ is isomorphic to $3.\G_{2}(3)$ by \cite{Conway}.
We note that $3.\G_{2}(3)$ has a character of degree
$5832$ which means $2184 \in \cod(G)$, and this is a contradiction.

If $H$ is isomorphic to $\A_9$, then $G$ is isomorphic to $2.\A_9$ by \cite{Conway}. We note that $2.\A_9$ has a character degree of $8$ which means that $45360 \in \cod(G)$, contradiction.
If $H$ is isomorphic to $\J_2$, then $G$ is isomorphic to $2.\J_2$ by \cite{Conway}. We note that $2.\J_2$ has a character degree of $6$ which means that $201600 \in \cod(G)$, a contradiction.

If $H \cong \PSL(4,3)$, then $G$ is isomorphic to $2.\PSL(4,3)$ by \cite{Conway}. We note that $40$ is a character degree of $2.\PSL(4,3)$, so $303264 \in cod(G)$, a contradiction.

If $H \cong \McL$ then $G$ is isomorphic to $3.\McL$ by \cite{Conway}. We note that $126$ is a character degree of $3.\McL$, so $21384000 \in \cod(G)$, a contradiction.

If $H \cong \Sy_4(5)$, then $G$ is isomorphic to $2.\Sy_4(5)$ by \cite{Conway}. We note that $12$ is a character degree of $2.\Sy_4(5)$ so $780000 \in \cod(G)$, a contradiction.

If $H \cong \G_2(4)$, then $G$ is isomorphic to $2.\G_2(4)$ by \cite{Conway}. We note $2.\G_2(4)$ has a character degree of $12$ so $41932800 \in \cod(G)$, a contradiction.

If $H \cong \HS$, then $G$ is isomorphic to $2.\HS$ by \cite{Conway}. We note that $56$ is a character degree of $2.\HS$ which means that $1584000 \in \cod(G)$, a contradiction.

If $H \cong \ON$, then $G$ is isomorphic to $3.\ON$ by \cite{Conway}. We note that $342$ is a character degree of $3.\ON$ so $4042241280$ is in $\cod(G)$, a contradiction.

If $H \cong \M_{24}$, then the Schur multiplier of $H$ is trivial by \cite{Conway}, implying $N=1$, a contradiction.

If $H \cong \G_2(q)$, then the Schur multiplier of $H$ is trivial by \cite{Tong-Viet}, implying $N=1$, a contradiction.

Thus $C_G(N)=N$.\\

{\bf Step 5:} Let $\lambda$ be a non-principal character in $\irr(N)$ and $\theta \in \irr(I_G(\lambda)|\lambda)$. We show that
$\frac{|I_G(\lambda)|}{\theta(1)} \in \cod(G)$. Also, $\theta(1)$ divides $|I_G(\lambda)/N|$ and $|N|$ divides $|G/N|$. Especially, $I_G(\lambda)<G$, i.e. $\lambda$ is not $G$-invariant.

Let $\lambda$ be a non-principal character in $\irr(N)$. Given $\theta\in\irr(I_G(\lambda)|\lambda)$.
Note that $\chi=\theta^G\in\irr(G)$ and $\chi(1)=|G:I_G(\lambda)|\cdot\theta(1)$ by Clifford theory (see chapter 6 of \cite{Isaacs}). Then $\ker(\chi)=1$ by Step 2 and $\cod(\chi)=\frac{|I_G(\lambda)|}{\theta(1)}$. Especially, we have that $\theta(1)=\theta(1)/\lambda(1)$ divides $|I_G(\lambda)/N|$, and then $|N|$ divides
$\frac{|I_G(\lambda)|}{\theta(1)}$.
Since $\cod(G)=\cod(G/N)$ and $|G/N|$ is divisible by every element in $\cod(G/N)$, we have that $|N|\mid|G/N|$.

Next we show $I_G(\lambda)<G$. Otherwise we may assume $I_G(\lambda)=G$. Then $\mathrm{ker}(\lambda)\unlhd G$. Furthermore $\mathrm{ker}(\lambda)=1$ by Step 1 and $N$ is a cyclic subgroup with prime order by Step 3. Therefore $G/N$ is abelian for $G/N=N_G(N)/C_G(N)\leq\Aut(N)$ by  Normalizer-Centralizer Theorem, a contradiction.\\

{\bf Step 6:} Final contradiction.

By Step 3, $N$ is an elementary abelian $p$-subgroup for some prime $p$ and we assume $|N|=p^n$. By the Normalizer-Centralizer Theorem, $H\cong G/N=N_G(N)/C_G(N)\leq\Aut(N)$ Then we have that $n>1$. Note that in general $\Aut(N)=\GL(n,p)$. By Step 5, $|N|\mid |G/N|$ so we need only to look at the prime power divisors of $|H|$. \\

(Case 1)  $H\cong \Sy_4(4)$

Now $|H|=2^8\cdot3^2\cdot5^2\cdot17$. It follows easily that $|H|$ does not divide $|\GL(n,2)|$ when $2\leq n\leq 7$, $|\GL(2,3)|$, or $|\GL(2,5)|$.

 Thus $|N|=2^8$. By Step 5, we know that $\frac{|T|}{\theta(1)}\in\cod(G)$ and we show that $|N|$ divides $\frac{|T|}{\theta(1)}$.
Upon inspection of the codegrees of $G$, we have that $\frac{|T|}{\theta(1)}$ must equal to $2^8\cdot 3\cdot 5^2$, $2^8\cdot 3^2\cdot 5$, $2^8\cdot 5^2$, $2^8\cdot 3\cdot 5$, or $2^8\cdot 17$. Thus  $\frac{|T/N|}{\theta(1)}=3\cdot 5^2$, $3^2\cdot 5$, $5^2$, $3\cdot 5$, or $17$. From this, we conclude that $|T/N|_2=|\theta(1)|_2$. By chapter 6 of \cite{Isaacs}, we note that $|T/N|$ is a sum of squares which are of the same form as $\theta(1)^2$. Thus, we know that $|T/N|_2=1$. But then, $|G/N:T/N|=|G:T|\geq 2^8$, a contradiction.\\

 %If $|N|=9$, then $|\Aut(N)|=|\GL(2,3)|=48$, a contradiction with $|H|\mid |\Aut(N)|$. Since  $|\GL(2,2)|=6$, $|\GL(3,2)|=168$ and $|\GL(4,2)|=20160$, we have a contradiction for  $|N|=2^2, 2^3$ or $2^4$.\\

(Case 2) $H\cong \U_4(2)$.
 Now $|H|=2^6\cdot3^4\cdot5$. If $|\GL(n,2)|$ when $2\leq n\leq5$ or $|\GL(m,3)|$ when $2\leq m\leq3$, it can be checked that $|H|$ doesn't divide the order $|\Aut(N)|$.

 If $|\GL(6,2)|$ and $|\GL(4,3)|$, we have that $|H|$ divides the order $|\Aut(N)|$.

The first circumstance is that $|N|=2^6$. %If we look at the subgroups of $U_4(2)$ (see \cite{Conway}), we can deduce that $27\nmid|T/N|$.
By Step 5, we know that $\frac{|T|}{\theta(1)}\in\cod(G)$ and we show that $|N|$ divides $\frac{|T|}{\theta(1)}$.
Upon inspection of the codegrees of $G$, we have that $\frac{|T|}{\theta(1)}$ must equal to $2^6\cdot3^4$, $2^6\cdot3^3$, $2^6\cdot3^2$ or $2^6\cdot5$. Thus  $\frac{|T/N|}{\theta(1)}=3^4$, $3^3$, $3^2$ or $5$. From this, we conclude that $|T/N|_2=|\theta(1)|_2$. In any case, whenever we have the $p$-parts of $|T/N|$ and $\theta(1)$ equal to each other, we know that they must both be $1$. This is because $|T/N|$ is a sum of squares of the same form as $\theta(1)^2$ by chapter 6 of \cite{Isaacs}. Hence, we have $|T/N|_2=1$. But then, $|G/N:T/N|=|G:T|\geq 2^6$, a contradiction.

Another situation is that $|N|=3^4$. By Step 5, we know that $\frac{|T|}{\theta(1)}\in\cod(G)$ and we show that $|N|$ divides $\frac{|T|}{\theta(1)}$.
Upon inspection of the codegrees of $G$, we have that $\frac{|T|}{\theta(1)}$ must equal to $2^6\cdot3^4$, $2^5\cdot3^4$, $2^4\cdot3^4$, $2^3\cdot3^4$ or $3^4\cdot5$. Then $\frac{|T/N|}{\theta(1)}=2^6$, $2^5$, $2^4$, $2^3$ or $5$. From this, we conclude that $|T/N|_3=|\theta(1)|_3$. By chapter 6 of \cite{Isaacs}, we note that $|T/N|$ is a sum of squares which are of the same form as $\theta(1)^2$. Thus, we know that $|T/N|_3=1$. But then, $|G/N:T/N|=|G:T|\geq 3^4$, a contradiction.\\

(Case 3) $H\cong \Sy_{4}(q)$, for $q =2^f$, $q>4$.
%By Step 3, $N$ is an elementary abelian $r$-subgroup for some prime $r$ and we assume $|N|=r^n$, $n\in\N$. By the Normalizer-Centralizer Theorem, we see that $n>1$.

%Given a non-principle character $\lambda\in \irr(N)$. Let $T:=I_G(\lambda)$. By Step 5, we have that $\frac{|T|}{\theta(1)}\in \cod(G)$ for all $\theta\in \irr(T|\lambda)$.

%Since $N$ is abelian by Step 1, $|\irr(N)|=|N|$. Therefore, $|N|=|\irr(N)| > |G:T|$ since $|G:T|$ is the number of conjugates of $\lambda$ in $G$ which are all contained in $\irr(N)$.

Note that $\gcd(q^2+1,q+1) =1$ and $\gcd(q^2+1,q-1)=1$
and $\gcd(q+1, q-1)=1$ or $2$. Thus, $q^4$ is the largest power of a prime that divides the order of $\Sy_{4}(q)$, where $q=2^f>4$. Then, $|N|\le q^4$. Let $K$ be a maximal subgroup of $\Sy_4(q)$ such that $T/N\leq K$. Then is a maximal subgroup of $G/N$.
If $K$ is not the type in Lemma \ref{maxS4q}, then $|G:T|\ge q^4$, a contradiction.
Therefore $K$ is of type in Lemma \ref{maxS4q}.
If $K$ is of type ${C_{q-1}}^{2}:D_{8}$, then $|G:T|\ge\frac {1}{8}q^4(q^2+1)(q+1)^2$, and thus $|G:T|>q^4$, a contradiction.
If $K$ is of type ${C_{q+1}}^{2}:D_{8}$, then $|G:T|\ge\frac{1}{8}q^4(q^2+1)(q-1)^2$, and thus $|G:T|>q^4$, a contradiction.
If $K$ is of type ${C_{q^2+1}}:4$, then $|G:T|\ge\frac{1}{4}q^4(q+1)^2(q-1)^2$, and thus $|G:T|>q^4$, a contradiction.
If $K$ is of type $SO_{4}^{+}(q)$, then $|G:T|\ge2q^2(q^2+1)(q+1)(q-1)$, and thus $|G:T|>q^4$, a contradiction.
If $K$ is of type $S_{z}(q)$, then $|G:T|\ge q^2(q+1)^2(q-1)$, and thus $|G:T|>q^4$, a contradiction.
When the maximal subgroup is one of these cases, this means that it is a contradiction for $|G/N|>q^4$.\\

(Case 4) $H\cong \U_4(3)$.
 Now $|H|=2^7\cdot3^6\cdot5\cdot7$.
 %If $|N|=9$, then $|\Aut(N)|=|\GL(2,3)|=48$, a contradiction.
We note that $|\U_4(3)|$ does not divide $|\GL(n,2)|$ when $2\leq n\leq7$ or $|\GL(m,3)|$ when $2\leq m \leq 5$.

Thus $|N|=3^6$. By Step 5, we know that $\frac{|T|}{\theta(1)}\in\cod(G)$ and we show that $|N|$ divides $\frac{|T|}{\theta(1)}$.
Upon inspection of the codegrees of $G$, we have that $\frac{|T|}{\theta(1)}$ must equal to $2^7\cdot3^6$, $2^5\cdot3^6$, $2^4\cdot3^6$, $2^3\cdot3^6$, $3^6\cdot7$ or $3^6\cdot5$. Thus $\frac{|T/N|}{\theta(1)}=2^7$, $2^5$, $2^4$, $2^3$, $7$, or $5$ and we conclude that $|T/N|_3=|\theta(1)|_3$. By chapter 6 of \cite{Isaacs}, we note that $|T/N|$ is a sum of squares which are of the same form as $\theta(1)^2$. Thus, we know that $|T/N|_3=1$. But then, $|G/N:T/N|=|G:T|\geq 3^6$, a contradiction.\\

(Case 5) $H\cong{_{}^{2}\F_{2}(2)}'$.
Note that $|H|=2^{11}\cdot3^3\cdot5^2\cdot13$. We can also simply check that $|H|$ does not divide the order of any possible general linear group with these conditions.\\

(Case 6) $H\cong \J_{3}$.
Note that $|H|=2^7\cdot3^5\cdot5\cdot17\cdot19$. Similarly, we can check that $|H|$ does not divide $|\GL(n,2)|$ when $2\leq n\leq7$ or $|\GL(m,3)|$ when $2\leq n\leq5$.\\

(Case 7) $H\cong \G_{2}(3)$.
Now $|H|=2^6\cdot3^6\cdot7\cdot13$. We can also simply check that $|H|$ does not divide $|\GL(n,2)|$ when $2\leq n\leq6$ or $|\GL(m,3)|$ when $2\leq m\leq5$.

Thus $|N|=3^6$. By Step 5, we know that $\frac{|T|}{\theta(1)}\in\cod(G)$ and we show that $|N|$ divides $\frac{|T|}{\theta(1)}$.
Upon inspection of the codegrees of $G$, we have that $\frac{|T|}{\theta(1)}$ must equal to $2^5\cdot13\cdot3^6$, $13\cdot7\cdot3^6$, $2^6\cdot3^6$, $2^3\cdot7\cdot3^6$, $3^6\cdot2^5$, $3^6\cdot13$, $2^3\cdot3^6$, or $ 3^6\cdot7$. Thus  $\frac{|T/N|}{\theta(1)}=2^5\cdot13$, $13\cdot7$, $2^6$, $2^3\cdot7$, $2^5$, $2^3$, $ 13$ or $7$ and we conclude that $|T/N|_3=|\theta(1)|_3$. By chapter 6 of \cite{Isaacs}, we note that $|T/N|$ is a sum of squares which are of the same form as $\theta(1)^2$. Thus, we know that $|T/N|_3=1$. But then, $|G/N:T/N|=|G:T|\geq 3^6$, a contradiction.\\

(Case 8) $H \cong \A_9$.
Now $|H| = 2^6\cdot 3^4\cdot 5\cdot 7$. By inspection, $|H|$ only divides $|\GL(n,2)|$ when $n=6$.

%Let $T := I_G(\lambda).$ Then $1 < |G:T| < |N| = 64$. %Examining \cite{Conway}, we may find that $9 \nmid |T/N|.$
By Step $5$, we know that $\frac{|T|}{\theta(1)} \in \cod(G)$ and that $|N| \mid \frac{|T|}{\theta(1)}$, so $\frac{|T|}{\theta(1)} = 2^6\cdot3\cdot5$, $2^6\cdot3^3$, $2^6\cdot3^4$, $2^6\cdot3\cdot5\cdot7$, or $2^6\cdot3^3\cdot5$. Then $\frac{|T/N|}{\theta(1)}=3\cdot5$, $3^3$, $3^4$, $3\cdot5\cdot7$, or $3^3\cdot5$. By chapter $6$ of \cite{Isaacs}, we note that $|T/N|$ is a sum of squares of numbers with the same form as $\theta(1)^2.$ This implies $|T/N|_2 = 1$. Thus $|G/N:T/N| = |G:T| \geq 2^6 > |G:T|$, a contradiction.\\

(Case 9) $H \cong \J_2$.
Now $|H| = 2^7\cdot 3^3\cdot 5^2\cdot 7$. By inspection, $|H|$ does not divide $|\GL(n,2)|$ with $2 \leq n \leq 7$, $|\GL(m,3)|$ with $2 \leq m \leq 3$, or $|\GL(2,5)|$. \\

(Case 10) $H \cong \PSL(4,3)$.
Now $|H| = 2^7\cdot 3^6\cdot 5\cdot 13$. By inspection, $|H|$ does not divide $|\GL(n,2)|$ with $2 \leq n \leq 7$ or $|\GL(m,3)|$ with $2 \leq m \leq 3$. Thus $|N|=3^4, 3^5$, or $3^6$.

Suppose $|N|=3^{4}$.
Let $K$ be a maximal subgroup of $\mathrm{PSL}(4,3)$ such that $T/N\leq K$. Then $K$ is a maximal subgroup of $G/N$.
If $K$ is not of the type $3^3:\mathrm{L}_3(3)$, then $|G:T|>3^4$ by Lemma \ref{psl43}, a contradiction.
By Step $5$, we can again deduce that $\frac{|T|}{\theta(1)}= 2^6\cdot3^6\cdot 5, 2^7\cdot 3^5\cdot 5, 2^5\cdot3^6\cdot 5, 2^7\cdot3^6, 2^6\cdot3^4\cdot 13, 2^6\cdot3^4\cdot5, 2^5\cdot3^6, 2^6\cdot3^5, 2^2\cdot3^6\cdot 5, 2^5\cdot3^4\cdot 5, 2^7\cdot3^4, 3^6\cdot 13, 2^5\cdot3^5$, or $2^3\cdot 3^6$. Thus, $\frac{|T/N|}{\theta(1)}  =2^6\cdot3^2\cdot 5, 2^7\cdot 3\cdot 5, 2^5\cdot3^2\cdot 5, 2^7\cdot3^2, 2^6\cdot 13, 2^6\cdot5, 2^5\cdot3^2, 2^6\cdot3, 2^2\cdot3^2\cdot 5, 2^5\cdot 5, 2^7, 3^2\cdot 13, 2^5\cdot3$, or $2^3\cdot 3^2$. We conclude that $\frac{|T/N|_3}{\theta(1)_3}\leq 3^2$, i.e. $\frac{|T/N|_3}{3^2}\leq \theta(1)_3$. By chapter $6$ of \cite{Isaacs}, $|T/N|$ is a sum of squares which are of the same form as $\theta(1)^2$, we know that $|T/N|_3\leq 3^4$. If $K\cong 3^3:\mathrm{L}_3(3)$, then $|T/N|\leq 3^4\cdot 2^4\cdot 13$, and so $|G/N:T/N| = |G:T| \geq 3^{2}\cdot2^3\cdot 5 >3^4 > |G:T|$, a contradiction. If $K\cong 3^3:\mathrm{L}_3(3)$, then $|T/N|\leq 3^4\cdot 2^4\cdot 13$, and so $|G/N:T/N| = |G:T| \geq 3^{2}\cdot2^3\cdot 5 >3^4 > |G:T|$, a contradiction.

Suppose $|N|=3^{5}$.
Let $K$ be a maximal subgroup of $\mathrm{PSL}(4,3)$ such that $T/N\leq K$. Then $K$ is a maximal subgroup of $G/N$.
If $K\cong (4\times\mathrm{A}_6):2$, $\mathrm{S}_6$, or $\mathrm{S}_4\times\mathrm{S}_4$, then $|G:T|>3^4$ by Lemma \ref{psl43}, a contradiction.
By Step $5$, we can again deduce that $\frac{|T|}{\theta(1)}= 2^6\cdot3^6\cdot 5$, $2^7\cdot 3^5\cdot 5$, $2^5\cdot3^6\cdot 5$, $2^7\cdot3^6$, $2^5\cdot3^6$, $2^6\cdot3^5$, $2^2\cdot3^6\cdot 5$, $3^6\cdot 13$, $2^5\cdot3^5$, or $2^3\cdot 3^6$. Thus, $\frac{|T/N|}{\theta(1)} =2^6\cdot3\cdot 5$, $2^7\cdot 5$, $2^5\cdot3\cdot 5$, $2^7\cdot3$, $2^5\cdot3$, $2^6$, $2^2\cdot3\cdot 5$, $3\cdot 13$, $2^5$, or $2^3\cdot 3$. We conclude that $\frac{|T/N|_3}{\theta(1)_3}\leq 3$, i.e. $\frac{|T/N|_3}{3}\leq \theta(1)_3$. By chapter $6$ of \cite{Isaacs}, $|T/N|$ is a sum of squares which are of the same form as $\theta(1)^2$, we know that $|T/N|_3\leq 3^2$. If $K\cong 3^3:\mathrm{L}_3(3)$, then $|T/N|\leq 3^2\cdot 2^4\cdot 13$, and so $|G/N:T/N| = |G:T| \geq 3^{4}\cdot2^3\cdot 5 >3^5 > |G:T|$, a contradiction. If $K\cong \mathrm{U}_4(2)$, then $|T/N|\leq 3^2\cdot 2^7\cdot 5$, and so $|G/N:T/N| = |G:T| \geq 3^{4}\cdot 13 >3^5 > |G:T|$, a contradiction. If $K\cong 3^4:2(\mathrm{A}_4\times \mathrm{A}_4).2$, then $|T/N|\leq 3^2\cdot 2^6$, and so $|G/N:T/N| = |G:T| \geq 3^{4}\cdot 2\cdot 5\cdot 13 >3^5 > |G:T|$, a contradiction.

Suppose $|N|=3^{6}$.
%By \cite{Conway}, we have $40 \nmid |T/N|.$
By Step $5$, we can again deduce that $\frac{|T|}{\theta(1)}= 2^6\cdot3^6\cdot 5, 2^5\cdot3^6\cdot 5, 2^7\cdot3^6,  2^5\cdot3^6, 2^2\cdot3^6\cdot 5, 3^6\cdot 13$, or $2^3\cdot 3^6$. Thus, $\frac{|T/N|}{\theta(1)}  = 2^6\cdot 5, 2^5\cdot 5, 2^7, 2^5, 2^2\cdot 5, 13$, or $2^3$. We conclude that $|T/N|_3=|\theta(1)|_3$. By chapter $6$ of \cite{Isaacs}, $|T/N|$ is a sum of squares which are of the same form as $\theta(1)^2$, we know that $|T/N|_3=1$ and $|G/N:T/N| = |G:T| \geq 3^{6} > |G:T|$, a contradiction.\\

(Case 11) $H \cong \McL$.
Now $|H| = 2^7\cdot 3^6\cdot 5^3\cdot 7\cdot 11$. By inspection, $|H|$ does not divide $|\GL(n,2)|$ with $2 \leq n \leq 7$, $|\GL(m,3)|$ with $2 \leq m \leq 6$, or $|\GL(l,5)|$ with $2 \leq l \leq 3$. \\

(Case 12) $H \cong \Sy_4(5)$.
Now $|H| = 2^6\cdot 3^2\cdot 5^4\cdot 13$. By inspection, $|H|$ does not divide $|\GL(n,2)|$ with $2 \leq n \leq 6$, $|\GL(2,3)|$, or $|\GL(l,5)|$ with $2 \leq l \leq 3$.

Thus $|N|=5^{4}$. By Step $5,$ we know that $\frac{|T|}{\theta(1)} \in \cod(G)$ and that $|N| \mid \frac{|T|}{\theta(1)}$, so $\frac{|T|}{\theta(1)} = 2^2\cdot3\cdot5^4, 5^4\cdot13, 2^3\cdot3\cdot5^4, 2^2\cdot3^2\cdot5^4, 2^4\cdot3\cdot5^4, 2^3\cdot3^2\cdot5^4, 2^5\cdot3\cdot5^4$, or $2^6\cdot3^2\cdot5^4.$ This implies $\frac{|T/N|}{\theta(1)} = 2^2\cdot3, 13, 2^3\cdot3, 2^2\cdot3^2, 2^4\cdot3, 2^3\cdot3^2, 2^5\cdot3$, or $2^6\cdot3^2$. We conclude that $|T/N|_5=|\theta(1)|_5$. By chapter $6$ of \cite{Isaacs}, $|T/N|$ is a sum of squares which are of the same form as $\theta(1)^2$, we know that $|T/N|_5=1$ and $|G/N:T/N| = |G:T| \geq 5^{4} > |G:T|$, a contradiction.\\

(Case 13) $H \cong \G_2(4)$.
Now $|H| = 2^{12}\cdot 3^3\cdot 5^2\cdot 7\cdot 13$. By inspection, $|H|$ does not divide $|\GL(n,2)|$ with $2 \leq n \leq 11$, $|\GL(m,3)|$ with $m=2,3$, or $|\GL(2,5)|$.

Thus $|N|=2^{12}$. By Step $5$, we know that $\frac{|T|}{\theta(1)} \in \cod(G)$ and that $|N| \mid \frac{|T|}{\theta(1)}$, so $\frac{|T|}{\theta(1)} = 2^{12}\cdot13$, $2^{12}\cdot3\cdot5$, $2^{12}\cdot3\cdot7$, $2^{12}\cdot3^2\cdot5$, $2^{12}\cdot3\cdot5^2$, or $2^{12}\cdot3^3\cdot5\cdot7$. This implies $|T/N| = 13$, $3\cdot5$, $3\cdot7$, $3^2\cdot5$, $3\cdot5^2$, or $3^3\cdot5\cdot7$. We conclude that $|T/N|_2=|\theta(1)|_2$. By chapter $6$ of \cite{Isaacs}, $|T/N|$ is a sum of squares which are of the same form as $\theta(1)^2$, we know that $|T/N|_2=1$ and $|G/N:T/N| = |G:T| \geq 2^{12} > |G:T|$, a contradiction.\\

(Case 14) $H \cong \HS$.
Now $|H| = 2^9\cdot 3^2\cdot 5^3\cdot 7\cdot 11$. By inspection, $|H|$ does not divide $|\GL(n,2)|$ with $2 \leq n \leq 9$, $|\GL(m,3)|$ with $m=2$, or $|\GL(l,5)|$ with $2 \leq l \leq 3$. \\

(Case 15) $H \cong \ON$.
Now $|H| = 2^9\cdot 3^4\cdot 5\cdot 7^3\cdot 11\cdot 19\cdot 31$. By inspection, $|H|$ does not divide $|\GL(n,2)|$ with $2 \leq n \leq 9$, $|\GL(m,3)|$ with $2 \leq m \leq 4$, or $|\GL(l,7)|$ with $2 \leq l \leq 3$. \\

(Case 16) $H \cong \M_{24}$.
Now $|H| = 2^{10} \cdot 3^{3} \cdot 5 \cdot 7 \cdot 11 \cdot 23$. By inspection, $|H|$ does not divide $|\GL(n,2)|$ with $2 \leq n \leq 10$, or $|\GL(m,3)|$ with $2 \leq m \leq 3$.
\end{proof}
%The only possible case is that $|N|=3^3$. Let $T:=I_G(\lambda)$. Then $1<|G:T| < |N| = 27$ for $|G:T|$ is the number of  all conjugates of $\lambda$. If we look at the subgroups of $\PSL(3,3)$ (see \cite{Conway}), we can deduce that $13 \nmid |T/N|$. By Step 5, we know that $\frac{|T|}{\theta(1)} \in \cod(G)$ and we show that $|N|$ divides $\frac{|T|}{\theta(1)}$.
%Upon inspection of the codegrees of $G$, we have that $\frac{|T|}{\theta(1)}$ must equal to $2^4 \cdot 3^3$ or $2^3 \cdot 3^3$. Since  $\frac{|T/N|}{\theta(1)} = 2^4$ or $2^3$, the $3$-parts and the $13$-parts of $|T/N|$ and $\theta(1)$ are equal. By chapter 6 of \cite{Isaacs}, we note that $|T/N|$ is a sum of squares which are of the same form as $\theta(1)^2$.  Thus, the $3$-part and $13$-part of $|T/N|$ are both $1$. Then, $T/N$ is a $2$-group and $|G/N : T/N|=|G:T|\geq 3^3\cdot 13 > 27$ which is a contradiction.\\

%\section{Acknowledgements}

\end{document}